%% file: Rep_Heis_.tex
\documentclass[a4paper, oneside, 11pt]{amsart}
\usepackage{import}

\title[The category of representations of the Heisenberg conformal net]{The crossed braided tensor category of twisted and untwisted representations of the Heisenberg conformal net}
\author{Adrià Marín-Salvador}
\date{}

\import{.}{Commands.tex}
\import{.}{Preamble.tex}

\begin{document}

\begin{abstract}
We show that the category of twisted/untwisted representations of the Heisenberg conformal net $\Heis$ is a continuous Tambara-Yamagami category for the group $\R$. We compute the associators and the $\mathbb{Z}/2$-crossed braiding. By taking a $\Z/2$-equivariantization, we obtain an explicit computation of the braided continuous tensor category of representations of the fixed-points conformal net $\Heis^{\Z/2}$. This provides the first explicit computation of a category of representations of a conformal net containing irreducible representations whose tensor product is a direct integral of irreducible representations. 
\end{abstract}

\maketitle

\tableofcontents
\newpage

\section{Introduction}
\addtocontents{toc}{\protect\setcounter{tocdepth}{-1}}

\subsection*{Conformal Field Theory} Conformal Field Theory (CFT), the study of Quantum Field Theories which are invariant under conformal transformations, has been described both as complex analysis in the quantum domain \cite{Polyakov_2012} and as a generalization of group theory \cite{MR1002038}. There exist two main formalizations of two dimensional unitary chiral CFT: conformal nets and unitary Vertex Operator Algebras (VOAs). On the one hand, conformal nets encode the vacuum Hilbert space of states of the theory, together with the local algebras of bounded operators associated to open regions of spacetime \cite{bmt88, bsm90, bgl93, gf93, w95}. On the other hand, VOAs axiomatize the algebras of unbounded local operators supported on points, together with their operator product expansions \cite{MR843307,MR996026,MR3119224}. 

The theories of conformal nets and VOAs are conjectured to be equivalent. In \cite{MR3796433} (see also \cite{MR4483020}) the authors construct, for any unitary VOA $V$ satisfying certain technical conditions, an associated conformal net $\A_V$. Conversely, given a conformal net $\A$, \cite{henriques2025conformalnetassociatedunitary} constructs an associated unitary VOA $V_\A$. These two constructions are conjectured to produce a bijection between the classes of conformal nets and unitary VOAs.

Given a conformal net $\A$, its category of representations $\Rep(\A)$ is a balanced tensor category \cite{frs, lon89, gf93, Gui21, balanceRepA}, and the conformal net is called \emph{rational} if $\Rep(\A)$ is a rigid, finite, semisimple category. In this situation, $\Rep(\A)$ is a modular tensor category \cite{klm01}. Similarly, a VOA $V$ is called \emph{rational} if its category of representations $\Rep(V)$ is a finite semisimple category, and it is furthermore \emph{strongly rational} if it is in addition self-dual, of CFT type, and $C_2$-cofinite. The category of representations of a strongly rational VOA is a modular tensor category \cite{MR1225125, MR1327541, MR1383584, MR1344848,  MR1344849, MR2151865, MR2468370}. 
The breakthrough papers \cite{Gui21, Gui_2026} prove that, for a large class of strongly rational unitary VOAs, their modular tensor categories of representations are equivalent to the underlying semisimple tensor categories of representations of their associated conformal nets through the construction in~\cite{MR3796433}. Whilst the formal properties of braided tensor categories of representations are easier to establish for conformal nets, the fusion rules of representations of particular examples are much more accessible using VOA techniques. 

Away from rationality (which in terms of the associated VOAs means away from $C_2$-cofiniteness), categories of representations of conformal nets and their associated unitary VOAs are expected to behave quite differently. Indeed, when $V$ is a general unitary VOA, the category $\Rep(V)$ is not even expected to admit a tensor structure. On the contrary, the category of representations of a conformal net is always a tensor category.

We propose the following analogy. Given a real Lie group $G$, one can consider the category $\Rep(G)$ of unitary representations of $G$, which is a tensor category. Denoting by $K$ the maximal compact subgroup of $G$, and by $\mathfrak{g}$ the Lie algebra of $G$, there also exists the category $(\mathfrak{g}, K)\text{-Mod}$ of Harish-Chandra modules over $G$. The category $(\mathfrak{g}, K)\text{-Mod}$ is not in general a tensor category. If $G$ is compact, then $(\mathfrak{g}, K)\text{-Mod}=\Rep(G)$ is a tensor category, but for non-compact Lie groups there exist Harish-Chandra modules which are not unitarizable, and unitary representations of $G$ which do not differentiate to Harish-Chandra modules. We think of representations of conformal nets as behaving as unitary representations of Lie groups, representations of VOAs as resembling Harish-Chandra modules, and strong rationality in CFT as mimicking compactness~of~Lie~groups.

\subsection*{Continuous tensor categories} Categories of unitary representations of non-compact Lie groups have two remarkable features: their space of (simple) objects has a natural topology, the Fell topology, and one can take direct integrals of unitary representations. Categories of representations of conformal nets also admit a Fell topology and allow for direct integrals of objects. However, asking merely for such a topology on the space of objects is too weak of a requirement, as one would also want to make precise the idea that fusion rules, associators and unitors are continuous. Tensor categories with a topological space of simple objects and direct integrals are a shadow of a more general structure, that of a continuous tensor category, which also includes the continuity of the extra tensor data. 
\begin{definition*}(\cite{AM2025})
    A \emph{continuous tensor category} is a (coherently associative) algebra object in the bicategory of $\mathrm{C}^*$-algebras, $\mathrm{C}^*$-correspondences, and unitary intertwiners.
\end{definition*}
We think of a $\mathrm{C}^*$-algebra $A$ as topologizing its underlying linear category of representations $\Rep(A)$, by endowing its space of objects with the natural Fell topology and providing the notion of direct integrals of objects. These two structures are no longer well defined if one only has access to $\Rep(A)$, as opposed to $A$. Given another $\mathrm{C}^*$-algebra $B$, we think of continuous functors between $\Rep(A)$ and $\Rep(B)$ as those given by tensoring with a $\mathrm{C}^*$-correspondence. We can further define braided and balanced continuous tensor categories in the obvious way, and see them as topologized versions of their underlying linear braided or balanced tensor categories. Given a locally compact group $G$, the category of unitary representations of $G$ is canonically a continuous tensor category, see \cite[Ex. 3.15]{AM2025}. We conjecture that categories of representations of conformal nets are also continuous tensor~categories.
\begin{mainconjecture}\label{conjectureintro}
    The category $\Rep(\A)$ of representations of a conformal net $\A$ is a balanced continuous tensor category. 
\end{mainconjecture}
Given a conformal net $\A$, one can construct a $\mathrm{C}^*$-algebra $\mathrm{C}^*(\A)$ such that representations of $\A$ are in bijection with a certain class of representations of $\mathrm{C}^*(\A)$ \cite{MR1147469, MR1199171, klm01}. We note that this does not (even partially) solve our conjecture, as we require representations of $\A$ to be in bijection with all representations of the associated $\mathrm{C}^*$-algebra, not only a subclass.

If $\A$ is rational, then $\Rep(\A)$ is a unitary modular tensor category by \cite{klm01}, and hence our conjecture holds.
\begin{maintheorem}
    The category $\Rep(\Heis)$ of representations of the Heisenberg conformal net is a balanced continuous tensor category. The category $\Rep(\Heis^{\Z/2})$ of representations of the $\Z/2$-fixed points of the Heisenberg conformal net is a $\Z/2$-equivariantization of a  $\Z/2$-crossed balanced continuous tensor category.
\end{maintheorem}

The precise description of these continuous tensor categories is the content of Theorems \ref{thm: CharacterizationRepHeis} and \ref{thm: RepFixedPointsMainTheorem}. An $\R$-family of representations of the Heisenberg conformal net was constructed in \cite{bmt88}, and their fusion rules were shown to agree with addition on $\R$. Further evidence towards Conjecture \ref{conjectureintro} is provided by Virasoro conformal nets at central charge $c\geq 1$. It has been shown that irreducible representations of these conformal nets are indexed by positive real numbers \cite{MR2031030, MR3627412}. The tensor categories of representations of Virasoro conformal nets have been conjectured to be equivalent to those of the modular doubles of certain quantum groups of $\mathfrak{sl}(2,\R)$ \cite{MR1877816}. The associators for a full tensor subcategory of representations of these modular doubles have appeared conjecturally in \cite{MR3197004,liu2025turaevviroinvariantmodulardouble}, exhibiting continuous~features.

\subsection*{Twisted modules} In order to study representations of $\Heis^{\Z/2}$, we need the notion of twisted representations of $\Heis$. Let $\A$ be a conformal net being acted on by a discrete group $G$. Every element $g\in G$ defines a category $\Rep^g(\A)$ of $g$-twisted representations of $\A$ (see Definition \ref{def: TwistedRep}), so that $\Rep^e(\A)~=~\Rep(\A)$, where $e\in G$ is the unit. The following is a generalization of \cite{muger05}, which produces a $G$-crossed braided structure on $\Rep^G(\A)$.
\begin{theorem*}(\cite{GcrossedbraidedRep})
   Let $G$ be a discrete group acting faithfully on a conformal net $\A$. Then, the category $\Rep^G(\A):=\bigoplus_{g\in G}\Rep^g(\A)$ is canonically a $G$-crossed balanced tensor~category.
\end{theorem*}
Therefore, the $G$-equivariantization $(\Rep^G(\A))^G$ of $\Rep^G(\A)$ is canonically a balanced tensor category.
\begin{theorem*}(\cite{fixedpoints})
 Let $G$ be a finite group acting faithfully on a conformal net $\A$. Then, there is an equivalence of balanced tensor categories 
 \begin{equation}\label{eq: FinxedPointsIntro}
 \Rep(\A^G)\cong (\Rep^G(\A))^G.  
 \end{equation}
\end{theorem*}
The equivalence between the underlying braided tensor categories of \eqref{eq: FinxedPointsIntro} was proved, for $\A$ rational, in \cite{muger05}. In the current paper, we prove that the category $\Rep^{\Z/2}(\Heis)$ of twisted and untwisted representations of $\Heis$ is a continuous Tambara-Yamagami tensor category for the group $\R$, defined as follows. 

\subsection*{Continuous Tambara-Yamagami categories} Given a finite group $G$, the category $\text{Vec}\, G$ of $G$-graded, finite-dimensional vector spaces admits a canonical tensor structure induced by convolution on $G$. Taking $G$ to be abelian, Tambara and Yamagami produced in \cite{TY} tensor categories containing $\text{Vec}\,G$ plus an extra simple object $\tau$ satisfying $\C_g\otimes\tau = \tau\otimes \C_g=\tau$ and  $\tau\otimes\tau = \oplus_{g\in G}\mathbb{C}_g$, where $\mathbb{C}_g$ denotes the one-dimensional vector space in degree $g\in G$. Tensor categories with these fusion rules were shown to be classified by a symmetric nondegenerate bicharacter on $G$ and a sign $\xi\in\{\pm 1\}$. The bicharacter $\chi$ (which is only considered up to the action of $\Aut(G)$) controls the $\text{Vec}\,G-\text{Vec}\,G$-bimodule structure of the category $\text{Vec}\cdot \tau$ generated by $\tau$, and $\xi$ determines the sign of the $(\tau,\tau,\tau)$-associator.

If $G$ is a locally compact group, it also produces a continuous tensor category $\Hilb\, G$ whose simple objects are indexed by $G$ \cite[Ex. 3.11]{AM2025}. If $G$ is abelian, $\chi: G\times G\to U(1)$ is a continuous symmetric nondegenerate bicharacter and $\xi\in \{\pm1\}$ is a sign, we can also produce a continuous tensor category $\TYcal(G, \chi, \xi)$ containing $\Hilb\,G$ plus an extra simple object $\tau$, whose fusion rules are such that $\tau\otimes\tau$ is a direct integral of the $G$-indexed simple objects \cite[Prop. 4.3]{AM2025}. We call continuous tensor categories with these fusion rules \emph{continuous Tambara-Yamagami tensor categories} for the group $G$ (see \cite[Def.~4.2]{AM2025}).

\begin{theorem*}[\cite{AM2025}]
    Continuous Tambara-Yamagami tensor categories for $G$ are classified by pairs $(\chi, \xi)$, for $\chi:G\times G\to U(1)$ a continuous symmetric nondegenerate bicharacter (up to automorphisms of $G$) and $\xi\in\{\pm1\}$ a sign, under the assignment $(\chi,\xi)\mapsto \TYcal(G, \chi,\xi)$.
\end{theorem*}

 We show in Proposition \ref{prop: braidingsTY} that $\TYcal(G, \chi, \xi)$ admits a unique $\Z/2$-crossed braiding, and that the induced braiding on the subcategory $\Hilb\,G$ is given by the map $C_0(G\times G)\to C_0(G\times G)$, $f(g,h)\mapsto \chi(g,h)f(g,h)$. 
 
 \subsection*{Main results} There are, up to automorphisms of $\R$, two continuous symmetric nondegenerate bicharacters on $\R$, namely
\[
\begin{array}{cccc}
\chi_+:&\R\times\R&\to &U(1)  \\
     &(x,y)&\mapsto &e^{ixy} 
\end{array}
\hspace{2cm}\text{and}
\hspace{2cm}
\begin{array}{cccc}
\chi_-:&\R\times\R&\to &U(1)  \\
     &(x,y)&\mapsto &e^{-ixy} .
\end{array}
\]

\begin{reptheorem}{thm: MainThmCrossed}\label{theoremintroTY}
The $\Z/2$-crossed braided tensor category $\Rep^{\mathbb{Z}/2}(\Heis)$ of twisted and untwisted representations of the Heisenberg conformal net is equivalent to the continuous Tambara-Yamagami tensor category
\[
\TYcal(\R, \chi_-, +1)
\]
 with its unique $\Z/2$-crossed braiding.
\end{reptheorem}

As a corollary, we obtain that $\Rep(\Heis)$ is equivalent to $\Hilb\,\R$, with the braiding induced by the map $C_0(\R\times \R)\to C_0(\R\times \R)$ given by $f(x,y)\mapsto e^{-ixy}f(x,y)$. In Theorem \ref{thm: CharacterizationRepHeis}, we also show that the continuous balance is induced by the map $C_0(\R)\to C_0(\R)$ given by $f(x)\mapsto e^{-ix^2}f(x)$. The latter result was well-known to experts but only partially available in the literature \cite{bmt88, fredenhagen, MR4677578}. We obtain that Conjecture \ref{conjectureintro} holds for $\A = \Heis.$

Combining the equivalence \eqref{eq: FinxedPointsIntro} with Theorem \ref{thm: MainThmCrossed}, we obtain the following result. \vspace{-.1cm}

\begin{reptheorem}{thm: RepFixedPointsMainTheorem}\label{maintheoremintro}
The braided tensor category $\Rep(\Heis^{\mathbb{Z}/2})$ of representations of the conformal net $\Heis^{\mathbb{Z}/2}$ is equivalent to the $\Z/2$-equivariantization 
\[
\TYcal(\R, \chi_-, +1)^{\Z/2}
\]
of the continuous Tambara-Yamagami tensor category $\TYcal(\R, \chi_-, +1)$ with its unique $\Z/2$-crossed braiding.
\end{reptheorem}

In Section \ref{sec: GettingSignsRight}, we show that there are two possible balances on $\TYcal(\R, \chi_-, +1)^{\Z/2}$, and we explicitly compute them. Although we cannot determine which of the two corresponds to the balance on $\Rep(\Heis^{\Z/2})$ (see Conjecture \ref{Conjecture: WhichCrossedBalance}), both are continuous. We expect the theory of crossed products of $\mathrm{C}^*$-algebras and $\mathrm{C}^*$-correspondences (\cite{echterhoff}) to provide a theory of equivariantizations of continuous categories by finite groups, which, used on Theorem \ref{thm: RepFixedPointsMainTheorem}, should prove that Conjecture \ref{conjectureintro} holds for the conformal net~$\Heis^{\Z/2}$.

The continuous tensor category $\TYcal(\R, \chi_-, +1)^{\Z/2}$ has irreducible objects parametrized by the half-line with a double origin, plus an extra two irreducible representations, which we denote by $\tau_+$ and $\tau_-$. Let us denote by $H_x$ the irreducible representation of $\Heis^{\Z/2}$ corresponding to a point $x\in\R_{>0}$, and by $H^\tau_\pm$ the irreducible representations corresponding to $\tau_\pm$. Denoting by $\lambda$ the restriction of the Lebesgue measure to $\R_{>0}$, it holds that all four fusion products $H^\tau_\pm\otimes H^\tau_{\pm}$ in $\Rep(\Heis^{\Z/2})$ are given by the direct integral
\[
\int^\oplus_{\R_{>0}} H_x \text{d}\lambda(x).
\]

\subsection*{Outlook}
We expect that, if $G$ is a (not necessarily finite) compact group acting on a conformal net $\A$, a version of the $G$-twisted representations of $\A$ also controls representations of $\A^G$ (cf. the equivalence \eqref{eq: FinxedPointsIntro}). We say that an $(\A^G\subset \A)$-\emph{soliton} is a representation of $\A^G$ with compatible actions of all the algebras $\A(I)$ whenever $I\subset S^1$ is away from a fixed point $\mathrm{p}\in S^1$. Combining the results in \cite{Gui_2025, fixedpoints}, it holds that, if $G$ is finite, the category $\text{Sol}_{\A^G\subset \A}$ of $(\A^G\subset \A)$-solitons is equivalent to $\Rep^G(\A)$. We expect $\text{Sol}_{\A^G\subset \A}$ to have the structure of a continuous $G$-crossed balanced tensor category. In particular, an object of $\text{Sol}_{\A^G\subset \A}$ is a direct integral of $g$-twisted representations, for $g\in G$. Taking a continuous equivariantization, we expect an equivalence of balanced tensor categories (even if $G$ is not finite)
\begin{equation}\label{eq: Lastintro}\Rep(\A^G)\cong\big(\text{Sol}_{\A^G\subset \A}\big)^G.\end{equation}

Applied to $O(n)$ acting on $\Heis^{\otimes n}$, or to subgroups of $\Aut(G)$ acting on loop group conformal nets associated to a compact, simple Lie group $G$ (these are associated to affine VOAs),~we expect the equivalence \eqref{eq: Lastintro} to produce many more examples of non-rational conformal~nets whose categories of representations can be described using the language of continuous tensor~categories.

\section*{Acknowledgements}

I am grateful to André Henriques for constant conversations, and to Sven Möller for making me realize there is a unique $\Z/2$-crossed braiding on a finite Tambara-Yamagami tensor category. I thank Josep Fontana McNally for proofreading this paper. This work has been funded by the EPSRC grant
EP/W524311/1.

For the purpose of Open Access, the author has applied a CC BY public copyright license
to any Author Accepted Manuscript version arising from this submission.

\addtocontents{toc}{\protect\setcounter{tocdepth}{5}}

\section{Preliminaries}
In this paper, all Hilbert spaces and $\mathrm{C}^*$-algebras are assumed to be separable, and all measure spaces, standard measure spaces.

\subsection{Continuous Tambara-Yamagami tensor categories}
In \cite{AM2025}, we introduced and classified continuous Tambara-Yamagami tensor categories and Tambara-Yamagami $\mathrm{W}^*$-tensor categories. In order to make this paper self-contained, we recall the main definitions and results here. We first introduce continuous tensor categories.

Let $A$ be a $\mathrm{C}^*$-algebra. A \emph{(right) Hilbert $A$-module} consists of an algebraic (right) $A$-module $\mathcal{E}$ together with a $\mathbb{C}$- and $A$-sesquilinear function $\langle -,-\rangle:\mathcal{E}\times \mathcal{E}\to A$ satisfying $\langle \xi, \eta\rangle = \langle \eta, \xi\rangle^*$, that $\langle \xi,\xi \rangle\in A$ is a positive element for all $\xi, \eta\in\mathcal{E}$ and that $(\mathcal{E}, ||\langle-, -\rangle||^{1/2})$ is complete. Given Hilbert $A$-modules $\mathcal{E}$ and $\mathcal{D}$, a $\mathbb{C}$-linear map $\mathcal{E}\to \mathcal{D}$ is \emph{adjointable} if there exists an adjoint $u^*: \mathcal{D}\to \mathcal{E}$ with respect to the $A$-valued inner product. Given another $\mathrm{C}^*$-algebra $B$, a \emph{$B - A$-correspondence} is a Hilbert $A$-module $\mathcal{E}$ with a nondegenerate left $*$-action of $B$ by adjointable operators. Note that a $\C-A$-correspondence is simply a Hilbert $A$-module and an $A-\C$-correspondence is a nondegenerate representation of $A$ on a Hilbert space. An \emph{intertwiner} between $A-B$-correspondences is a $B$-linear adjointable operator. If $C$ is another $\mathrm{C}^*$-algebra, and $\mathcal{E}$ and $\mathcal{D}$ are $C-B$- and $B-A$-correspondences respectively, there exists a relative tensor product construction producing a new $C-A$-correspondence 
\(
\mathcal{E}\otimes_B\mathcal{D}
\), see \cite[II.7.4.4]{OAbook}. We write $\text{Corr}(B, A)$ for the set of $B-A$-correspondences.

\begin{definition}(\cite[Def. 3.2]{AM2025})
    The bicategory $\MorC$ of continuous categories is the bicategory with objects $\mathrm{C}^*$-algebras, 1-morphisms $\Hom_{\MorC}(A,B) = \text{Corr}(B,A)$ and 2-~morphisms intertwiners. Composition of correspondences is given by the relative tensor product.
\end{definition}

Given a continuous category $A\in\MorC$, we think of $A$ as topologizing its underlying linear category $\Rep(A)$ of nondegenerate representations. Indeed, the set of (irreducible) representations of $A$ has a topology, the Fell topology, and one can take direct integrals of representations of $A$. If one only has access to $\Rep(A)$ as opposed to $A$, these two notions are no longer well-defined. Therefore we think of the assignment $A\mapsto \Rep(A)$ as forgetting the topology of a continuous category and recovering its underlying linear category. This assignment can be upgraded to a forgetful bifunctor $\mathfrak{F}$ from $\MorC$ into the bicategory of linear categories \cite[Prop. 3.6]{AM2025}

The maximal tensor product of $\mathrm{C}^*$-algebras makes $\MorC$ into a monoidal bicategory \cite[Prop. 3.7]{AM2025}, and leads us to the following definition.

\begin{definition}(\cite[Def. 3.8]{AM2025})
    A \emph{continuous tensor category} is a (coherently associative) algebra object in the monoidal bicategory $\MorC$ all whose structure data are unitaries.
\end{definition}

Therefore, a continuous tensor category consists of the following data: 
\begin{enumerate}
    \item a $\mathrm{C}^*$-algebra $A\in \MorC$;
    \item an $A-A\otimes A$-correspondence $\Tcal$ that encodes the tensor product;
    \item an $A-\C$-correspondence $\1$, that is, an $A$-representation $\1$ that encodes the unit;
    \item a unitary intertwiner $\alpha: \Tcal\otimes_{A\otimes A}\big(\Tcal\otimes A\big)\xrightarrow{\cong} \Tcal\otimes_{A\otimes A}\big(A\otimes \Tcal\big)$ that encodes the associator;
    \item left and right unitors given by unitary intertwiners $\lambda: \Tcal\otimes_{A\otimes A}(A\otimes \1)\xrightarrow{\cong} A$ and $\rho: \Tcal\otimes_{A\otimes A}(\1\otimes A)\xrightarrow{\cong} A$,
\end{enumerate}
satisfying a list of coherences.

\begin{example}
    Let $G$ be a locally compact group. The $\mathrm{C}^*$-algebra $C_0(G)$ of continuous functions on $G$ vanishing at infinity has a canonical structure of a continuous tensor category. Indeed, the group operation on $G$ induces a left action of $C_0(G)$ on the trivial Hilbert $C_0(G)\otimes C_0(G) = C_0(G\times G)$-module $C_0(G\times G)$. The unit is given by the 1-dimensional $C_0(G)$-representation given by multiplication after evaluating on the identity, and the associator can be chosen to be trivial. We denote this continuous tensor category by $\Hilb\, G$.
\end{example}

The underlying linear category $\Rep(A)$ of a continuous category $A\in\MorC$ is naturally a $\mathrm{W}^*$-category, and the underlying linear tensor category of a continuous tensor category is a $\mathrm{W}^*$-tensor category, defined as follows. A $*$-category is a $\C$-linear category equipped with a dagger structure $*:\Hom(X,Y)\to \Hom(Y,X)$ which is $\C$-anti-linear and satisfies $f^{**} = f$ and $(f\circ g)^* = g^*\circ f^*.$ A $*$-functor between $*$-categories is a linear functor compatible with the $*$-structure. We only consider idempotent and direct sum complete $*$-categories.

A $\mathrm{W}^*$\emph{-category} is a $*$-category $\Ccal$ such that $\End_\Ccal(X)$ is a von Neumann algebra for all $X\in \Ccal.$ A $*$-functor of $\mathrm{W}^*$-categories is \emph{normal} if it induces normal maps between von Neumann algebras of endomorphisms.

\begin{definition}
A $\mathrm{W}^*$\emph{-tensor category} is a $\mathrm{W}^*$-category $T$ with a monoidal structure whose tensor functor
\[
\otimes:T\times T\to T
\]
is a bilinear functor normal in each variable, and whose associators and unitors are unitary. 
\end{definition}

The bifunctor $\mathfrak{F}: \MorC\to \mathrm{W}^*\text{Cat}$ from the bicategory of continuous categories into the bicategory of $\mathrm{W}^*$-categories extends to a bifunctor  \[
\mathfrak{F}: \text{ContTensCat}\to \text{$\mathrm{W}^*$TensCat}
\] from the bicategory $\text{ContTensCat}$ of continuous tensor categories into the bicategory $\text{$\mathrm{W}^*$TensCat}$ of $\mathrm{W}^*$-tensor categories \cite[Prop. 3.10]{AM2025}. We do not know if this bifunctor is full.

\begin{example}\label{ex HilbW*Tensor}
    Given $G$ a locally compact group, with operation $m:G\times G\to G$, we continue denoting by $\Hilb\, G$ the image of the continuous tensor category $\Hilb\,G$ under the bifunctor $\mathfrak{F}$. The $\mathrm{W}^*$-tensor category $\Hilb\, G$ can be described as follows. The simple objects of $\Rep(C_0(G))$ are parametrized by points in $G$. Given $g\in G$, the corresponding irreducible representation of $C_0(G)$ has underlying Hilbert space $\mathbb{C}$ and a function $f\in C_0(G)$ acts by multiplication by $f(g)$. We denote such a representation by $\delta_g\in \Rep(C_0(G))$. More generally, given a locally compact, Hausdorff space $T$, a continuous map $p:T\to G$ and a measure $\upsilon$ on $T$, we obtain an object $\int^\oplus_{t\in T}\delta_{p(t)}\mathrm{d}\upsilon(t)\in \Rep(C_0(G))$ as follows. The underlying Hilbert space of $\int^\oplus_{t\in T}\delta_{p(t)}\mathrm{d}\upsilon(t)$ is $L^2(\upsilon):=L^2(T, \upsilon)$ and the $C_0(G)$-action is given by pointwise multiplication after pulling back along $p$. The full subcategory of $\Rep(C_0(G))$ on objects of this form is equivalent to $\Rep(C_0(G))$, see for example \cite[Prop.~4.4]{AM2025}.

The $\mathrm{W}^*$-tensor structure of $\Hilb\,G$ can be described as follows. Given objects $\int^\oplus_{t\in T}\delta_{p(t)}\mathrm{d}\upsilon(t)$ and $\int^\oplus_{s\in S}\delta_{n(s)}\mathrm{d}\nu(s)$ in $\Rep(C_0(G))$ defined as above, the tensor product $\int^\oplus_{t\in T}\delta_{p(t)}\mathrm{d}\upsilon(t)\otimes \int^\oplus_{s\in S}\delta_{n(s)}\mathrm{d}\nu(s)$ in $\Rep(C_0(G))$ is defined to be
\begin{equation}\label{eq: TensorHilbR}
\int^\oplus_{t\in T}\delta_{p(t)}\mathrm{d}\upsilon(t)\otimes \int^\oplus_{s\in S}\delta_{n(s)}\mathrm{d}\nu(s)  := \int^\oplus_{(t, s)\in T\times S}\delta_{m\big(p(t), n(s)\big)}\mathrm{d}(\upsilon\times\nu)(t, s).
\end{equation}
The associator is that of $\Hilb$ and the unit is $\delta_e$ for $e\in G$ the unit of $G$.
\end{example}

\begin{remark}
    If $G$ is a locally compact group, we write $\delta_g\in\Rep(C_0(G))$ for the simple representation indexed by $g\in G$, as described in Example \ref{ex HilbW*Tensor}. In addition, whenever we write $\int^\oplus_{t\in T}\delta_{p(t)}\mathrm{d}\upsilon(t)\in \Rep(C_0(G))$ we mean that $T$ is a locally compact, Hausdorff space, $p$ is a continuous map, $\upsilon$ is a measure on $T$, and the object $\int^\oplus_{t\in T}\delta_{p(t)}\mathrm{d}\upsilon(t)\in \Rep(C_0(G))$ is defined as in Example \ref{ex HilbW*Tensor}. In order to avoid cluttering the notation, we sometimes write $\upsilon$ or $(p,\upsilon)$ in place of $\int^\oplus_{t\in T}\delta_{p(t)}\mathrm{d}\upsilon(t)$ when the rest of the data is clear from the context. In line with this notation, we may also write Equation \eqref{eq: TensorHilbR} as $\upsilon\otimes\nu = \upsilon\times\nu$. If we want to recall the maps $p$ and $n$, we write $(p, \upsilon)\otimes (n, \nu) = (p+n, \upsilon\times\nu)$.
\end{remark}

We now introduce continuous Tambara-Yamagami tensor categories. In \cite{TY}, Tambara and Yamagami introduced a class of fusion categories with a unique non-invertible simple object. Given a finite abelian group $G$ (whose group operation we write additively), one can construct the finite semisimple category whose simple objects are given by elements of $G$ plus an extra simple object $\tau$, and one can consider fusion rules satisfying
\[
g\otimes h \cong  g+h\hspace{2cm} g\otimes \tau \cong\tau\cong \tau\otimes g \hspace{2cm}\tau\otimes\tau \cong \bigoplus\limits_{g\in G} g,
\]
for $g,h\in G$. Tambara and Yamagami proved that these fusion rules admit compatible associators and classified fusion categories with these fusion rules. Let $\alpha$ be a compatible set of associators for the Tambara-Yamagami fusion rules above. Then, the subcategory $\mathcal{C}_G$ spanned by $\{g\}_{g\in G}$ is a tensor subcategory with a bimodule $\text{Vec}\cdot \tau$. This implies that the associator on $\mathcal{C}_G$ can be trivialized and that the tensor category $\mathcal{C}_G$ is equivalent to $\text{Vec}\,G$, the category of $G$-graded, finite dimensional vector spaces with tensor product given by convolution. Fix $g,h\in G$ and choose isomorphisms $\phi_g: g\otimes\tau\cong \tau$ and $\phi_h: \tau\otimes h\cong \tau$. Then, the composition
\[
\tau \xrightarrow{\phi_h^{-1}} \tau\otimes h \xrightarrow{\phi^{-1}_g\otimes\id} (g\otimes \tau)\otimes h\xrightarrow{\alpha_{g,\tau, h}}g\otimes (\tau\otimes h) \xrightarrow{\id\otimes \phi_h} g\otimes \tau \xrightarrow{\phi_g} \tau
\]
is independent of $\phi_g$ and $\phi_h$ and is a multiple $\chi(g,h)\cdot\id_{\tau}$ of the identity $\id_\tau$. The pentagon identities imply that $(g,h)\in G\times G\mapsto \chi(g,h)\in\mathbb{C}^\times$ defines a symmetric nondegenerate bicharacter on $G$. Actually, the bicharacter $\chi$, together with a sign $\xi\in\{\pm1\}$ controlling the sign of the $(\tau,\tau,\tau)$-associator, completely determine the tensor category, see \cite[Thm.~3.2]{TY}. The bicharacter is taken up to the action of $\Aut(G)$. In addition, the fusion category $\TYcal(G, \chi, \xi)$ associated to such a bicharacter $\chi$ and a sign $\xi$ is explicitly given in \cite[Def. 3.1]{TY}.

Given a locally compact abelian group $G$, in \cite[Sec. 4.1]{AM2025} we described, on the continuous category $C_0(G)\oplus \mathbb{C}$, a tensor $\mathrm{C}^*$-correspondence $\TYcal_G$ mimicking the Tambara-Yamagami fusion rules. Here, instead of describing $\TYcal_G$, we describe its image $\mathfrak{F}(\TYcal_G): \Rep\Big(\big(C_0(G)\oplus\mathbb{C}\big)\otimes \big(C_0(G)\oplus\C\big)\Big)\to \Rep\big(C_0(G)\oplus\C\big)$ under the bifunctor $\mathfrak{F}$. There is a canonical equivalence of $\mathrm{W}^*$-categories  $\Rep\big(C_0(G)\oplus\C\big)\cong \Rep(C_0(G))\oplus\Rep(\C)$, and we may denote by $\tau$ the canonical simple object of $\Rep(\C) = \Hilb$. Similarly, we have $\Rep\Big(\big(C_0(G)\oplus\mathbb{C}\big)\otimes \big(C_0(G)\oplus\C\big)\Big)\cong \Rep\big(C_0(G)\otimes C_0(G)\big)\oplus\Rep(C_0(G)\otimes\C)\oplus \Rep\big(\C\otimes C_0(G)\big)\oplus \Rep(\C)$. 

The $\mathrm{C}^*$-tensor correspondence $\TYcal_G\in \text{Corr}\Big(\big(C_0(G)\oplus\mathbb{C}\big)\otimes \big(C_0(G)\oplus\C\big), C_0(G)\oplus\C\Big)$ is such that its image under $\mathfrak{F}$ induces the $\mathrm{W}^*$-tensor product 
\[\hspace{-.5cm}\mathfrak{F}(\TYcal_G): \Rep\big(C_0(G)\otimes C_0(G)\big)\oplus\Rep\big(C_0(G)\otimes\C\big)\oplus \Rep\big(\C\otimes C_0(G)\big)\oplus \Hilb\cdot \tau\to \Rep\big(C_0(G)\big)\oplus\Hilb\cdot \tau\]
restricting to the tensor product $\Rep(C_0(G)\otimes C_0(G))\to \Rep(C_0(G))$ of $\Hilb\,G$, together with
\[
\upsilon\otimes\tau \cong\tau\otimes\upsilon\cong L^2 (\upsilon)\cdot \tau \in \Rep(\C) \hspace{2cm} \tau\otimes\tau\cong L^2(G)\in \Rep(C_0(G)),
\]
for all $\int^\oplus_{t\in T}\delta_{p(t)}\mathrm{d}\upsilon(t)\in \Rep(C_0(G))$, where $L^2(G)\in \Rep(C_0(G))$ denotes the regular representation of $C_0(G)$. 
\begin{definition}(\cite[Def. 4.2]{AM2025})
    A \emph{Tambara-Yamagami continuous tensor category} for an abelian group $G$ is a continuous tensor category whose underlying continuous category is $C_0(G)\oplus\C$ and whose tensor $\mathrm{C}^*$-correspondence is $\TYcal_G$.
\end{definition}

We next introduce the classification of Tambara-Yamagami continuous tensor categories for $G$. Actually, in Section \ref{sec: RepZ/2AsTY} we will also need a characterization of $\mathrm{W}^*$-tensor categories that arise as the image under $\mathfrak{F}$ of continuous Tambara-Yamagami tensor categories. We introduced the following definition in \cite[Def. 4.6]{AM2025}, and proved that it is a characterization of $\mathrm{W}^*$-tensor categories arising under $\mathfrak{F}$ from continuous Tambara-Yamagami tensor categories in \cite[Thm. 4.18]{AM2025}.

\begin{definition}\label{def: WTY}
A \emph{Tambara-Yamagami $\mathrm{W}^*$-tensor category for $G$} is a $\mathrm{W}^*-$tensor category whose underlying $\mathrm{W}^*$-category is $\Rep(C_0(G))\oplus\Hilb\cdot\tau$ and is such that it
\begin{enumerate}
\item admits natural isomorphisms
\[
\upsilon\otimes\nu \overset{}{\cong}  \upsilon\times \nu\hspace{1cm}
\upsilon\otimes \tau \overset{}{\cong} L^2(\upsilon)\cdot\tau\overset{}{\cong}  \tau\otimes \upsilon,
\]
for all $\int^\oplus_{t\in T}\delta_{p(t)}\mathrm{d}\upsilon(t),\int^\oplus_{s\in S}\delta_{n(s)}\mathrm{d}\nu(s) \in\Rep(C_0(G))$;
\item satisfies that $\tau\otimes\tau\neq 0$ and $\Hom_{\Rep(C_0(G))\oplus \Hilb\cdot \tau}(\tau\otimes \tau, \tau) = 0$.
\end{enumerate}
\end{definition}

Note that it holds by definition that the image under $\mathfrak{F}$ of any continuous Tambara-Yamagami tensor category is a Tambara-Yamagami $\mathrm{W}^*$-tensor category, but in a Tambara-Yamagami $\mathrm{W}^*$-tensor category we do not require that $\tau\otimes \tau\cong L^2(G)$. However, the constraint that the object $\tau\otimes\tau$ admits no morphisms to $\tau$ is enough to argue that any Tambara-Yamagami $\mathrm{W}^*$-tensor category arises  as the image under $\mathfrak{F}$ of a continuous Tambara-Yamagami tensor category, see Theorem \ref{thm: TYClassificationThm}.

Both continuous Tambara-Yamagami tensor categories and Tambara-Yamagami $\mathrm{W}^*$-tensor categories are classified by a choice of a continuous symmetric bicharacter $G\times G\to U(1)$ which is nondegenerate in the sense that it induces an isomorphism $G\cong \hat{G}$ from $G$ into its Pontryagin dual, and a sign $\xi\in\{\pm1\}$. The role of $\chi$ and $\xi$ in building the associators is the same as in the case of finite Tambara-Yamagami tensor categories. Following the finite case \cite{TY}, we describe next the Tambara-Yamagami $\mathrm{W}^*$-tensor category associated to such a pair $(\chi, \xi)$. Let $\chi: G\times G\to U(1)$ be a continuous symmetric nondegenerate bicharacter and let $\xi\in \{\pm1\}$ be a sign. We write $\mu$ for the Haar measure on $G$. Note that the object $\int^\oplus_{g\in G}\delta_g\mathrm{d}\mu(g)\in \Rep(C_0(G))$ is the regular representation $L^2(G)=L^2(G, \mu)\in \Rep(C_0(G))$. We denote by $\hat{G}$ the Pontryagin dual of $G$, by $\hat{\mu}$ its Haar measure, and by $\mathcal{F}: L^2(\hat{G})\to L^2(G)$ the unitary Fourier~transform. We write $j_\chi: G\xrightarrow{\cong}\hat{G}$ for the isomorphism $j_\chi(y)(x) = \chi(x,y).$ Then, there is a unique positive real number $c_\chi\in \R_{>0}$ such that $(j_\chi)_*\mu = c_\chi\cdot \hat{\mu}$, and we define $J_\chi: L^2(G)\to L^2(\hat{G})$ as the unitary $J_\chi(f)(\eta) = c_\chi^{-1/2}f(j_\chi^{-1}(\eta))$ for all $f\in L^2(G).$

We use additive notation for the operation on $G$. Given a Hilbert space $H\in \Hilb$ and a representation $K\in \Rep(C_0(G))$, we write $H\cdot K\in \Rep(C_0(G))$ for the representation $K$ with multiplicity $H$. Note that the underlying Hilbert space of $H\cdot K$ is the Hilbert space tensor product $H\otimes K$. 

\begin{definition}\label{def: TYcal(G,chi,xi)}
The Tambara-Yamagami $\mathrm{W}^*-$tensor category $\mathcal{TY}(G, \chi, \xi)$ is the following $\mathrm{W}^*$-tensor category.
\begin{enumerate}
\item The underlying $\mathrm{W}^*$-category is $\Rep(C_0(G)\oplus \C)\cong \Rep(C_0(G))\oplus \Hilb\cdot \tau$;
\item The tensor functor is given by
\[
\upsilon\otimes \nu = \upsilon\times \nu,\hspace{1cm} \upsilon\otimes \tau= L^2(\upsilon)\cdot \tau,\hspace*{1cm} \tau\otimes \upsilon = L^2(\upsilon)\cdot\tau, \hspace{1cm} \tau\otimes \tau = \mu, 
\]
for any objects $\int^\oplus_{t\in T}\delta_{p(t)}\mathrm{d}\upsilon(t),\, \int^\oplus_{s\in S}\delta_{n(s)}\mathrm{d}\nu(s)\in\Rep(C_0(G))$;
\item The associators are as follows, where we use objects $\int^\oplus_{t\in T}\delta_{p(t)}\mathrm{d}\upsilon(t)$, $\int^\oplus_{s\in S}\delta_{n(s)}\mathrm{d}\nu(s)$, and $\int^\oplus_{r\in R}\delta_{z(r)}\mathrm{d}\eta(r)$ of $\Rep(C_0(G))$, as well as $L^2(G) = \int^\oplus_{g\in G}\delta_g\mathrm{d}\mu(g)\in \Rep(C_0(G))$,
\[
\begin{array}{cccc}
\alpha_{\upsilon, \nu, \eta}:& \upsilon\times\nu\times\eta&\xrightarrow{\id}& \upsilon\times\nu\times \eta\vspace{.3cm}\\
\alpha_{\tau, \upsilon,\nu} = \alpha_{\upsilon, \nu, \tau} :& L^2(\upsilon\times \nu)\cdot \tau& \xrightarrow{\id} &L^2(\upsilon\times\nu )\cdot \tau\vspace{.3cm}\\
\alpha_{\upsilon,\tau,\nu}  :& L^2(\upsilon\times \nu)\cdot \tau& \to &L^2(\upsilon\times\nu )\cdot \tau\\ 
&f(t,s)&\mapsto& \chi\big(p(t), n(s)\big)f(t,s)\vspace{.3cm}\\
\alpha_{\upsilon,\tau,\tau}:&L^2( \upsilon)\cdot \mu& \to &\upsilon \times\mu\\
&f(t, x)&\mapsto& f\big(t, p(t)+x\big)\vspace{.3cm}\\
\alpha_{\tau,\tau, \upsilon}:& \mu\times\upsilon& \to &L^2(\upsilon)\cdot \mu\\
&f(x,t)&\mapsto& f\big(t,-p(t)+x\big)\vspace{.3cm}\\
\alpha_{\tau,\upsilon,\tau}:& L^2(\upsilon)\cdot \mu &\to & L^2(\upsilon)\cdot \mu\\
&f(t,x)&\mapsto &\chi\big(p(t), x\big)\cdot f(t,x)
\end{array}
\]
and 
\[	
\alpha_{\tau,\tau,\tau} : L^2(G,\mu)\cdot \tau\xrightarrow{J_\chi}L^2(\hat{G},\hat{\mu})\cdot \tau\xrightarrow{\xi\cdot\mathcal{F}}L^2(G,\mu)\cdot \tau.
\]
\end{enumerate}
\end{definition}

Such a $\mathrm{W}^*$-tensor category $\TYcal(G, \chi, \xi)$ is indeed a Tambara-Yamagami $\mathrm{W}^*$-tensor category \cite[Cor. 4.5]{AM2025}, meaning that the pentagon equations hold. We can now state the classification result of continuous Tambara-Yamagami and Tambara-Yamagami $\mathrm{W}^*$-tensor~categories.

\begin{theorem}(\cite[Thm. 4.18]{AM2025})\label{thm: TYClassificationThm} Let $G$ be a locally compact abelian group. There are bijections
\[\hspace{-.8cm}\begin{tikzcd}[column sep=small]
    \begin{array}{c}\nicefrac{\left\{\begin{array}{l}(\chi, \xi)\ |\ \text{$\chi\phantom{,}$: $G\times G\to U(1)$ a continuous}\\     \text{symmetric nondegenerate bicharacter} \\ \text{and $\xi\in\{\pm1\}$}   \end{array}\right\}}{\text{Aut}(G)}\end{array}
    &&
    \begin{array}{c} \nicefrac{\left\{\begin{array}{l}\text{Tambara-Yamagami }\mathrm{W}^*\text{-}\\     \text{tensor categories for $G$}   \end{array}\right\}}{\mathrm{W}^*-\otimes\ \text{equiv.}} \end{array}
    \\
    \begin{array}{c} \nicefrac{\left\{\begin{array}{l}\text{continuous Tambara-Yamagami}\\     \text{tensor categories for $G$}   \end{array}\right\}}{\text{continuous $\otimes$ equiv.}} \end{array}
	\arrow["{\TYcal(G, -,-)}", from=1-1, to=1-3]
	\arrow["{\TYcal(G, -, -)}"', from=1-1, to=2-1]
	\arrow["\mathfrak{F}"', from=2-1, to=1-3]
\end{tikzcd}\]
\end{theorem}

    Theorem \ref{thm: TYClassificationThm} implies that, given a Tambara-Yamagami $\mathrm{W}^*$-tensor category $\mathcal{C}$, there exists a unique (up to equivalence) continuous tensor category whose image under $\mathfrak{F}$ is equivalent to $\mathcal{C}$. For this reason, in the rest of the paper we only work with $\mathrm{W}^*$-tensor categories as opposed to with continuous tensor categories. However, when we characterize certain $\mathrm{W}^*$-tensor categories as being Tambara-Yamagami $\mathrm{W}^*$-tensor categories, the result above ensures that we are also characterizing them as continuous tensor categories.

\subsection{$G$-crossed balanced categories}

Let $G$ be a discrete group and let $\mathcal{C}$ be a $\mathrm{W}^*$-tensor category (admitting all countable orthogonal direct sums). In this section, we recall the notion of a $G$-crossed balanced $\mathrm{W}^*$-tensor structure on $\mathcal{C}$ and of its equivariantization $\mathcal{C}^G$. We refer the reader to \cite[Sec. 2.2]{GcrossedbraidedRep} and \cite[Sec. 2.2]{fixedpoints} for a more detailed introduction.

We write $\underline{G}$ for the tensor category with objects $G$ and only identity morphisms. We denote by $\underline{\text{Aut}}_\otimes(\mathcal{C})$ the tensor category of $\mathrm{W}^*$-tensor automorphisms of $\mathcal{C}$ and unitary $\mathrm{W}^*$-tensor natural transformations. An \emph{action} of $G$ on $\mathcal{C}$ is a tensor functor
\(
T:\underline{G}\to \underline{\text{Aut}_\otimes(\mathcal{C})}.
\)
The data of $T$ provides, in particular, $\mathrm{W}^*$-functors $T_g: \mathcal{C}\to \mathcal{C}$ with unitary tensor structures $s_g: T_g(-)\otimes T_g(-)\xrightarrow{\cong} T_g(-\otimes -)$, and unitary natural isomorphisms $\eta_{g,h}:T_g\circ T_h \xrightarrow{\cong} T_{gh}$ for all $g, h\in G$.

\begin{definition}\label{def: G-X-braidedWTensorCat}
    A \textit{$G$-crossed braided $\mathrm{W}^*$-tensor category} consists of
    \begin{enumerate}
        \item a $\mathrm{W}^*$-tensor category $\mathcal{C}$;
        \item an action $T: \underline{G}\to \underline{\text{Aut}}_\otimes(\mathcal{C})$ of $G$ on $\mathcal{C}$;
        \item a grading $\mathcal{C} = \bigoplus\limits_{g\in G}\mathcal{C}_g$ compatible with the tensor product in the sense that $-\otimes-$ provides functors 
        \[
        \Ccal_g\times\Ccal_h\to \Ccal_{gh}
        \]
        for all $g,h\in G$;
        \item for every $g\in G$, $X\in \mathcal{C}_g$ and $Y\in \mathcal{C}$, a unitary isomorphism
        \[
        \beta_{X,Y} : X\otimes Y\to T_g(Y)\otimes X
        \]
        natural in $X$ and $Y$.
    \end{enumerate}
    This data is required to satisfy
    \begin{enumerate}
        \item $T_g(\mathcal{C}_h) \subset \Ccal_{ghg^{-1}}$ for all $g,h\in G$;
        \item the following diagram commutes 
\[\begin{tikzcd}
	{T_g(X)\otimes T_g(Y)} && {T_{ghg^{-1}}T_g(Y)\otimes T_g(X)} \\
	{T_g(X\otimes Y)} && {T_g(T_h(Y)\otimes X)}
	\arrow["{\beta_{T_g(X),T_g(Y)}}", from=1-1, to=1-3]
	\arrow["\cong", from=1-1, to=2-1]
	\arrow["\cong"', from=1-3, to=2-3]
	\arrow["{T_g(\beta_{X,Y})}"', from=2-1, to=2-3]
\end{tikzcd}\]
for all $g,h\in G$, $X\in\mathcal{C}_h$ and $Y\in\mathcal{C}$;
\item the following diagram commutes
\begin{equation}\label{eq: Xbraiding1}\begin{tikzcd}
	{X\otimes Y\otimes Z} && {T_g(Y\otimes Z)\otimes X} \\
	{T_g(Y)\otimes X\otimes Z} && {T_g(Y)\otimes T_g(Z)\otimes X}
	\arrow["{\beta_{X, Y\otimes Z}}", from=1-1, to=1-3]
	\arrow["{\beta_{X,Y}\otimes\id_Z}"', from=1-1, to=2-1]
	\arrow["{\id_{T_g(Y)}\otimes \beta_{X,Z}}"', from=2-1, to=2-3]
	\arrow["\cong"', from=2-3, to=1-3]
\end{tikzcd}\end{equation}
for all $g\in G$, $X\in \mathcal{C}_g$ and $Y,Z\in \mathcal{C}$;
\item the following diagram commutes
\begin{equation}\label{eq: Xbraiding2}
\begin{tikzcd}
	{X\otimes Y\otimes Z} && {T_{gh}(Z)\otimes X\otimes Y} \\
	{X\otimes T_h(Z)\otimes Y} && {T_gT_h(Z)\otimes X\otimes Y}
	\arrow["{\beta_{X\otimes Y,Z}}", from=1-1, to=1-3]
	\arrow["{\id_X\otimes \beta_{Y,Z}}"', from=1-1, to=2-1]
	\arrow["{\beta_{X, T_h(Z)}\otimes \id_Y}"', from=2-1, to=2-3]
	\arrow["\cong"', from=2-3, to=1-3]
\end{tikzcd}\end{equation}
for all $g,h\in G$, $X\in \mathcal{C}_g$, $Y\in\mathcal{C}_h$ and $Z\in\mathcal{C}$.
    \end{enumerate}
\end{definition}
The unlabelled vertical isomorphisms in the diagrams above are constructed using the natural transformations $s$ and $\eta$, and the associators are omitted for readability. Given $(\mathcal{C}, \otimes^\Ccal, T^\Ccal, \beta^\Ccal)$ and $(\mathcal{D}, \otimes^\Dcal, T^\Dcal, \beta^\Dcal )$ two $G$-crossed braided $\mathrm{W}^*$-tensor categories, a functor between them consists of a $\mathrm{W}^*$-tensor functor $(F,\Psi): (\Ccal,\otimes^\Ccal)\to (\Dcal, \otimes^\Dcal)$ preserving the $G$-grading, together with unitary natural isomorphisms
\(
\Phi_g(X): F(T^\Ccal_g(X))\xrightarrow{\cong } T_g^\Dcal(F(X))
\)
indexed by $g\in G$ and $X\in \Ccal$. These are required to be compatible with the structure data of the action $T$ and to satisfy the following compatibility with the $G$-crossed braiding, 
\begin{equation}\label{eq: Gcrossedfunctor}\begin{tikzcd}
{F(X\otimes^\Ccal Y)} && {F(T_g^\Ccal(Y)\otimes^\Ccal X)} \\
{F(X)\otimes^\Dcal F(Y)} && {F(T_g^\Ccal(Y))\otimes^\Dcal F(X)} \\
& {T_g^\Dcal(F(Y))\otimes^\Dcal F(X)}
\arrow["{F(\beta^\Ccal_{X,Y})}", from=1-1, to=1-3]
\arrow["{\Psi_{X,Y}}"', from=1-1, to=2-1]
\arrow["{\Psi_{T_g^\Ccal(Y), X}}", from=1-3, to=2-3]
\arrow["{\beta_{F(X),F(Y)}^\Dcal}"', from=2-1, to=3-2]
\arrow["{\Phi_{g}(Y)\otimes\id_{F(X)}}", from=2-3, to=3-2]
\end{tikzcd}\end{equation}
for all $g\in G$ and $X\in \mathcal{C}_g$ and $Y\in\Ccal$. We call such a triple $(F,\Psi,\Phi)$ a \emph{$G$-crossed braided $\mathrm{W}^*$-tensor functor}. A full list of the required compatibilities can be found in \cite[Sec. 2.2]{GcrossedbraidedRep}. A $G$-crossed braided $\mathrm{W}^*$-tensor functor is an equivalence of $G$-crossed braided $\mathrm{W}^*$-tensor categories if its underlying $\mathrm{W}^*$-functor is an equivalence of $\mathrm{W}^*$-categories.

\begin{definition}Given a $G$-crossed braided $\mathrm{W}^*$-tensor category $\mathcal{C}$, a \emph{$G$-crossed balance} on $\mathcal{C}$ is a natural family of unitary isomorphisms $\theta: X\to T_g(X)$ parametrized by $g\in G$ and $X\in \mathcal{C}_g$ such that $T_g(\theta_X) = \theta_{T_g(X)}$ and the following diagram commutes
\begin{equation}\label{eq: Xtwist}
\begin{tikzcd}
{X\otimes Y} && {T_{gh}(X\otimes Y)} \\
&& {T_{gh}(X)\otimes T_{gh}(Y)} \\
&& {T_{ghgh^{-1}g^{-1}}T_{ghg^{-1}}(X)\otimes T_{ghg^{-1}}T_g(Y)} \\
{T_g(Y)\otimes X} && {T_{ghg^{-1}}(X)\otimes T_g(Y)}
\arrow["{\theta_{X\otimes Y}}", from=1-1, to=1-3]
\arrow["{\beta_{X,Y}}"', from=1-1, to=4-1]
\arrow["\cong", from=1-3, to=2-3]
\arrow["\cong"', from=3-3, to=2-3]
\arrow["{\beta_{T_g(Y), X}}"', from=4-1, to=4-3]
\arrow["{\theta_{T_{ghg^{-1}}(X)}\otimes\theta_{T_g(Y)}}"', from=4-3, to=3-3]
\end{tikzcd}\end{equation}
for all $g,h\in G$, $X\in \mathcal{C}_g$, $Y\in\mathcal{C}_h$. A \emph{$G$-crossed balanced $\mathrm{W}^*$-tensor category} is a $G$-crossed braided $\mathrm{W}^*$-tensor category with a $G$-crossed balance.
\end{definition}

Given $\Ccal$ and $\Dcal$ two $G$-crossed braided $\mathrm{W}^*$-tensor categories with balances $\theta^\Ccal$ and $\theta^\Dcal$ respectively, \emph{a $G$-crossed balanced $\mathrm{W}^*$-tensor functor} $(F,\Psi,\Phi):\Ccal\to \Dcal$ is a functor of $G$-crossed braided $\mathrm{W}^*$-tensor categories such that $\Phi_g(X)\circ F(\theta^\mathcal{C}_X) = \theta^\mathcal{D}_{F(X)}$ for all $g\in G$ and $X\in\Ccal_g$.

Let $\mathcal{C}$ be a $G$-crossed balanced $\mathrm{W}^*$-tensor category. Then, one can canonically obtain a balanced category $\mathcal{C}^G$ by means of an equivariantization procedure, see \cite[Sec. 2.2]{fixedpoints} for a detailed construction, proofs, and further references. 
 
The $G$-\emph{equivariantization} $\mathcal{C}^G$ of $\mathcal{C}$ is the $\mathrm{W}^*$-category with objects pairs $(X,u)$ where $X\in \mathcal{C}$ and $u = \{u_g\}_{g\in G}$ is a family of unitary isomorphisms $u_g: T_g(X)\cong X$ such that the diagram
\[\begin{tikzcd}
{T_gT_h(X)} && {T_g(X)} \\
{T_{gh}(X)} && X
\arrow["{T_g(u_h)}", from=1-1, to=1-3]
\arrow["\cong"', from=1-1, to=2-1]
\arrow["{u_g}", from=1-3, to=2-3]
\arrow["{u_{gh}}"', from=2-1, to=2-3]
\end{tikzcd}\]
commutes for every $g,h\in G$. A morphism between objects $(X,u)$ and $(Y,v)$ in $\mathcal{C}^G$ is a morphism from $X$ to $Y$ in $\mathcal{C}$ compatible with $u_g$ and $v_g$ for all $g\in G$.

The $\mathrm{W}^*$-category $\mathcal{C}^G$ inherits a tensor structure
\[
(X,u)\otimes (Y,v):=(X\otimes Y,u\otimes v)
\]
from that of $\mathcal{C}$, where $u\otimes v = \{(u_g\otimes v_g)\circ s_g^{-1}(X,Y)\}_{g\in G}$. The fact that the action of $G$ on $\mathcal{C}$ is by tensor automorphisms implies that the associators, unit and unitors of $\mathcal{C}$ descend to data on $\mathcal{C}^G$ providing its tensor structure. Finally, the $G$-crossed braiding $\beta$ on $\mathcal{C}$ induces an honest braiding $\beta^G$ on $\mathcal{C}^G$ as follows. If $X\in \mathcal{C}_g$ and $u_g: T_g(Y)\cong Y$ is an isomorphism, we can define
\[
X\otimes Y\xrightarrow{\beta_{X,Y}}T_g(Y)\otimes X\xrightarrow{u_g\otimes \id} Y\otimes X.
\]
These isomorphisms can be extended linearly to obtain, for every $(X,u),(Y,v)\in \mathcal{C}^G$, a unitary isomorphism
\[
{\beta^G}_{(X,u),(Y,v)}: (X,u)\otimes (Y,v)\xrightarrow{\cong}(Y,v)\otimes(X,u).
\]
The hexagon axioms for $\beta^G$ follow from the required coherences on $\beta$. Finally, the $G$-crossed balance $\theta$ on $\mathcal{C}$ provides an honest balance on $\mathcal{C}^G$. If $X\in\mathcal{C}_g$ and $u_g:T_g(X)\cong X$ is an isomorphism, we can write the map
\[
X\xrightarrow{\theta_X}T_g(X)\xrightarrow{u_g}X,
\]
which can be extended linearly to obtain unitary isomorphisms 
\[
{\theta^G}_{(X,u)}: (X,u)\to (X,u)
\]
for every $(X,u)\in \mathcal{C}^G$. The diagram \eqref{eq: Xtwist} implies that ${\theta}^G$ is a balance on the braided $\mathrm{W}^*$-tensor category $(\mathcal{C}^G,{\beta^G}).$

Given two $G$-crossed braided $\mathrm{W}^*$-tensor categories $\Ccal$ and $\Dcal$, it is straightforward to check that a $G$-crossed braided $\mathrm{W}^*$-functor from $\Ccal$ to $\Dcal$ defines a braided $\mathrm{W}^*$-tensor functor between the equivariantizations $\Ccal^G$ and $\Dcal^G$. If $\Ccal$ and $\Dcal$ are furthermore $G$-crossed balanced, and the functor is compatible with the crossed balances, then the induced functor $\Ccal^G\to \Dcal^G$ is compatible with the balances.

\subsection{Conformal nets}
\label{Sec: ConfNets}

Let $S^1:=\{z\in \C\ |\ |z| = 1\}$ denote the standard circle. We say that an \emph{interval} of $S^1$ is an open, connected, non-empty, non-dense subset of $S^1$, and we denote by $\Jcal$ the collection of intervals of $S^1$. We write $\Mob$ for the set of Möbius transformations of $S^1$, that is, transformations of the form
\[
z\mapsto \frac{az+b}{\bar{b}z+\bar{a}}
\]
with $a,b\in \C$, $|a|^2-|b|^2 = 1$. Then, $\Mob$ is a Lie group isomorphic to $PSL(2, \R).$ In particular, the group $\Rot$ of rotations of the circle embeds in $\Mob$. We denote an anticlockwise rotation of angle $\theta$ by $R_\theta$. Given a Hilbert space $H$ and a $*$-subalgebra $A\subset B(H)$, we denote by $A''\subset B(H)$ the von Neumann algebra defined as the double commutant of $A$ in $B(H).$

\begin{definition}
A \emph{Möbius covariant net on $S^1$} is a tuple $(H_0, \A, U,\Omega)$ where $H_0$ is a Hilbert space equipped with a nonzero vector $\Omega\in H_0$, $U$ is a strongly continuous representation of $\Mob$ on $H_0$ and $\A$ is an assignment of a von Neumann algebra $\A(I)$ acting on $H_0$ for every interval $I\in\Jcal$. This data is required to satisfy, for every $I,J\in\Jcal$ and $\varphi\in\Mob$,
\begin{enumerate}
\item isotony: if $J\subset I$, then $\A(J)\subset \A(I)$;
\item locality: if $I\cap J = \emptyset$, then $\A(J)$ and $\A(I)$ commute in $B(H_0)$,
\item Möbius covariance: $U(\varphi)\A(I)U(\varphi)^* = \A(\varphi I)$,
\item positivity of the energy: the representation $U$ is of positive energy, meaning that the conformal Hamiltonian $L_0$, defined by $U(R_\theta) = e^{i\theta L_0}$ is positive,
\item uniqueness of the vacuum: the vector $\Omega\in H_0$ is the unique vector, up to a constant, which is invariant under $U$,
\item cyclicity of the vacuum: $\Omega$ is cyclic for the von Neumann  algebra $\A(S^1):=\{\A(I)\ |\ I\subset S^1\}''\subset B(H_0)$.
\end{enumerate}
\end{definition}

One of the consequences of the definition of Möbius covariant net is Haag duality: given a Möbius covariant net $(H, \A, U, \Omega)$ and $I\in\Jcal$, the commutant of $\A(I)$ in $B(H_0)$ is $\A(I^c)$, where $I^c$ denotes the interior of $S^1\setminus I$ \cite[Thm. 2.19(ii)]{gf93}. It also holds that each of the local algebras $\A(I)$ is a type $\mathrm{III}_1$ factor \cite[Prop. 1.2]{MR1410566}, and the Reeh-Schlieder Theorem: the vacuum vector $\Omega$ is cyclic and separating for each of the algebras $\A(I)$ \cite[Cor. 2.8]{gf93}.

Let $\Diff^+(S^1)$ denote the group of orientation-preserving diffeomorphisms of $S^1$. A strongly continuous projective unitary representation $V$ of $\Diff^+(S^1)$ on a Hilbert space $H$ is a strongly continuous homomorphism $V: \Diff^+(S^1)\to \mathcal{P}U(H)$. Writing $[-]: U(H)\to \mathcal{P}U(H)$ for the projection map, we say that $V$ is an extension of a unitary representation $U$ of $\Mob$ on $H$ if, for every $\varphi\in \Mob$, it holds that $[U(\varphi)] = V(\varphi)$.

\begin{definition}\label{def: ConformalNet}
A Möbius covariant net $(H_0, \A ,U, \Omega)$ is said to be a \emph{conformal net} if there is an extension of $U$ to a unitary projective representation of $\Diff^+(S^1)$ on $H_0$, which we also denote by $U$, such that for all $I\in\Jcal$ and $\varphi\in \Diff^+(S^1)$, it holds that
\begin{enumerate}
\item $U(\varphi)\A(I) U(\varphi)^* = \A(\varphi I)$;
\item if $\varphi|_{I} = \id_I$, then $\text{Ad}(U(\varphi))|_{\A(I)} = \id_{\A(I)}.$
\end{enumerate}
\end{definition}

Let $(H_0,\A,U, \Omega)$ be a conformal net, which we simply denote by $\A$ from now on.

\begin{definition}\label{def: Rep}
A representation of a conformal net $\A$ consists of a Hilbert space $H$ and a collection of $*$-homomorphisms $\pi_I:\A(I)\to B(H)$ for every interval $I\in\Jcal$ such that
\[
\pi_I|_{\A(J)} = \pi_J
\]
whenever $J\in\Jcal$ is another interval such that $J\subset I$. A morphism between two representations of $A$ is a bounded linear map between the underlying Hilbert spaces equivariant with respect to the actions of $\A(I)$ for all $I\in\Jcal$. We write $\Rep(\A)$ for the category of representations~of~$\A$.
\end{definition}

Since we only consider representations on separable Hilbert spaces, all the $*$-actions $\A(I)\to B(H)$ are automatically normal, for any representation $H\in \Rep(\A)$ \cite[V.5.1]{MR1873025}. By definition, the vacuum Hilbert space $H_0$ comes with the structure of a representation of $\A$, called the \emph{vacuum representation}. For an interval $I\in \Jcal$ and $x\in \A(I)$, we write $\pi_{0,I}(x)$, or simply $\pi_0(x)$, for the action of $x$ on $H_0$. The category $\Rep(\A)$ is a braided tensor category, the structure of which can be produced in different equivalent ways \cite{gf93, was98, bdh15,Gui21}. The relevant construction for this paper is that of  \cite[Sec. 2]{Gui21}, also considered in \cite[Sec. 3]{GcrossedbraidedRep}. We note that these two references slightly differ in some conventions, in particular in the order in which representations are tensored. We follow the conventions of the latter, see~Remark~\ref{rk: bop}. 

Let us reproduce here the construction of the braided tensor structure on $\Rep(\A)$. Given $H,K\in \Rep(\A)$ two representations of $\A$ and $I\in \Jcal$ an interval, we write $\Hom_{\A(I)}(H, K)$ for the space of bounded linear maps which are equivariant with respect to the actions of the von Neumann algebra $\A(I)$. Since $\A(I)$ is a type $\mathrm{III}$-factor and its actions on $H$ and $K$ are normal, there are unitary maps in $\Hom_{\A(I)}(H, K)$. The space of \emph{$I$-bounded vectors} of $H$ is defined to be
\[
H(I):=\{\xi\in H\,|\, \xi = T\Omega\ \text{  for some $T\in\Hom_{\A(I^c)}(H_0, H) $}\}.
\]
Given $\xi\in H(I)$, there is a unique operator $T\in\Hom_{\A(I^c)}(H_0, H) $ such that $\xi = T\Omega$, by the Reeh-Schlieder Theorem, and we denote it by $T = Z(\xi, I)$. If $H = H_0$, then $H_0(I) = \A(I)\Omega$ by Haag duality, and the inclusion $H_0(I)\subset H_0$ is dense. Since there are unitary operators in $\Hom_{\A(I^c)}(H_0, H)$, the space $H(I)$ is dense in $H$ for any representation $H\in\Rep(\A)$.  If $I_1\in \Jcal$ is another interval with $I_1\subset I$, then $H(I_1)\subset H(I)$ is a dense inclusion. Given $K\in \Rep(\A)$ and $J\in \Jcal$ disjoint from $I$, we define the \emph{Connes fusion of $H$ and $K$ over $I$ and $J$}, denoted
\begin{equation}\label{eq: ConnesFusion2Ints}
H(I)\boxtimes K(J),
\end{equation}
as the completion of the algebraic tensor product $H(I)\otimes K(J)$ with respect to the positive sesquilinear form
\[
\langle \xi\otimes\eta, \xi'\otimes\eta'\rangle : = \langle Z(\eta', J)^*Z(\eta, J)Z(\xi', I)^*Z(\xi, I)\Omega,\Omega\rangle 
\]
for $\xi,\xi'\in H(I)$ and $\eta, \eta'\in  K(J)$. Note that the swap map $H(I)\otimes K(J)\to K(J)\otimes H(I)$ induces a canonical unitary $H(I)\boxtimes K(J)\cong K(J)\boxtimes H(I)$. Given $I_1,J_1\in  \Jcal$ with $I_1\subset I$ and $J_1\subset J$, we have a unitary equivalence
\begin{equation}\label{eq: EquivInclusionIntervals}
H(I_1)\boxtimes K(J_1)\xrightarrow{\cong} H(I)\boxtimes K(J)
\end{equation}
induced by the inclusion $H(I_1)\otimes K(J_1)\hookrightarrow H(I)\otimes K(J)$. Given two distinct points $z,\zeta\in S^1$, the \emph{Connes fusion of $H$ and $K$ over $z$ and $\zeta$} is defined to be the Hilbert space
\[
H(z)\boxtimes K(\zeta):=\varinjlim\limits_{(z,\zeta)\in I\times J}H(I)\boxtimes K(J)\ = \Bigg(\bigsqcup\limits_{(z,\zeta)\in I\times J} H(I)\boxtimes K(J)\Bigg)\Big/\cong,
\]
where the limit runs over pairs of disjoint intervals $I,J\in\Jcal$ containing $z$ and $\zeta$ respectively, and $\cong$ denotes the equivalence obtained via the unitaries $H(I_1)\boxtimes K(J_1)\cong H(I)\boxtimes K(J)$ in Equation \eqref{eq: EquivInclusionIntervals} whenever $I_1,J_1\in\Jcal$ are intervals such that $I_1\subset I$ and $J_1\subset J$. Given any disjoint intervals $I,J\in\Jcal$ containing $z$ and $\zeta$ respectively, the canonical map $H(I)\boxtimes K(J)\to H(z)\boxtimes K(\zeta)$ is a unitary equivalence.

We next relate the Connes fusion of representations over different sets of disjoint intervals. For any continuous path $\gamma = (\alpha, \beta):[0,1]\to \Conf_2(S^1)=\{(z,\zeta)\in S^1\times S^1\,|\, z \neq \zeta\}$ from $(z_0,\zeta_0):= \gamma(0)$ to $(z_1,\zeta_1):=\gamma(1)$, we obtain a canonical unitary equivalence, called a \emph{path-continuation},
\[
\gamma^\bullet: H(z_0)\boxtimes K(\zeta_0)\xrightarrow{\cong}  H(z_1)\boxtimes K(\zeta_1),
\]
defined as follows. If $\gamma$ is small enough so that there exist $I,J\in \Jcal$ disjoint intervals such that $\alpha([0,1])\subset I$ and $\beta([0,1])\subset J$, we define $\gamma^\bullet$ to be the composition
\[
H(z_0)\boxtimes K(\zeta_0)\xrightarrow{\cong } H(I)\boxtimes K(J)\xrightarrow{\cong} H(z_1)\boxtimes K(\zeta_1).
\]
For a general $\gamma$, we pick a partition $0 = t_0< t_1<\ldots <t_n = 1$ of $[0,1]$ such that each $\gamma|_{[t_i,t_{i+1}]}$ is small in the sense above, and define $\gamma^\bullet := \gamma|_{[t_{n-1}, 1]}^\bullet\circ\ldots\circ \gamma|_{[0, t_1]}^\bullet$. It is easy to see that $\gamma^\bullet$ is independent of the chosen partition. Given intervals $I_0,I_1,J_0,J_1\in\Jcal$ with $I_i$ and $J_i$ disjoint for $i = 0,1$, and a path $\gamma:[0,1]\to \Conf_2(S^1)$ from a point $(z_0,\zeta_0):=\gamma(0)\in I_0\times J_0$ to a point $(z_1,\zeta_1):=\gamma(1)\in I_1\times J_1$, we define the path continuation $\gamma^\bullet: H(I_0)\boxtimes K(J_0)\to H(I_1)\boxtimes K(J_1)$ to be the unitary equivalence
\[
H(I_0)\boxtimes K(J_0)\xrightarrow{\cong}H(z_0)\boxtimes K(\zeta_0)\xrightarrow[\cong]{\gamma^\bullet} H(z_1)\boxtimes K(\zeta_1)\xrightarrow{\cong} H(I_1)\boxtimes K(J_1).
\]

Fix $I,J\in \Jcal$ disjoint intervals. We next endow $H(I)\boxtimes K(J)$ with an action $\pi^{H \boxtimes K}$ of $\A$. Let us denote by $\pi^H$ and $\pi^K$ the actions of $\A$ on $H$ and $K$ respectively. Given an interval $L\in\Jcal$ such that $L\subset I$, and $x\in \A(L)$, we write $\pi^{H\boxtimes K}_L(x): H(I)\boxtimes K(J)\to H(I)\boxtimes K(J)$ for the morphism induced by $\pi^{H}_L(x)\otimes \id: H(I)\otimes K(J)\to H(I)\otimes K(J)$. If $L\in\Jcal$ is such that $L\subset J$ and $x\in\A(L)$, we similarly define $\pi^{H\boxtimes K}_L(x)$ as the completion of $\id\otimes\pi^K_L(x)$. Given a generic interval $L\in \Jcal$ and $x\in \A(L)$, let $M\in \Jcal$ be an interval disjoint from $L$, and $\gamma:[0,1]\to \Conf_2(S^1)$ a path from $\gamma(0)\in I\times J$ to $\gamma(1)\in L\times M$. We define the action of $x$ on $H(I)\boxtimes K(J)$ as the composition
\begin{equation}\label{eq: ActionOnConnesFusion}
\pi^{H\boxtimes K}_{L}(x) : H(I)\boxtimes K(J)\xrightarrow{\gamma^\bullet} H(L)\boxtimes K(M)\xrightarrow{\pi^{H\boxtimes K}_L(x)} H(L)\boxtimes K(M)\xrightarrow{(\gamma^\bullet)^{-1}} H(I)\boxtimes K(J).
\end{equation}
It can be seen that this provides a well-defined action of $\A$ on $H(I)\boxtimes K(J)$ \cite[Thm. 2.15]{Gui21}. If we had taken another path $\rho:[0,1]\to \Conf_2(S^1)$ from $I\times J$ to $M\times L$ and had considered the representation defined by
\[
H(I)\boxtimes K(J)\xrightarrow{\rho^\bullet} H(M)\boxtimes K(L)\xrightarrow{ \pi^{H\boxtimes K}_L(x)} H(M)\boxtimes K(L)\xrightarrow{(\rho^\bullet)^{-1}} H(I)\boxtimes K(J),
\]
we would have obtained an isomorphic representation of $\A$. The Connes fusion of morphisms of $\A$-representations, defined in the obvious way, is again a morphism of $\A$-representations, and so are path continuations.

Let $S_-^1$ and $S_+^1$ denote the lower and upper hemispheres of $S^1$ respectively. We define the tensor product of two representations $H,K\in\Rep(\A)$ of $\A$ to be
\begin{equation}\label{eq: ConnesFusionRepA}
H\boxtimes K:=H(S_-^1)\boxtimes K(S^1_+)\in \Rep(\A),
\end{equation}
where the $\A$-action is defined as above. Providing an associator for this tensor product is straightforward, see \cite[Sec. 2.5]{Gui21}. The unit object is the vacuum representation $H_0$ and the unitors can be found in \cite[Sec. 2.6]{Gui21}.

Let us next discuss the braiding. We denote by $\varrho: [0,1]\to \Conf_2(S^1)$ the path $\varrho(t) = (-ie^{\pi i t}, i e^{\pi i t})$, which goes from $(-i, i)\in S^1_-\times S^1_+$ to $(i, -i)\in S^1_+\times S^1_-$. The braiding between $H$ and $K$ is defined to be
\[
\mathbb{B}_{H, K}: H\boxtimes K = H(S^1_-)\boxtimes K(S^1_+)\xrightarrow{\varrho^\bullet} H(S^1_+)\boxtimes K(S^1_-)\cong K(S^1_-)\boxtimes H(S^1_+) = K\boxtimes H,
\]
where the non-labelled equivalence is induced by the swap map. One can show that $\mathbb{B}$ indeed satisfies the hexagon conditions and hence it provides a braiding on $\Rep(\A)$, see \cite[Thm. 3.8 and Prop. 2.12]{Gui21} or \cite[Thm. 3.25 and Prop. 3.5]{GcrossedbraidedRep}.

\begin{remark}\label{rk: bop}
    Given a braided tensor category $(\mathcal{C},\otimes, \beta)$, one can produce its braided opposite tensor category $\mathcal{C}^{\text{bop}} = (\mathcal{C}, \otimes^{\text{op}}, \beta^{\text{op}})$. The underlying linear category of $\mathcal{C}^{\text{bop}}$ is that of $\mathcal{C}$ but the tensor product is defined to be $X\otimes^{\text{bop}}Y := Y\otimes X$ for $X,Y\in\mathcal{C}$. The braiding is given by $\beta^{\text{op}}_{X, Y} = Y\otimes X\xrightarrow{\beta_{Y, X}^{-1}}X\otimes Y$. What we have described in this section is the braided monoidal opposite of the braided tensor category of representations of $\A$ described in \cite{Gui21}. We do so in order to make our constructions agree with those in the more general setting of twisted representations of $\A$, where the freedom to take $\Rep(\A)$ or $\Rep(\A)^{\text{bop}}$ disappears \cite{GcrossedbraidedRep}.   
\end{remark}

In \cite{balanceRepA}, we showed that $\Rep(\A)$ also admits a canonical balance, obtained as follows. Since $\A$ is a conformal net, its vacuum Hilbert space $H_0$ comes equipped with a projective unitary action of $\Diff^+(S^1)$. Recall that the Lie algebra of $\Diff^+(S^1)$ is the Lie algebra $\Vect(S^1)$ of vector fields on $S^1$, where the Lie bracket is the negative of the usual Lie bracket of vector fields. We let $\exp: \Vect(S^1)\to \Diff^+(S^1)$ denote the exponential map and write $\Vect_{\C}(S^1)$ for the complexification of $\Vect(S^1)$. For each $n\in \Z$, we define the complex vector field $L_n(e^{i\theta}):=-ie^{in\theta}\frac{d}{d\theta}\in \Vect_{\mathbb{C}}(S^1)$. These vectors $L_n$ form the Witt algebra $\mathscr{W}$, which is a dense Lie algebra of $\Vect_{\C}(S^1).$ We define a $*$-structure on $\mathscr{W}$ by setting $L_n^* = L_{-n}$. Then, a vector $X\in \mathscr{W}$ is self-adjoint if and only if $iX\in \Vect(S^1)$. For a self-adjoint vector $X\in \Vect_\C(S^1)$, we write $\exp_{iX}$ for the one parameter subgroup of $\Diff^+(S^1)$ given by $t\in\R\mapsto \exp(itX)$, so that $\exp_{iL_0}$ is the subgroup of rotations of $S^1$. Denoting by $\widetilde{\Diff^+}(S^1)$ the universal cover of ${\Diff^+}(S^1)$, we write $\widetilde{\exp}_X: \R\to \widetilde{\Diff^+}(S^1)$ for the one-parameter subgroup of $\widetilde{\Diff^+}(S^1)$ lifting $\exp_X$, and $\widetilde{\exp}(X):=\widetilde{\exp}_X(1)$.

It can be shown that the projective unitary action of $\Diff^+(S^1)$ on the vacuum Hilbert space $H_0$ induces a unitary projective action of the universal cover $\widetilde{\Diff^+}(S^1)$ of $\Diff^+(S^1)$ on any representation $H\in \Rep(\A)$, see \cite[Thm. 2.2]{Gui21}. Note that $\widetilde{\Diff^+}(S^1)$ contains the universal covering space $\widetilde{\Mob}$ of $\Mob$, which is generated by $\widetilde{\exp}(iX)$ where $X = \overline{a_{1}}L_{-1} + a_0L_0 + a_1 L_1\in\Vect_{\C}(S^1)$ with $a_1\in\C$ and $a_0\in\R$. The unitary projective action of $\widetilde{\Diff^+}(S^1)$ on a representation $H\in \Rep(\A)$ restricts to a unitary projective representation of $\widetilde{\Mob}$. By \cite{MR58601}, such a projective representation lifts to a unique unitary representation of $\widetilde{\Mob}$ on $H$. Given $X$ as above, we write $e^{iX}$ for the action of $\widetilde{\exp}(iX)\in \widetilde{\Mob}$ on $H$. We define the balance on $H\in \Rep(\A)$ to be 
\[
\theta_H:=e^{-2\pi i L_0}: H\xrightarrow{\cong }H.
\]
It can be shown that $e^{-2\pi i L_0}$ is indeed an isomorphism of $\A$-representations, and that it produces a balance for the braided $\mathrm{W}^*$-tensor category $\Rep(A)$.

\begin{theorem}(\cite{Gui21, balanceRepA})
The $\mathrm{W}^*$-category $\Rep(\A)$, equipped with the Connes fusion $\boxtimes$, the braiding $\mathbb{B}$ and the balance $\theta$ is a balanced $\mathrm{W}^*$-tensor category.
\end{theorem}
The braided tensor structure on $\Rep(\A)$ as described in this section is produced in \cite[Sec. 2]{Gui21}, and the balance is constructed in \cite{balanceRepA}.

\subsection{The Heisenberg conformal net}
\label{Sec: HeisDef}
In the current section, we introduce the Heisenberg conformal net, which is the main object of study of this paper. We follow \cite{MR2981815}. Let $L\R:= C^\infty(\R /\Z, \R)$ denote the loop group of $\R$, equipped with point-wise addition. Throughout the paper, we identify $\R/\Z\cong S^1$ via the usual exponential map $t\mapsto e^{2\pi i t}$. Since every loop $f\in L\R$ is periodic, we can write it as a Fourier series
\[
f(\theta) =\sum\limits_{k\in \Z} \hat{f}_ke^{2\pi ik\theta}, \hspace{1cm} \hat{f}_k = \int_0^{1}e^{-2\pi ik\theta}f(\theta){d\theta}
\]
with coefficients in $\C$. We can introduce a semi-norm
\begin{equation}\label{eq: seminorm}
||f||^2 = 2\pi \sum\limits_{k= 1}^\infty k\cdot ||\hat{f}_k||^2
\end{equation}
and a complex structure $\mathcal{J}(\hat{f}_k) = -i\text{sgn}(k)\hat{f}_k$ to obtain the Hilbert space 
\(
\overline{L\R/\R},
\)
where we include $\R\hookrightarrow L\R$ as the constant loops and complete with respect to the semi-norm \eqref{eq: seminorm}. The scalar product $\langle -,-\rangle$ on $\overline{L\R/\R}$ is obtained by polarization. The Hilbert space $H_0$ on which the Heisenberg net is defined is the Fock space of $\overline{L\R/\R}$, namely
\[
H_0:=e^{\overline{L\R/\R}}:=\bigoplus\limits_{n = 0}^\infty \text{Sym}^n(\overline{L\R/\R}).
\]
The set of vectors
\[
\bigg\{e^h:=\bigoplus_{n = 0}^\infty \frac{1}{\sqrt{n!}}h^{\otimes n}\bigg\}_{h\in L\R/\R}
\]
spans a dense subset of the vector space $e^{\overline{L\R/\R}}$ and it holds that $\langle e^h, e^f\rangle = e^{\langle h, f\rangle}$. The group $\Mob$ acts on $\overline{L\R/\R}$ by continuously extending
\[
u_\varphi: f\mapsto f\circ \varphi^{-1}
\]
for $\varphi\in \Mob$ and $f\in L\R/\R$. This induces a representation $U$ of $\Mob$ on the Fock space $H_0$
\[
U(\varphi)(e^h) = e^{u_\varphi h}
\]
for $\varphi\in\Mob$ and $h\in \overline{L\R/\R}$. In addition, given $f\in L\R$, we define the Weyl operator $W(f)\in U(H_0)$ by
\[
W(f)(e^h) := e^{-\frac{1}{2}||f||^2-\langle f, h\rangle} e^{f+h}
\]
for all $h\in L\R/\R$. Note that $W(f)W(g) = e^{-i\omega(f,g)}W(f+g) = e^{-2i\omega(f,g)}W(g)W(f)$, where 
\begin{equation}\label{eq: CocycleHeis}
\omega(f,g):= \text{Im}\langle f, g\rangle = \frac{1}{2}\int_0^{1} f(\theta) g'(\theta) \mathrm{d}\theta =: \frac{1}{2}\int fg'
\end{equation}
is the symplectic form coming from the scalar product. In addition, the adjoint action of $U(\varphi)$ on the Weyl operator $W(f)$, for $\varphi\in \Mob$ and $f\in L\R$ is given by
\[
U(\varphi)W(f) U(\varphi)^* = W(u_\varphi f).
\]
Given an interval $I\in\Jcal$, we write $L_I\R:=\{f\in L\R\ |\ \text{supp}(f)\subset I\}$. The family of von Neumann algebras
\[
\Heis(I):=\{W(f)\ |\ f\in L_I\R\}''\subset B(H_0),
\]
together with the representation $U$ and the vacuum vector $\Omega:=e^0$, form a Möbius covariant net \cite{GLW98}. By \cite[Thm. 9.3.1]{LG}, the representation $U$ of $\Mob$ on $H_0$ can be extended to a projective unitary representation of $\Diff^+(S^1)$ which we continue denoting by $U$. This makes the tuple $(H_0, U, \Heis, e^0)$ into a conformal net.

\begin{definition}
The Heisenberg conformal net is the conformal net $(H_0, \Heis, U, e^0)$
\end{definition}

\subsection{The crossed balanced $\otimes$-category of twisted representations of a conformal net}\label{sec: CrossedBalancedRepGA}
Let $\A = (\A, H_0, U, \Omega)$ be a conformal net. An automorphism of $\A$ is a von Neumann algebra automorphism $\varphi_I: \A(I)\to\A(I)$ for every $I\in\Jcal$ such that for every inclusion $J\subset  I$ of intervals in $\Jcal$, the following diagram commutes
\[\begin{tikzcd}
{\A(I)} & {\A(I)} \\
{\A(J)} & {\A(J)}.
\arrow["{\varphi_I}", from=1-1, to=1-2]
\arrow[hook, from=2-1, to=1-1]
\arrow["{\varphi_J}"', from=2-1, to=2-2]
\arrow[hook, from=2-2, to=1-2]
\end{tikzcd}\]
Automorphisms are furthermore required to be equipped with a unitary $V_\varphi\in U(H_0)$ such that, for every $I\in \Jcal$, it holds that
\[
\pi_{0,I}\circ\varphi(-) = V_\varphi\circ \pi_{0,I}(-)\circ V_\varphi^*,
\]
and $V_\varphi(\Omega) = \Omega$. We denote by $\Aut(\A)$ the group of automorphisms of $\A$.

\begin{definition}
Let $G$ be a discrete group. An \emph{action of $G$ on a conformal net $\A$} is a group homomorphism $G\to \Aut(\A)$.
\end{definition}

\begin{example}
The Heisenberg conformal net admits a canonical $\Z/2$-action as follows. The group $\Z/2 = \{1,\sigma\}$ acts on $\R$ by $x\xmapsto{\sigma} -x$, and hence it also acts on the loop group $L\R$ by $f\xmapsto{\sigma} -f$, where $(-f)(t):= -f(t)$. This action preserves the subspaces $L_I\R\subset L\R$ for every interval $I\in\Jcal$. We denote by $V_{\sigma}: H_0\to H_0$ the unitary
\[
V_{\sigma}(e^h) = e^{-h},
\]
which preserves the vacuum vector $e^0$ and satisfies that
\[
V_{\sigma}\circ W(f)\circ V_{\sigma}^*(e^h) = e^{-\frac{1}{2}||f||^2-\langle f,-h\rangle}e^{-f+h} = W(-f)(e^h)
\]
for all $f,h\in L\R$. This implies that the action $\sigma$ extends continuously to automorphisms $\sigma_I$ of the von Neumann algebras $\Heis(I)$. Hence, we obtain compatible actions of $\Z/2$ on each of the von Neumann algebras $\Heis(I)$. This produces an action $\Z/2\to \Aut(\Heis)$.
\end{example}

Automorphisms of a conformal net $\A$ allow us to define twisted representations of $\A$. Let us fix, once and for all, the point $\mathrm{p}:=1\in S^1$.

\begin{definition}\label{def: TwistedRep}
Let $\varphi\in\Aut(\A)$ be an automorphism of $\A$. A \emph{$\varphi$-twisted representation of $\A$} consists of a Hilbert space $H$ with a collection of $*$-actions $\pi^\varphi_I:\A(I)\to B(H)$ indexed by $I\in\Jcal$, such that for every pair of intervals $I,J\in\Jcal$ with $J\subset I$, 
\[\begin{tikzcd}
{\A(I)} & {\A(I)} & {B(H)} \\
& {\A(J)}
\arrow["{\varphi^{-1}_I}", from=1-1, to=1-2]
\arrow["{\pi^\varphi_I}", from=1-2, to=1-3]
\arrow[hook, from=2-2, to=1-1]
\arrow["{\pi^\varphi_J}"', from=2-2, to=1-3]
\end{tikzcd}\]
commutes if $\mathrm{p}\in \mathrm{cl}({I})$, $\mathrm{p}\notin{\mathrm{cl}({J})}$, and $J$ is to the right of $\mathrm{p}$ ($\mathrm{cl}(L)$ denotes the closure of an interval $L$); and
\[\begin{tikzcd}
{\A(I)} & {B(H)} \\
{\A(J)}
\arrow["{\pi^\varphi_I}", from=1-1, to=1-2]
\arrow[hook, from=2-1, to=1-1]
\arrow["{\pi^\varphi_J}"', from=2-1, to=1-2]
\end{tikzcd}\]
commutes otherwise. We write $\Rep^\varphi(\A)$ for the category of $\varphi$-twisted representations of $\A$, where morphisms are bounded linear operators commuting with the actions of all the von Neumann algebras $\A(I)$. We define the $\mathrm{W}^*$-category $\Rep^{\Aut(\A)}(\A) :=\bigoplus\limits_{\varphi\in \Aut(\A)}\Rep^\varphi(\A)$, whose objects are countable orthogonal direct sums of twisted $\A$-representations.
\end{definition}
\begin{remark}
Note that, although $\Aut(\A)$ may be naturally a topological group, we are treating it as a discrete group, and an object of $\Rep^{\Aut(\A)}(\A)$ decomposes as a direct sum of objects living in finitely many of the subcategories $\Rep^\varphi(\A)$.
\end{remark}

The condition in Definition \ref{def: TwistedRep} can also be formulated as follows. Let $q: \R\to S^1$ be the map $t\mapsto e^{2\pi i t}$. Given an interval $I\in \Jcal$, we write $\hat{I}\subset \R_{>0}$ for the lift of $I$ to $\R_{> 0}$ as close to zero as possible so that $\mathrm{cl}(\hat{I})\subset\R_{>0}$. We call an inclusion $J\subset I$ in $\Jcal$ \emph{standard} if $\hat{J}\subset \hat{I}$, and \emph{special} if it is not standard. For an inclusion $J\subset I$, we write $\delta^{\varphi^{-1}}_{J\subset I}: \A(J)\to \A(I)$ for the inclusion $\A(J)\hookrightarrow \A(I)$ if $J\subset I$ is standard, and for $\A(J)\hookrightarrow\A(I)\xrightarrow{\varphi^{-1}_I}\A(I)$ if $J\subset I$ is special. Then, the condition in Definition \ref{def: TwistedRep} is equivalent to the diagram
\[\begin{tikzcd}
	{\A(I)} & {B(H)} \\
	{\A(J)}
	\arrow["{\pi^H_I}", from=1-1, to=1-2]
	\arrow["{\delta_{J\subset I}^{\varphi^{-1}}}", hook, from=2-1, to=1-1]
	\arrow["{\pi^H_J}"', from=2-1, to=1-2]
\end{tikzcd}\]
commuting for all inclusions $J\subset I$ of intervals in $\Jcal$. 

Fix a discrete group $G$ acting on $\A$, and write $\Rep^{G}(\A)$ for the $G$-graded category $\bigoplus\limits_{g\in G}\Rep^g(\A)$, where we denote by $g$ both an element of $G$ and the automorphism of $\A$ it induces. It is known that the category $\Rep^{G}(\A)$ admits the structure of a $G$-crossed braiding \cite{muger05}. In particular, given $g\in G$, the action $T_g:\Rep^G(\A)\to \Rep^G(\A)$ of $g$ on $\Rep^G(\A)$ is given by sending an $h$-twisted representation $(H^h, \pi^H)\in \Rep^h(\A)$ to the $ghg^{-1}$-twisted representation of $\A$ on $H^h$ given by $\pi^H\circ g^{-1}$. In \cite[Thm.~3.31]{GcrossedbraidedRep}, using techniques from \cite{Gui21}, we endowed the $G$-crossed braided $\mathrm{W}^*$-tensor category $\Rep^{G}(\A)$ with the structure of a $G$-crossed balance.

\begin{theorem}\label{Thm: GcrossedBalancedRepGA}
Let $\A$ be a conformal net with an action of a discrete group $G$. The $G$-crossed braided $\mathrm{W}^*$-tensor category $\Rep^{G}(\A)$ of $G$-twisted representations of $\A$ admits a natural $G$-crossed balance.
\end{theorem}

In order to produce the $G$-crossed balance on $\Rep^{G}(\A)$, in \cite{GcrossedbraidedRep} we generalize the constructions sketched in Section \ref{Sec: ConfNets} to twisted representations. In particular, given an element $g\in G$, we show that any $g$-twisted representation $H^g\in \Rep^g(\A)$ admits a projective unitary action of $\widetilde{\Diff^+}(S^1)$ providing an honest unitary action of $\widetilde{\Mob}$ (see \cite[Sec. 3.5]{GcrossedbraidedRep}), and we define the $G$-crossed balance $\theta_H$ on $H^g$ by
\[
e^{-2\pi i L_0}: H^g\to T_g(H^g),
\]
the action of $\widetilde{\exp}(-2\pi i L_0)$ on $H^g$, see \cite[Sec. 3.6]{GcrossedbraidedRep}.

We recall the construction of the $G$-crossed braided tensor structure on $\Rep^G(\A)$ from \cite{muger05}, as we will use it later. Müger works in the language of localized endomorphisms of $\A$ as opposed to representations of $\A$, and hence he produces the $G$-crossed braided structure on a category equivalent but not equal to $\Rep^{G}(\A)$. Let us describe this category first.

Let $\Jcal_\mathrm{p}:=\{I\in \Jcal\ |\ \mathrm{p}\notin \mathrm{cl}({I})\}$ be the collection of intervals of $S^1$ not containing $\mathrm{p}$ in their closure. We write 
\[
\A_\infty := \bigcup\limits_{I\in \Jcal_\mathrm{p}}\A(I)\subset B(H_0)
\]
for the $*$-subalgebra of $B(H_0)$ given by the union of all the $*$-algebras $\A(I)$, for $I\in \Jcal_\mathrm{p}$.

\begin{definition}
Let $\End\,\A_\infty$ be the $\C$-linear strict tensor category whose objects are unital $*$-algebra homomorphisms from $\A_\infty$ to itself, with    
\begin{align*}
\Hom(\rho,\sigma) & = \{s\in\A_\infty\ |\ s\rho(x) = \sigma(x)s\ \ \forall x\in \A_\infty\}\\
t\circ s &= ts,\text{      for $s\in \Hom(\rho,\sigma)$ and $t\in\Hom(\sigma,\eta)$}\\
\rho\otimes\sigma &= \rho\circ\sigma\\
s\otimes t & = s\rho(t) = \rho'(t)s,\text{      for $s\in \Hom(\rho,\rho')$ and $t\in\Hom(\sigma,\sigma')$}
\end{align*}
where $\rho,\rho',\sigma,\sigma'\in \End\,\A_\infty$. The unit is given by $\id_{\A_\infty}$.
\end{definition}

Fix $I_0\in \Jcal_{\mathrm{p}}$ and let $\varphi\in \Aut(\A)$ be an automorphism. We say an object $\rho\in \End\,\A_\infty$ is $\varphi$-localized in $I_0$ if
\[
\begin{array}{clc}
\rho(x) &= x \hspace{1cm} & \forall J\in \Jcal_{\mathrm{p}}\text{ such that $J\subset(\mathrm{p},\partial_-I_0)$ and all $x\in \A(J)$},\\
\rho(x) &= \varphi(x) \hspace{1cm} & \forall J\in \Jcal_{\mathrm{p}}\text{ such that $J\subset(\partial_+I_0,\mathrm{p})$ and all $x\in \A(J)$}.
\end{array}
\]
Here, $\partial_- I_0$, $\partial_+I_0$, $(\mathrm{p}, \partial_-I_0)$ and $(\partial_+I_0, \mathrm{p})$ are taken with the usual orientation of $S^1$, meaning in particular that $\partial_+I_0$ is counter-clockwise to $\partial_-I_0$ along $I_0$. By diffeomorphism covariance of $\A$, every $\varphi$-localized endomorphism $\rho\in\End\,\A_\infty$ is transportable, meaning that for every $J\in\Jcal_{\mathrm{p}}$, there exists $\rho'\in\End\,\A_\infty$ which is $\varphi$-localized in $J$ and satisfies that $\rho\cong \rho'$ unitarily. Given an automorphism $\varphi\in \Aut(\A)$, we write $\varphi-\Loc_{I_0}(\A)$ for the $\mathrm{W}^*$-subcategory of $\End\,\A_\infty$ whose objects are $\varphi$-localized endomorphisms in $I_0$ and whose morphisms between objects $\rho$ and $\sigma$ are given by $s\in \A(I_0)$ such that $s\rho(x) = \sigma(x)s$ for all $x\in \A_\infty$.

Let $G$ be a discrete group acting on $\A$. For every $g\in G$, we also denote by $g\in \Aut(\A)$ the automorphism of $\A$ induced by $g$. We say an endomorphism $\rho\in \End\,\A_\infty$ is $G$-localized in $I_0$ if it is $g$-localized in $I_0$ for some $g\in G$.

\begin{definition}\label{def: CategoryOfAut(A)LocEndomorphisms}
We write $G-\Loc_{I_0}\A: = \bigoplus\limits_{g\in G} g-\Loc_{I_0}(\A)$, whose objects are countable orthogonal direct sums of endomorphisms in $\End\,\A_\infty$ which are $G$-localized in $I_0$.
\end{definition}

Given $g,h\in G$ and $\rho\in g-\Loc_{I_0}(\A)$ and $\sigma\in h-\Loc_{I_0}(\A)$ localized endomorphisms, their tensor product
\[
\rho\otimes \sigma = \rho\circ\sigma,
\]
is $gh$-localized. Therefore, the tensor structure of $\End\,\A_\infty$ restricts to a tensor structure on $G-\Loc_{I_0}(\A)$ compatible with the $G$-grading. We let $G$ act on $G-\Loc_{I_0}\A$ by $\gamma_g(\rho) = g\circ\rho\circ g^{-1}$ and $\gamma_g(s) = g(s)$ for $g\in G$, $\rho,\sigma\in\End\,\A_\infty$ and $s\in\Hom(\rho,\sigma)$. The $G$-crossed braiding is defined as follows. Fix $g,h\in G$ and let $\rho,\sigma \in G-\Loc_{I_0}\A$ be endomorphisms with $\rho$ $g$-localized in $I\in \Jcal_{\mathrm{p}}$ and $\sigma$ $h$-localized in $J\in \Jcal_{\mathrm{p}}$, where $I, J\subset I_0$. Let $I', J'\in \Jcal_{\mathrm{p}}$ such that $(\partial_-I', \partial_+J')$ is covered by $I_0$, taking the usual orientation of $S^1$. By transportability, there exist $\rho'$ and $\sigma'$ localized in $I'$ and $J'$ respectively, and unitaries $u\in \Hom(\rho,\rho')$ and $v\in\Hom(\sigma,\sigma')$. By \cite[Lemm. 2.14]{muger05}, we have $\rho'\otimes\sigma' = \gamma_g(\sigma')\otimes\rho'$, and hence we can define the braiding of $\rho$ and $\sigma$ as the composite
\[
c_{\rho,\sigma}: \rho\otimes\sigma \xrightarrow{u\otimes v} \rho'\otimes \sigma'
= \gamma_g(\sigma')\otimes\rho'\xrightarrow{\gamma_g(v^*)\otimes u^*} \gamma_g(\sigma)\otimes\rho. \]
It can be shown that $c_{\rho,\sigma}$ is independent of the choices of $I',J',\rho',\sigma'$ and that it defines a $G$-crossed braiding on $G-\Loc_{I_0}\A$, see \cite[Prop. 2.17]{muger05}.

\begin{theorem}
Let $G$ be a discrete group acting on a conformal net $\A$. The category $G-\Loc_{I_0}(\A)$ of $G$-localized endomorphisms of $\A$ admits a canonical structure of a $G$-crossed braided tensor category. 
\end{theorem}

\begin{remark}
The braided $\mathrm{W}^*$-tensor full subcategory $\{\id\}-\Loc_{I_0}(\A)$ of $G-\Loc_{I_0}(\A)$ on the $\id$-localized endomorphisms is also known as the DHR category of $\A$, after Doplicher-Haag-Roberts \cite{MR258394}. We will denote it here by $\Loc_{I_0}(\A)$.
\end{remark}

Given $g\in G$ and $\rho\in g-\Loc_{I_0}(\A)$ a $g$-localized endomorphism, we can produce a $g$-twisted representation of $\A$ on $H_0$ by letting, for all $I\in \Jcal_\mathrm{p}$ and $x\in \A(I)$, the element $x$ act on $H_0$~by
\[
\pi_{0, I}(\rho(x)):H_0\to H_0.
\]
Since $\rho$ is $g$-localized, these actions extend to a $g$-twisted representation of $\A$ on $H_0$. We abuse the notation to write $(H_0, \pi_0\circ \rho)$ for this representation. The assignment $\rho \in  g-\Loc_{I_0}(\A) \mapsto (H_0, \pi_0\circ \rho)\in \Rep^g(\A)$ is extended to morphisms by sending an element $s\in \Hom_{G-\Loc_{I_0}(\A)}(\rho, \rho')$ to $\pi_0(s):H_0\to H_0$. We obtain a $\mathrm{W}^*$-functor $\mathfrak{E}: G-\Loc_{I_0}(\A)\to \Rep^{G}(\A)$, which can be further upgraded to a $\mathrm{W}^*$-tensor functor $(\mathfrak{E}, \Psi): G-\Loc_{I_0}(\A)\to \Rep^{G}(\A)$. The following result is the content of \cite[Thm. 4.9]{GcrossedbraidedRep}, where the $G$-crossed braided $\mathrm{W}^*$-tensor structure on $\Rep^G(\A)$ is that in Theorem \ref{Thm: GcrossedBalancedRepGA}.

\begin{theorem}
The $\mathrm{W}^*$-tensor functor $(\mathfrak{E}, \Psi): G-\Loc_{I_0}(\A)\to \Rep^{G}(\A)$ is an equivalence of $\mathrm{W}^*$-tensor categories. In addition, it can be upgraded to an equivalence of $G$-crossed braided $\mathrm{W}^*$-tensor categories.
\end{theorem}

\begin{remark}
Whenever there is no possible confusion, we will also denote by $(\mathfrak{E}, \Psi)$ the equivalence of braided $\mathrm{W}^*$-tensor categories $\Loc_{I_0}(\A)\cong \Rep(\A)$ obtained by restriction to the identity piece of $G-\Loc_{I_0}(\A)$.
\end{remark}

The action of $G$ on $\A$ defines a fixed-points conformal net $\A^G$ as follows. Recall that every $g\in G$ comes equipped, by definition, with a unitary $V_g\in U(H_0)$ implementing the action of $g$ on $\A$. The vacuum Hilbert space of $\A^G$ is $H_0^G : = \{\xi\in H_0\ |\ V_g\xi = \xi\ \text{ for all $g\in G$}\}$, with vacuum vector $\Omega\in H_0^G$. Given an interval $I\in \Jcal$, we define
\[
\A^G(I) : = \A(I)^G = \{x\in \A(I)\ |\ gx = x\ \text{ for all $g\in G$}\}.
\]
The projective action of $\Diff^+(S^1)$ on $H_0$ restricts to a projective action on $H_0^G$, which we continue denoting by $U$, see \cite[Thm. 6.1.9]{weiner2007conformalcovariancerelatedproperties} and \cite[Rk. 3.14]{GcrossedbraidedRep}.

\begin{definition}
The fixed points conformal net of a conformal net $(\A, H_0, U, \Omega)$ under the action of a group $G$ is the tuple $(\A^G, H_0^G, U|_{H_0^G}, \Omega)$.
\end{definition}

We showed in \cite{fixedpoints} that the category of representations of $\A^G$ is controlled by the category $\Rep^G(\A)$ of twisted representations of $\A$. Let $(\Rep^G(\A))^G$ denote the $G$-equivariantization of the $G$-crossed balanced $\mathrm{W}^*$-tensor category $\Rep^G(\A)$. An object in $(\Rep^G(\A))^G$ consists of an object $H\in \Rep^G(\A)$ together with, for every $g\in G$, a unitary $U_g: T_g(H)\to H$ satisfying certain coherence conditions. Recall that the underlying Hilbert space of $T_g(H)$ is $H$, and therefore we obtain a unitary action of $G$ on $H$ by $g\mapsto U_g$. We write $H^G$ for the fixed points of $H$ under this action. The $G$-twisted action of $\A$ on $H$ restricts to an honest action of $\A^G$ on the fixed-points Hilbert space $H^G$. A morphism in $(\Rep^G(\A))^G$ restricts to a morphism between the fixed-points representations of the domain and the target. We write
\[
\mathfrak{R}: (\Rep^G(\A))^G\to \Rep(\A^G)
\]
for the functor sending a $G$-equivariant, $G$-twisted representation $H$ of $\A$ to the restricted honest representation of $\A^G$ on the fixed-points Hilbert space $H^G$.

\begin{theorem}(\cite[Thm. 4.19]{fixedpoints})
Let $\A$ be a conformal net and let $G$ be a finite group acting faithfully on $\A$. Then, the $\mathrm{W}^*$-functor
\[
\mathfrak{R}:(\Rep^G(\A))^G\xrightarrow{} \Rep(\A^G)
\]
can be canonically upgraded to an equivalence of balanced $\mathrm{W}^*$-tensor categories.
\end{theorem}

\section{The category of representations of the Heisenberg conformal net}
\label{sec: TheCatOfReps}

In this section we characterize the underlying $\mathrm{W}^*$-category of $\Rep(\Heis)$ and its fusion rules. We show that they both agree with those of $\Hilb\,\R$. We do not show yet an equivalence of $\mathrm{W}^*$-tensor categories $\Rep(\Heis)\cong \Hilb\,\R$ as we do not argue compatibility with the associators. This is done in Section \ref{sec: RepZ/2AsTY}.

We first provide an equivalence of $\mathrm{W}^*$-categories $\Rep(\Heis)\cong \Rep(C_0(\R))$, see Theorem \ref{thm: RepHeisIsRepR}. We make use of two auxiliary categories: the category $\Rep(\R)$ of unitary representations of $\R$ and the category $\Rep^{pe}(\widetilde{L\R})$ of unitary positive energy representations of a central $U(1)$-extension of the loop group of $\R$ (a positive representation of $\widetilde{L\R}$ is a representation extending to $U(1)\ltimes \widetilde{L\R}$ where the action of $U(1)$ has positive spectrum). We then construct the following collection of functors
\[\begin{tikzcd}
	{\Rep(C_0(\R))} & {\Rep(\R)} & {\Rep^{pe}(\widetilde{L\R})} & {\Rep(\Heis)}
	\arrow["\cong", from=1-1, to=1-2]
	\arrow["{\mathfrak{L}}"', curve={height=18pt}, from=1-1, to=1-3]
	\arrow["{\mathfrak{M}}", shift left=2, curve={height=-24pt}, from=1-1, to=1-4]
	\arrow["\cong", from=1-2, to=1-3]
	\arrow["{\mathfrak{M}_0}", shift left, from=1-3, to=1-4]
	\arrow["{\mathfrak{R}}", shift left, from=1-4, to=1-3]
\end{tikzcd}\]
and show that they are all equivalences. The left-most equivalence is well-known and comes from the fact that $\Rep(\R)\cong \Rep(C_0(\hat{\R}))\cong \Rep(C_0(\R))$, where $\hat{\R}$ denotes the Pontryagin dual of $\R$. The middle equivalence comes from a good understanding of the positive energy representation theory of Heisenberg groups \cite{LG}, see Proposition \ref{prop: FunctorRepC0RtoReppeLR}. The right-most equivalence uses heavily the fact that representations of conformal nets are automatically of positive energy \cite{weiner}, see Proposition \ref{prop: R0Factors}. We then show in Proposition \ref{prop: TensoratorForM} that the functor
\[
\mathfrak{M}: \Rep(C_0(\R))\to \Rep(\Heis)
\]
defined as composition $\Rep(C_0(\R))\xrightarrow{\cong} \Rep(\R)\xrightarrow{\cong} \Rep^{pe}(\widetilde{L\R})\xrightarrow[\cong]{\mathfrak{M}_0} \Rep(\Heis)$ can be equipped with a unitary natural isomorphism (not necessarily compatible with the associators)
\[\begin{tikzcd}
	{\Hilb\,\R\times\Hilb\,\R} && {\Rep(\Heis)\times\Rep(\Heis)} \\
	{\Hilb\,\R} && {\Rep(\Heis)}
	\arrow["{\mathfrak{M}\times\mathfrak{M}}", from=1-1, to=1-3]
	\arrow["\otimes"', from=1-1, to=2-1]
	\arrow[between={0.2}{0.8}, Rightarrow, from=1-3, to=2-1]
	\arrow["\boxtimes", from=1-3, to=2-3]
	\arrow["{\mathfrak{M}}"', from=2-1, to=2-3]
\end{tikzcd}\]
when $\Rep(C_0(\R))$ is equipped with the tensor structure of $\Hilb\,\R$. We call such a unitary natural isomorphism a tensorator, although it might not be compatible with associators.

\subsection{Constructing irreducible representations}\label{Sec: ConstructingIrreps}

Let us first construct a family of irreducible representations of $\Heis$. We will later show these exhaust all irreducible representations up to isomorphism.

Fix a loop $l\in L\R$ and an interval $I\in\Jcal$. Let $\langle W(f)\rangle_{f\in L_I\R}$ be the $*$-algebra generated by $\{W(f)\ |\ f\in L_I\R\}$ in $B(H_0)$. We define the $*$-algebra homomorphism
\begin{equation}\label{eq: AutomorphismByl}
\begin{array}{cccc}
\rho_{l, I}: & \langle W(f)\rangle_{f\in L_I\R} &\to &\langle W(f)\rangle_{f\in L_I\R} \\
& W(f) &\mapsto & e^{i\int lf}W(f),
\end{array}
\end{equation}
where $\int lf:= \int_0^{1}l(\theta)f(\theta)\mathrm{d}\theta$. We will show that the $\Jcal$-indexed family of $*$-algebra homomorphisms $\rho_l=\{\rho_{l,I}\}_{I\in \Jcal}$ extends by continuity to a compatible family of automorphisms of the von Neumann algebras $\Heis(I)$, for all $I\in\Jcal$. We say that such a family of $*$-algebra homomorphisms $\rho_l$ is unitarily implemented on $H_0$ if there exists a unitary $V\in U(H_0)$ such that $\rho_{I, l}(W(f)) = VW(f)V^*$ for all $I\in \Jcal$ and $f\in L_I\R$. We define the charge of a loop $l\in L\R$ to be 
\[
q_l := \int_0^{1} l(\theta) {\mathrm{d}\theta} =:\int l.
\]

\begin{lemma}\label{lemm: IsInner}
The family of $*$-automorphisms $\rho_{l}$ is unitarily implemented on $H_0$ if and only if $q_l = 0$.
\end{lemma}
\begin{proof}
Assume that $q_l = 0$ and let $L$ be a primitive of $l$ on $(0,1)$. Since $\int l = 0$, we obtain that $L$ can be extended by continuity to $L\in L\R$. Given $f\in L_I\R$, we have
\begin{align*}
W(L) W(f) W(-L) = W(L)\big(e^{-i\omega(f, -L)}W(f-L)\big) = e^{-i\omega(f, -L)}e^{-i\omega(L, f-L)}W(f),
\end{align*}
and 
\[
-\omega(f, -L)-\omega(L, f-L) = \frac{1}{2}\int fl-\frac{1}{2}\int Lf' = \int lf.
\]
Hence, $\rho_l(W(f)) = W(L)W(f)W(L)^*$.

Assume that, given $l\in L\R$, there is a unitary $V\in U(e^{\overline{L\R/\R}})$ such that 
\[
e^{i\int lf}W(f) = VW(f)V^*
\]
for every $I\in\Jcal$ and $f\in L_I\R$. Let $I_1,I_2,\ldots, I_n\in\Jcal$ be a finite family of intervals covering $S^1$, and pick a loop $\varphi_j\in L_{I_j}\R$ for every $j = 1,\ldots,n$ such that $1 = \sum_{j = 1}^n\varphi_j$. For any $t\in \R$, let us write $P_t:=W(t\varphi_1)\cdots W(t\varphi_n) = c_t\cdot W(t\sum_{i = 1}^n \varphi_i) = c_t\cdot W(t\cdot 1)$, where $c_t\in U(1)$ can be computed using the Weyl relations above Equation \eqref{eq: CocycleHeis}. Since $t\cdot 1$ is a constant loop, $W(t\cdot 1) = 1$ is the identity on $e^{\overline{L\R/\R}}$. Therefore, $P_t = c_t\cdot 1$ acts by a scalar, and $VP_tV^* = P_t$. Now, we compute
\[
P_t = VP_tV^* = \prod\limits_{j = 1}^n e^{it\int \varphi_j l}W(t\varphi_j) = e^{it\int(\sum_j \varphi_j) l}P_t = e^{it\int l}P_t.
\]
Since $P_t \neq 0$, we get that $e^{it\int l} =1 $ for all $t\in \R$, implying that $\int l = 0$.
\end{proof}

\begin{remark}\label{rk: IntertwineActionsSameCharge}
From the proof of Lemma \ref{lemm: IsInner}, we see that if $l_1,l_2\in L\R$ are loops with the same charge, and $L_{21}$ is a primitive of $l_2-l_1$, it holds that $W(L_{21})$ intertwines the actions $\rho_{l_{1},I}(W(f))$ and $\rho_{l_2,I}(W(f))$ for all $I\in\Jcal$ and $f\in L_I\R$.
\end{remark}

By Lemma \ref{lemm: IsInner}, if $l\in L\R$ is a loop of charge $q_l = 0$, the family of $*$-homomorphisms $\rho_l=\{\rho_{l,I}(-) = W(L)\circ-\circ W(L)^*\}_{I\in\Jcal}$ extends by continuity to a compatible family of von Neumann algebra automorphisms $\Heis(I)\to \Heis(I)$, which we continue denoting by $\rho_l = \{\rho_{l,I}: \Heis(I)\to \Heis(I)\}_{I\in \Jcal}$. Assume that the charge of $l$ is nonzero and fix $I\in \Jcal$. Then, we can pick another loop $m\in L\R$ such that $m|_{I} = l|_I$ and $q_{m} = 0$. Since $q_{m} = 0$, the induced automorphism $\rho_{m, I}: \langle W(f)\rangle_{f\in L_I\R}\to \langle W(f)\rangle_{f\in L_I\R}$ extends by continuity to an automorphism $\rho_{m, I}: \Heis(I)\to \Heis(I)$. Since $m|_{I} = l|_I$, the $*$-algebra homomorphisms $\rho_{l,I}$ and $\rho_{m,I}$ agree on $\langle W(f)\rangle_{f\in L_I\R}$. Hence, $\rho_{l,I}$ also extends by continuity to an automorphism $\rho_{l,I}:\Heis(I)\to \Heis(I)$. The extension $\rho_{l,I}:\Heis(I)\to \Heis(I)$ is independent of the choice of $m$. Since each $\langle W(f)\rangle_{f\in L_I\R}\subset \Heis(I)$ is dense and the family of automorphism $\{\rho_{l, I}: \langle W(f)\rangle_{f\in L_I\R}\to \langle W(f)\rangle_{f\in L_I\R}\}_{I\in\Jcal}$ is compatible with the inclusion of intervals, so is the extension $\{\rho_{l,I}:\Heis(I)\to \Heis(I)\}_{I\in\Jcal}$. We obtain the following result.

\begin{lemma}
Fix a loop $l\in L\R$. The family of homomorphisms $\rho_{l,I}: \langle W(f)\rangle_{f\in L_I\R}\to \langle W(f)\rangle_{f\in L_I\R}$ for $I\in\Jcal$ defined by Equation \eqref{eq: AutomorphismByl} extends by continuity to a compatible family of automorphisms $\rho_{l,I}: \Heis(I)\to \Heis(I)$ for $I\in\Jcal$.
\end{lemma}

We can make the following definition.

\begin{definition}\label{def: IrrepsOfHeis}
Let $l\in L\R$. The representation $\pi_l$ of $\Heis$ on $H_0$ is given by, for every $I\in \Jcal$ and $a\in \Heis(I)$,
\[
\pi_{l,I}(a):= \pi_0\circ \rho_{l,I}(a).
\]
\end{definition}

The representations $\pi_l$ are irreducible representations of $\Heis$ by construction. By Lemma \ref{lemm: IsInner}, two such representations $\pi_l$ and $\pi_m$ are equivalent if and only if they have the same charge. We have therefore produced one irreducible representation of $\Heis$ for every charge $q\in \R$.

\subsection{The restriction functor $\Rep(\Heis)\to \Rep^{pe}(\widetilde{L\R})$}\label{sec: restrictionFunctor}
Recall that in Section \ref{Sec: HeisDef} we have defined a 2-cocycle on $L\R$ with coefficients in $\R$ by $\omega(f,g) = \frac{1}{2}\int fg'$, see Equation \eqref{eq: CocycleHeis}. This induces a 2-cocycle on $L\R$ with values in $U(1)$ by
\[
c(f,g):= e^{-i\omega(f,g)}
\]
and provides a central extension
\[
1\to U(1)\to \widetilde{L\R}\to L\R\to 1.
\]
The same cocycle restricted to $L_I\R$ for an interval $I\in\Jcal$ provides a central extension $\widetilde{L_I\R}$, which can be equivalently defined as the pullback
\[\begin{tikzcd}
{\widetilde{L_I\R}} & {\widetilde{L\R}} \\
{L_I\R} & {L\R}.
\arrow[from=1-1, to=1-2]
\arrow[from=1-1, to=2-1]
\arrow["\lrcorner"{anchor=center, pos=0.125}, draw=none, from=1-1, to=2-2]
\arrow[from=1-2, to=2-2]
\arrow[from=2-1, to=2-2]
\end{tikzcd}\]
\begin{definition}
    A \emph{(unitary) representation of $\widetilde{L\R}$} consists of a Hilbert space $H$ together with a strongly continuous group homomorphism $\widetilde{L\R}\to U(H)$. We write $\Rep(\widetilde{L\R})$ for the category of representations of $\widetilde{L\R}$.
\end{definition}

A \emph{representation} of $\widetilde{L\R}$ shall always mean a unitary representation. We now provide~a functor
\(
\mathfrak{R}_0:\Rep(\Heis)\to \Rep(\widetilde{L\R}).
\)
 Let $(K, \{\pi^K_I\})\in \Rep(\Heis)$ be a representation of $\Heis$. For every $I\in\Jcal$, we write $U(\Heis(I))$ for the unitary elements in $\Heis(I)$. The composition
\[
L_I\R\hookrightarrow \{W(f)\ |\ f\in L_I\R\}\subset U(\Heis(I))\xrightarrow{\pi_I^K} U(K)
\]
induces a projective representation of $L_I\R$ on $K$ with cocycle $e^{-i\omega}$, that is, an honest representation of the central extension $\widetilde{L_I\R}$. These representations are compatible with the~inclusions 
\[
\widetilde{L_J\R}\hookrightarrow\widetilde{L_I\R}
\]
whenever $J\subset I$, and hence produce a representation of the colimit topological group
\[
\colim_{I\in \INT}\widetilde{L_I\R}\cong \widetilde{L\R},
\]
where the equivalence follows from \cite[Thm. 4]{Hen19}. We extend this assignment trivially at the level of morphisms to obtain the required functor 
\[
\mathfrak{R}_0: \Rep(\Heis)\to \Rep(\widetilde{L\R}),
\]
which is fully faithful, as the image of the $*$-algebra generated by $\widetilde{L_I\R}$ in $\Heis(I)$ is dense for all $I\in\Jcal$.

The next step is to characterize the essential image of the functor $\mathfrak{R}_0$, for which we follow \cite[Sec. 3.2]{Hen19}. The group $U(1)$ acts on $L\R$ by precomposition with the inverse of a rotation, and hence it also acts on the central extension $\widetilde{L\R}$. We can construct the semi-direct product $U(1)\ltimes \widetilde{L\R}$. 

\begin{definition}
A representation $(H, \pi)$ of $\widetilde{L\R}$ is said to be \emph{of positive energy} if it extends to a representation of $U(1)\ltimes \widetilde{L\R}$ for which the generator $L_0:=-i\frac{d}{dt}\Big|_{t = 0}\big(\pi(e^{2\pi it}, 1)\big)$ of the action of $U(1)$ has positive spectrum. We furthermore require that the central $U(1)$ acts in the standard way. 
\end{definition}

We now argue that the essential image of $\mathfrak{R}_0: \Rep(\Heis)\to \Rep(\widetilde{L\R})$ lands in the full subcategory of $\Rep(\widetilde{L\R})$ on the positive energy representations, which we denote by $\Rep^\pe(\widetilde{L\R})$. Given an interval $I\in\Jcal$, the subgroup $\Diff_I(S^1)\subset \Diff^+(S^1)$ of diffeomorphisms of the circle supported on $I$ acts on the group $L_I\R$, and hence on the central extension $\widetilde{L_I\R}$, by precomposing with the inverse of a diffeomorphism $\varphi\in \Diff_I(S^1)$. Hence, we can construct the semi-direct product
\[
\Diff_I(S^1)\ltimes \widetilde{L_I\R}.
\]
The group $\Diff^+(S^1)$ admits a well-known central extension by $\R$ called the Virasoro-Bott group, which we denote by $\Diff^{\R}(S^1)$ \cite{bot77, kw09, tl99}. We write $\Diff_I^\R(S^1)$ for the~pullback
\[\begin{tikzcd}
{\Diff_I^\R(S^1)} & {\Diff^{\R}(S^1)} \\
{\Diff_I(S^1)} & {\Diff^+(S^1)}.
\arrow[from=1-1, to=1-2]
\arrow[from=1-1, to=2-1]
\arrow["\lrcorner"{anchor=center, pos=0.125}, draw=none, from=1-1, to=2-2]
\arrow[from=1-2, to=2-2]
\arrow[from=2-1, to=2-2]
\end{tikzcd}\]
By the diffeomorphism covariance of $\Heis$, there are canonical homomorphisms
\[
\Diff_I^\R(S^1)\ltimes \widetilde{L_I\R}\to U(\Heis(I)).
\]
Given a representation $(K, \pi)\in \Rep(\Heis)$, composing the map above with $\pi_I: \Heis(I)\to B(K)$, we obtain a family of homomorphisms
\[
\Diff_I^\R(S^1)\ltimes \widetilde{L_I\R}\to U(K)
\]
which are compatible with the inclusions $\Diff_J^\R(S^1)\ltimes \widetilde{L_J\R}\to \Diff_I^\R(S^1)\ltimes \widetilde{L_I\R}$ whenever $J\subset I$. By \cite[Thm. 11]{Hen19}, we obtain a strongly continuous unitary action 
\[
\Diff^{\R\times \Z}(S^1)\ltimes \widetilde{L\R}\to U(K),
\]
where $\Diff^{\R\times \Z}(S^1)$ denotes the universal cover of $\Diff^{\R}(S^1)$, which is a central extension by the group $\Z$. In particular, we obtain an action of $\widetilde{U(1)}\ltimes \widetilde{L\R}$, where $\widetilde{U(1)}$ denotes the universal cover of $U(1)\subset \Diff^+(S^1)$. By \cite[Thm. 3.8]{weiner}, the implementation of the Möbius symmetry in $K$ is of positive energy, meaning that the induced representation of the universal cover $\widetilde{\Mob}\to \Mob$ satisfies that the generator of rotations of $\widetilde{\Mob}$ has positive spectrum. Hence, the generator of $L_0$ of $\widetilde{U(1)}$ on $K$ has positive spectrum. In order to argue that $K$ is a positive energy representation of $\widetilde{L\R}$, we need to modify the action of $\widetilde{U(1)}$ so that it descends to an action of $U(1)$. We proceed as in the proof of \cite[Thm. 29]{Hen19}.

We can decompose $K$ as a direct integral according to the characters of $\Z\subset \widetilde{U(1)}$
\[
K = \int_{\theta\in U(1)}K_\theta.
\]
For every $\theta\in U(1)$ we can extend the character $n\mapsto \theta^n$ of $\Z$ to a character $z\mapsto z^{\log(\theta)/2\pi i}$ of $\widetilde{U(1)}$ by taking the principal branch of the logarithm. Let $\mathbb{C}_\theta$ be the 1-dimensional representation of $\widetilde{U(1)}$ induced by the character above. Then, the representation
\[
K' := \int_{\theta\in U(1)}K_\theta\otimes \overline{\mathbb{C}_\theta}
\]
of $\widetilde{U(1)}$ descends to a representation of $U(1)$ whose generator has positive spectrum, as the spectrum of $L_0$ has only been modified by a controlled bounded amount. Since $\overline{\mathbb{C}_\theta}$ is one-dimensional, $K'\cong K$ as vector spaces. This isomorphism allows us to equip $K'$ with an action of $\widetilde{L\R}$. We need to show that the actions of $U(1)$ and $\widetilde{L\R}$ on $K'$ assemble into an action of $U(1)\ltimes\widetilde{L\R}$. 

Let $\widetilde{U(1)}_*\subset \widetilde{U(1)}$ be a countable dense subgroup that contains $\mathbb{Z}\subset \widetilde{U(1)}$, and let $U(1)_*:= \widetilde{U(1)}_
*/\Z\subset U(1)$. Pick a $\widetilde{U(1)}_*$-invariant countable dense subgroup $\widetilde{L\R}_*\subset \widetilde{L\R}$. Since $\Z\subset \widetilde{U(1)}_*\ltimes\widetilde{L\R}_*$ is central, the Hilbert spaces $K_\theta$ carry representations of $\widetilde{U(1)}_*\ltimes\widetilde{L\R}_*$ for almost all $\theta$. By construction, on almost all $K_\theta\otimes\overline{\mathbb{C}_\theta}$, the action of $\widetilde{U(1)}_*\ltimes\widetilde{L\R}_*$ descends to an action of $U(1)_*\ltimes\widetilde{L\R}_*$. The actions of $U(1)_*$ and $\widetilde{L\R}_*$ on $K'$ hence assemble into an action of $U(1)_*\ltimes \widetilde{L\R}_*$. Since $U(1)_*\ltimes \widetilde{L\R}_*\subset U(1)\ltimes \widetilde{L\R}$ is dense and the actions of $U(1)$ and $\widetilde{L\R}$ on $K'$ are strongly continuous, these two assemble into an action of $U(1)\ltimes \widetilde{L\R}$ on $K'$. Therefore, $K'$ (and hence $K$) is a positive energy representation of $\widetilde{L\R}$.

By the discussion above, we obtain the following result.
\begin{proposition}\label{prop: R0Factors}
The fully faithful $\mathrm{W}^*$-functor $\mathfrak{R}_0: \Rep(\Heis)\to \Rep(\widetilde{L\R})$ factors through a fully faithful $\mathrm{W}^*$-functor $\mathfrak{R}:\Rep(\Heis)\to \Rep^{pe}(\widetilde{L\R})$.
\end{proposition}

We claim that $\mathfrak{R}$ is an equivalence of $\mathrm{W}^*$-categories. Before providing the inverse functor to $\mathfrak{R}$, we need a better understanding of the $\mathrm{W}^*$-category $\Rep^{pe}(\widetilde{L\R})$.

\subsection{Characterizing $\Rep^{pe}(\widetilde{L\R})$} We claim that there is an equivalence of $\mathrm{W}^*$-categories $\Rep(C_0(\R))\xrightarrow{\cong}\Rep^{pe}(\widetilde{L\R})$, constructed via a composition
\[
\mathfrak{L}: \Rep(C_0(\R))\xrightarrow{\cong}\Rep(\R)\xrightarrow[\cong]{\mathfrak{L}_0}\Rep^{pe}(\widetilde{L\R})
\]
through the category $\Rep(\R)$ of unitary representations of $\R$. We first discuss the functor $\Rep(C_0(\R))\xrightarrow{\cong}\Rep(\R)$. It is given by sending an object $\int^\oplus_{t\in T}\delta_{p(t)}\mathrm{d}\upsilon(t)\in \Rep(C_0(\R))$ to the representation of $\R$ on $L^2(T,\upsilon)$ where $x\in \R$ acts by multiplication by the function $t\mapsto e^{i xp(t)}$ on $T$. This functor (which is trivial on morphisms) is well-known to produce an equivalence of categories
\begin{equation}\label{eq: C0RandR}
\Rep(C_0(\R))\xrightarrow{\cong} \Rep(\R).
\end{equation}
This equivalence comes from the isomorphism of $\mathrm{C}^*$-algebras $C_0(\R)\cong C_0(\hat{\R})\cong C^*(\R)$, where $\hat{\R}\cong \R$ is the Pontryagin dual of $\R$ and $C^*(\R)$ denotes the group $\mathrm{C}^*$-algebra of $\R$. We also use the well-known equivalence $\Rep(C^*(\R))\cong \Rep(\R)$ between the $\mathrm{W}^*$-category of representations of the group $\mathrm{C}^*$-algebra $C^*(\R)$ and the $\mathrm{W}^*$-category of unitary representations of~$\R$.

We next discuss the representation theory of $\widetilde{L\R}$. Let $V$ be the subset of $L\R$ of loops with average zero. Then, $L\R$ breaks as a product $L\R = \R\times V$ of topological groups via the map $f\mapsto (\int f, f-\int f)$. Note that $V$ is a topological vector space carrying a nondegenerate skew-symmetric bilinear~form
\[
\begin{array}{cccc}
\omega: &V\times V&\to &\R
\\ &(f, g)&\mapsto &\frac{1}{2}\int  fg' .
\end{array}
\]
This produces the Heisenberg group $\widetilde{V}$ as the $U(1)$-central extension of $V$ classified by the cocycle $c(f, g) = e^{-i \omega(f,g)} = e^{-i\frac{1}{2} \int  fg' }$, see \cite[Def. 9.5.1]{LG}. Chasing the definitions, we see that the central extension $\widetilde{L\R}$ breaks as a product $\widetilde{L\R}\cong \R\times \widetilde{V}$ by $(f,z)\mapsto (\int f, f-\int f, z)$. Therefore,
\[
\Rep(\widetilde{L\R})\cong \Rep(\R\times \widetilde{V}).
\]
Note that the action of $U(1)$ on $\widetilde{L\R}$ by rotating the domain is trivial on the factor $\R$ of $\R\times \widetilde{V}$ and only acts on $\widetilde{V}$. Hence, we find
\[
U(1)\ltimes \widetilde{L\R}\cong \R\times(U(1)\ltimes \widetilde{V}).
\]
We say that a representation of $\widetilde{V}$ is of \emph{positive energy} if it extends to a representation of $U(1)\ltimes \widetilde{V}$ such that the generator of rotations of $U(1)$ has positive spectrum. We write $\Rep^{pe}(\widetilde{V})$ for the category of positive energy representations of $\widetilde{V}$. We require that the central $U(1)$ in $\widetilde{V}$ acts in the standard way. It is clear that we obtain an equivalence of categories
\begin{equation}\label{eq: RepLRsplits}
\Rep^{pe}(\widetilde{L\R})\cong \Rep(\R)\boxtimes\Rep^{pe}(\widetilde{V}).
\end{equation}
The functor $\Rep(\R)\boxtimes\Rep^{pe}(\widetilde{V})\to \Rep^{pe}(\widetilde{L\R})$ witnessing this equivalence sends a pair $(H, \pi^H)\otimes (K, \pi^K)\in \Rep(\R)\boxtimes\Rep^{pe}(\widetilde{V})$ to the representation $\pi^{H\otimes K}$ of $\widetilde{L\R}$ on $H\otimes K$ where $(f,z)\in\widetilde{L\R}$ acts~by
\[
\pi^{H\otimes K}(f,z)(h\otimes k) = z\cdot \pi^H\Big(\int f\Big)(h)\otimes\pi^K\bigg(f-\int f\bigg)(k),
\]
for any $h\in H$ and $k\in K$.

We now construct the so-called standard irreducible positive energy representation of $\widetilde{V}$. The underlying Hilbert space of this representation is the Fock space $e^{\overline{L\R/\R}}$ constructed in Section \ref{Sec: HeisDef}. Given an element $(f,z)\in \widetilde{V}$ and $e^h\in e^{\overline{L\R/\R}}$, we define
\[
\Pi(f,z)(e^h) = z\cdot e^{-\frac{1}{2}||f||^2-\langle f,h\rangle}e^{f+h}
\]
where $\langle-,-\rangle$ and $||-||$ are those on $\overline{L\R/\R}$.
\begin{proposition}\label{prop: StandardRep}(\cite[Prop. 9.5.8 and 9.5.10]{LG})
The pair $(e^{\overline{L\R/\R}},\Pi)$ defines an irreducible positive energy representation of $\widetilde{V}$. Furthermore, it is the only irreducible positive energy representation of $\widetilde{V}$ up to isomorphism, and any positive energy representation of $\widetilde{V}$ is unitarily equivalent to $H\cdot (e^{\overline{L\R/\R}},\Pi)$ for some multiplicity space $H\in \Hilb$.
\end{proposition}
 Hence, the $\mathrm{W}^*$-functor $\Hilb\to \Rep^{pe}(\widetilde{V})$ sending $\mathbb{C}$ to $(e^{\overline{L\R/\R}},\Pi)$ is an equivalence of $\mathrm{W}^*$-categories. We call $(e^{\overline{L\R/\R}},\Pi)$ the \emph{standard} positive energy unitary representation of $\widetilde{V}$. Proposition \ref{prop: StandardRep}, together with the discussion leading to Equation \eqref{eq: RepLRsplits} provides the following~result. 
\begin{proposition}\label{prop: FunctorRepC0RtoReppeLR}
The $\mathrm{W}^*$-functor $\mathfrak{L}: \Rep(C_0(\R))\to \Rep^{pe}(\widetilde{L\R})$ sending $\int^\oplus_{t\in T}\delta_{p(t)}\mathrm{d}\upsilon(t)\in C_0(\R)$ to the Hilbert space $L^2(T,\upsilon)\otimes e^{\overline{L\R/\R}}$ with the action
\begin{equation}\label{eq: RepOfLRGivenByRepOfC0R}
(f,z)\cdot(\psi\otimes e^h) = z\cdot e^{-\frac{1}{2}||f||^2-\langle f,h\rangle}\cdot\big(t\mapsto e^{i p(t)\int f}\psi(t)\big)\otimes e^{f+h}
\end{equation}
for $(f,z)\in\widetilde{L\R}$ and $\psi\otimes e^h\in L^2(T,\upsilon)\otimes e^{\overline{L\R/\R}}$ is an equivalence of $\mathrm{W}^*$-categories.
\end{proposition}
\begin{proof}
    Consider the functor $\mathfrak{L}_0: \Rep(\R)\to \Rep^{pe}(\widetilde{L\R})$ defined as follows. Given $(H, \pi^H)\in\Rep(\R)$, we send it to the representation of $\widetilde{L\R}$ on $H\otimes e^{\overline{L\R/\R}}$ where $(f,z)\in\widetilde{L\R}$ acts by
    \[
   \pi^H\Big(\int f\Big)\otimes \Pi\bigg( f-\int f, z\bigg).
    \]
    By Equation \eqref{eq: RepLRsplits} and Proposition \ref{prop: StandardRep}, the functor $\mathfrak{L}_0$ is a $\mathrm{W}^*$-equivalence. The functor $\mathfrak{L}:\Rep(C_0(\R))\to \Rep^{pe}(\widetilde{L\R})$ is the composition of $\mathfrak{L}_0$ and the equivalence in Equation~\eqref{eq: C0RandR}. 
\end{proof}

\subsection{The inverse functor $\Rep^{pe}(\widetilde{L\R})\to \Rep(\Heis)$}

In this section, we produce a functor $\mathfrak{M}_0: \Rep^{pe}(\widetilde{L\R})\to \Rep(\Heis)$ providing an inverse to the restriction functor $\mathfrak{R}: \Rep(\Heis)\to \Rep^{pe}(\widetilde{L\R})$ defined in Section \ref{sec: restrictionFunctor}.

Fix $(K,\pi)\in \Rep^{pe}(\widetilde{L\R})$. We first describe how to produce compatible actions of all the $*$-algebras $\langle W(f)\rangle_{L_I\R}$ for $I\in \Jcal$. Let $I\in \Jcal$ be an interval and $f\in L_I\R$ be a localized loop. Then, we define the action of $W(f)\in \langle W(f)\rangle_{L_I\R}$ on $K$ by $\pi(f,1)\in U(K)$. This produces a $*$-representation of the dense $*$-subalgebra $\langle W(f)\rangle_{f\in L_I\R}$ of $\Heis(I)$ generated by $\{W(f)\}_{f\in L_I\R}$ on the Hilbert space $K$. These representations are compatible with the inclusions of intervals $J\subset I$ of $S^1$. In order to obtain a representation of $\Heis$ on $K$, it is only left to show that the local representations $\langle W(f)\rangle_{f\in L_I\R}\to U(K)\subset B(K)$ extend by continuity to representations of the algebras $\Heis(I)$ for all $I\in \Jcal$.

\begin{definition}
A unitary representation $(K, \pi)$ of $\widetilde{L\R}$ is called \emph{locally normal} if, for every $I\in\Jcal$, the representation of $\langle W(f)\rangle_{f\in L_I\R}$ on $K$ given by $W(f)\mapsto \pi(f, 1)$ extends to a normal representation of the von Neumann algebra $\Heis(I)$.   
\end{definition}

The rest of this section is devoted to arguing that every positive energy representation of $\widetilde{L\R}$ is locally normal. We first focus on the irreducible positive energy representations of $\widetilde{L\R}$.

\begin{lemma}\label{lemma: irrepsOfLRAreLocNorm}
Every irreducible positive energy representation of $\widetilde{L\R}$ is locally normal.
\end{lemma}
\begin{proof}
By Proposition \ref{prop: FunctorRepC0RtoReppeLR}, irreducible positive energy representations of $\widetilde{L\R}$ are parametrized by points of $\R$, up to isomorphism. Given $x\in \R$, the corresponding irreducible positive energy representation of $\widetilde{L\R}$ has underlying Hilbert space $e^{\overline{L\R/\R}}$ and action
\begin{equation}\label{eq: ActionIrrepPE}
(f,z)\cdot e^h = z\cdot e^{-\frac{1}{2}||f||^2-\langle f, h\rangle}\cdot e^{i x\int f} e^{f+h}
\end{equation}
for $(f,z)\in \widetilde{L\R}$ and $e^h\in e^{\overline{L\R/\R}}$. The corresponding induced action of the algebra $\langle W(f)\rangle_{f\in L_I\R}$ on $e^{\overline{L\R/\R}}$ is the (restriction to $\langle W(f)\rangle_{f\in L_I\R}$ of the) action $\pi_l$ in Definition \ref{def: IrrepsOfHeis} for $l$ the constant loop with value $x$. Hence, the representation \eqref{eq: ActionIrrepPE} on $e^{\overline{L\R/\R}}$ is locally~normal.  
\end{proof}

In order to extend the conclusion of Lemma \ref{lemma: irrepsOfLRAreLocNorm} to all positive energy representations of $\widetilde{L\R}$ we introduce direct integrals of representations of $\widetilde{L\R}$. We refer the reader to \cite[App. A]{MR1452364} for an introduction to direct integrals of Hilbert spaces and representations. Given a Borel space $X$ with a positive standard measure $\lambda$, a \emph{$\lambda$-measurable bundle of Hilbert spaces on $X$} is an $X$-parametrized family $\{H_x\}_{x\in X}$ of Hilbert spaces together with a vector subspace $\Gamma\subset \prod\limits_{x\in X}H_x$ such that
\begin{enumerate}
\item there exists a sequence $(\xi_1,\xi_2,\ldots )$ of elements of $\Gamma$ such that, for all $x\in X$, the vectors $(\xi_1(x),\xi_2(x),\ldots)$ span a dense subset of $H_x$;
\item for all $\xi\in \Gamma$, the function $x\mapsto ||\xi(x)||$ is $\lambda$-measurable;
\item given $\eta\in \prod\limits_{x\in X}H_x$ if, for all $\xi\in \Gamma$, the function $x\mapsto \langle \xi(x), \eta(x)\rangle$ is $\lambda$-measurable, then $\eta\in \Gamma$.
\end{enumerate}
Let $(\{H_x\}_{x\in X},\Gamma)$ be a $\lambda$-measurable bundle of Hilbert spaces. The elements of $\Gamma$ are called \emph{measurable sections} of $\{H_x\}_{x\in X}$. Given $\xi,\eta\in \Gamma$ measurable sections, the function $x\mapsto \langle \xi(x),\eta(x)\rangle$ is measurable. We say that $\xi\in \Gamma$ is square-integrable if $\int_{x\in X}||\xi(x)||^2 \mathrm{d}\lambda(x)<\infty$. If $\xi,\eta\in\Gamma$ are square-integrable and $z\in \mathbb{C}$ is a complex number, then $\xi+\eta$ and $z\xi$ are square-integrable, and $x\mapsto \langle \xi(x),\eta(x)\rangle$ is integrable. We write
\(
\langle \xi,\eta\rangle:=\int_{x\in X}\langle \xi(x),\eta(x)\rangle\mathrm{d}\lambda(x).
\)
The square-integrable sections of $\{H_x\}_{x\in X}$, up to norm zero sections, form a Hilbert space, which we denote by $\int^\oplus_{x\in X}H_x\mathrm{d}\lambda(x)$ and call the \emph{direct integral of $(\{H_x\}_{x\in X}, \Gamma)$}. Since $\lambda$ is a standard measure, the Hilbert space $\int^\oplus_{x\in X}H_x\mathrm{d}\lambda(x)$ is separable. 

A \emph{measurable bundle of operators on $(\{H_x\}_{x\in X},\Gamma)$} is an assignment of an operator $T_x\in B(H_x)$ for every $x\in X$ such that, for any measurable section $\xi\in \Gamma$, the section $x\mapsto T_x\xi(x)$ is measurable. Given a measurable bundle of operators $\{T_x\}_{x\in X}$ on $H:=(\{H_x\}_{x\in X},\Gamma)$, the function $x\mapsto ||T_x||$ is measurable. If this function is furthermore essentially bounded, we say that $\{T_x\}_{x\in X}$ is \emph{essentially bounded}. In this situation, if $\xi\in\int^\oplus_{x\in X}H_x\mathrm{d}\lambda(x)$ is a square integrable section, then $x\mapsto T_x\xi(x)$ is also square-integrable, and we obtain a bounded linear map $T:=\int_{x\in X}^\oplus T_x\mathrm{d\lambda}(x)\in B\Big(\int^\oplus_{x\in X}H_x\mathrm{d}\lambda(x)\Big)$, with norm equal to the essential upper bound of $x\mapsto ||T_x||$.

\begin{definition}
    Let $(X,\lambda)$ be a standard measure space. A \emph{$\lambda$-measurable bundle of unitary representations of $\widetilde{L\R}$} is a $\lambda$-measurable bundle of Hilbert spaces $(\{H_x\}_{x\in X}, \Gamma)$ together with, for every $x\in X$, a representation $\pi_x: \widetilde{L\R}\to U(H_x)$ of $\widetilde{L\R}$ on $H_x$ such that, for every $(f,z)\in \widetilde{L\R}$, the assignment $x\mapsto \pi_x(f,z)$ is a $\lambda$-measurable field of operators on $(\{H_x\}_{x\in X}, \Gamma)$.
\end{definition}
Given a $\lambda$-measurable bundle of unitary representations of $\widetilde{L\R}$, denoted by $(\{(H_x, \pi_x)\}_{x\in X}, \Gamma)$, the map $x\mapsto ||\pi_x(f,z)||$ is essentially bounded for every fixed $(f,z)$, and we define
\[
\Bigg(\int_{x\in X}^\oplus H_x\mathrm{d}\lambda(x), \int_{x\in X}^\oplus \pi_x\mathrm{d}\lambda(x)\Bigg)\in \Rep(\widetilde{L\R})
\]
to be the \emph{direct integral of $(\{(H_x, \pi_x)\}_{x\in X}, \Gamma)$ with respect to $\lambda$}. The following result shows that the direct integral of locally normal unitary representations of $\widetilde{L\R}$ is locally normal.

\begin{proposition}\label{prop: IntegralsOfLocNormAreLocNorm}
Let $(\{(H_x, \pi_x)\}_{x\in X}, \Gamma)$ be a $\lambda$-measurable bundle of unitary representations of $\widetilde{L\R}$ such that $(H_x,\pi_x)$ is locally normal for every $x\in X$. Then, the direct integral $\big(\int_{x\in X}^\oplus H_x\mathrm{d}\lambda, \int_{x\in X}^\oplus \pi_x\mathrm{d}\lambda(x)\big)\in \Rep(\widetilde{L\R})$ is locally~normal.
\end{proposition}
\begin{proof}
For every interval $I\in \Jcal$, we continue denoting by $\pi_x$ the extension of $\pi_x: \widetilde{L_I\R}\to B(H_x)$ to $\Heis(I)$, which exists by hypothesis. Fix $I\in \Jcal$ and $a\in \Heis(I)$. We show that the assignment $x\in X\mapsto \pi_x(a)\in B(H_x)$ is a $\lambda$-measurable field of operators on $(\{H_x\}_{x\in X}, \Gamma)$. It is enough to show, for every pair $\xi,\eta\in \Gamma$ of measurable sections, that $x\mapsto \langle \pi_x(a)(\xi_x), \eta_x\rangle$ is a $\lambda$-measurable function.

Fix $\xi,\eta\in \Gamma$ measurable sections. Let $\{W(f_i)\}_{i}$ be a sequence of elements of $\langle W(f)\rangle_{f\in L_I\R}$ converging ultraweakly to $a\in \Heis(I)$. We then have that, for $x\in X$,
\begin{align*}
\langle \pi_{x}(a)(\xi_x), \eta_x\rangle & = \langle \pi_{x}(\lim\limits_{i} W(f_i))(\xi_x),\eta_x\rangle \\
& = \langle \lim_{i}\pi_{x}(W(f_i))(\xi_x),\eta_x\rangle \\ & =\lim_i\langle \pi_x(W(f_i))(\xi_x),\eta_x\rangle.
\end{align*}
Since each $x\mapsto \langle \pi_x(W(f_i))(\xi_x),\eta_x\rangle $ is measurable by hypothesis and the limit of measurable functions is measurable, the claim follows. In addition, the map $x\in X\mapsto ||\pi_x(a)||\leq ||a||$ is automatically essentially bounded.

By the argument above, we obtain, for every interval $I\in\Jcal$, a direct integral representation $(\int_{x\in X}^\oplus H_x\mathrm{d}\lambda(x), \int_{x\in X}^\oplus \pi_{x,I}\mathrm{d}\lambda)$ of $\Heis(I)$, which restricts to the representation $(\int_{x\in X}^\oplus H_x\mathrm{d}\lambda(x), \int_{x\in X}^\oplus \pi_{x}\mathrm{d}\lambda)$ of $\langle W(f)\rangle_{f\in L_I\R}\subset \Heis(I)$, as needed
\end{proof}

The following lemma is the last result to prove that every positive energy representation of $\widetilde{L\R}$ is locally normal. 

\begin{lemma}\label{lemm: PERepsAreIntOfIrreps}
Every positive energy representation of $\widetilde{L\R}$ can be written as the direct integral of irreducible positive energy representations of $\widetilde{L\R}$.
\end{lemma}
\begin{proof}
Let $(K,\pi)\in \Rep^{pe}(\widetilde{L\R})$. Then, by Proposition \ref{prop: FunctorRepC0RtoReppeLR}, there exists an object $\int^\oplus_{t\in T}\delta_{p(t)}\mathrm{d}\upsilon(t)\in\Rep(C_0(\R))$ such that $(K,\pi)\cong \mathfrak{L}\bigg(\int^\oplus_{t\in T}\delta_{p(t)}\mathrm{d}\upsilon(t)\bigg)$. The latter is the direct integral of the $T$-parametrized family of irreducible representations $\{\mathfrak{L}(\delta_{p(t)})\}_{t\in T}$ with respect to $\upsilon$.
\end{proof}

\begin{proposition}\label{prop: AllRepsAreLocNormal}
Let $(K,\pi)$ be a positive energy representation of $\widetilde{L\R}$. Then, for every $I\in\Jcal$, the representation of $\langle W(f)\rangle_{f\in L_I\R}$ on $K$ given by letting $W(f)$ act by $\pi(f,1)$ for all $f\in L_I\R$ extends by continuity to a $*$-representation of $\Heis(I)$. These representations are compatible with the inclusions $\Heis(J)\hookrightarrow\Heis(I)$ whenever $I,J\in\Jcal$ are intervals such that $J\subset I$, and hence they define a representation of the Heisenberg conformal net.
\end{proposition}
\begin{proof}
Take $(K,\pi)$ to be a positive energy representation of $\widetilde{L\R}$. Then, by Lemma \ref{lemm: PERepsAreIntOfIrreps}, $(K,\pi)$ is equivalent to a direct integral of irreducible positive energy representations of $\widetilde{L\R}$, hence to a direct integral of locally normal representations, by Lemma \ref{lemma: irrepsOfLRAreLocNorm}. By Proposition \ref{prop: IntegralsOfLocNormAreLocNorm}, $(K,\pi)$ is equivalent to a locally normal representation, hence locally normal. The rest of the statement in the Proposition is clear.
\end{proof}

By Proposition \ref{prop: AllRepsAreLocNormal}, we obtain an assignment from positive energy representations of $\widetilde{L\R}$ to representations of the Heisenberg conformal net. Extending this assignment trivially at the level of morphisms, we obtain a functor
\[
\mathfrak{M}_0: \Rep^{pe}(\widetilde{L\R})\to \Rep(\Heis).
\]

\begin{proposition}\label{prop: RM0Equivalences}
The functors $\mathfrak{R}: \Rep(\Heis)\leftrightarrows\Rep^{pe}(\widetilde{L\R}): \mathfrak{M}_0$ are inverses of each other.
\end{proposition}
\begin{proof}
By the construction, it is clear that the composition $\mathfrak{R}\circ \mathfrak{M}_0$ is the identity functor on $\Rep^{pe}(\widetilde{L\R})$. Hence, $\mathfrak{R}$ is essentially surjective. Since it is also fully faithful, it is an equivalence, and $\mathfrak{M}_0$ is an inverse to $\mathfrak{R}$.
\end{proof}

By composing the equivalences $\Rep(C_0(\R))\xrightarrow{\cong}\Rep(\R)\xrightarrow{\cong}\Rep^{pe}(\widetilde{L\R})\xrightarrow[\cong]{\mathfrak{M_0}}\Rep(\Heis)$ we obtain a functor $\mathfrak{M}:\Rep(C_0(\R))\to \Rep(\Heis)$. We summarize the main result of this section in the following theorem.

\begin{theorem}\label{thm: RepHeisIsRepR}
Let $\mathfrak{M}:\Rep(C_0(\R))\to \Rep(\Heis)$ be the functor sending $\int^\oplus_{t\in T}\delta_{p(t)}\mathrm{d}\upsilon(t)\in \Rep(C_0(\R))$ to the representation of $\Heis$ on $L^2(T,\upsilon)\otimes e^{\overline{L\R/\R}}$ where, given $I\in \Jcal$ and $f\in L_I\R$, the element $W(f)\in \Heis(I)$ acts on $\psi\otimes e^h\in L^2(T,\upsilon)\otimes e^{\overline{L\R/\R}}$ via
\[
W(f)\cdot (\psi\otimes e^h) = e^{-\frac{1}{2}||f||^2-\langle f,h\rangle }\cdot \big(t\mapsto e^{ip(t)\int f}\psi(t)\big)\otimes e^{f+h}.
\]
For a morphism $F$ in $\Rep(C_0(\R))$, we set $\mathfrak{M}(F) = F\otimes \id$. Then, $\mathfrak{M}$ is an equivalence of $\mathrm{W}^*$-categories.
\end{theorem}
\begin{proof}
    The functor $\mathfrak{M}: \Rep(C_0(\R))\to \Rep(\Heis)$ is the composition of the functors
\[\Rep(C_0(\R))\xrightarrow{\mathfrak{L}}\Rep^{pe}(\widetilde{L\R})\xrightarrow[]{\mathfrak{M_0}}\Rep(\Heis).\]
The functor $\mathfrak{L}: \Rep(C_0(\R))\xrightarrow{}\Rep^{pe}(\widetilde{L\R})$ is an equivalence by Proposition \ref{prop: FunctorRepC0RtoReppeLR}. The functor $\mathfrak{M}_0:\Rep^{pe}(\widetilde{L\R})\to \Rep(\Heis)$ is an equivalence by Proposition \ref{prop: RM0Equivalences}.
\end{proof}

\subsection{Providing a tensorator for the functor $\Hilb\,\R\to \Rep(\Heis)$} \label{sec: ProvidingTensorator}
We can now characterize the fusion rules on $\Rep(\Heis)$. We will show in Proposition \ref{prop: TensoratorForM} that, endowing $\Rep(C_0(\R))$ with the tensor structure of $\Hilb\,\R$, we can upgrade $\mathfrak{M}$ to a functor with a tensorator that witnesses compatibility with the fusion rules. We will not argue compatibility with the associators at this point, we defer this to Section \ref{sec: RepZ/2AsTY}. Recall that the tensor structure on $\Hilb\,\R$ is given, for objects $\int^\oplus_{t\in T}\delta_{p(t)}\mathrm{d}\upsilon(t),\, \int^\oplus_{s\in S}\delta_{n(s)}\mathrm{d}\nu(s)\in \Rep(C_0(\R))$, by
\[
\int^\oplus_{t\in T}\delta_{p(t)}\mathrm{d}\upsilon(t)\otimes\int^\oplus_{s\in S}\delta_{n(s)}\mathrm{d}\nu(s) = \int^\oplus_{(t, s)\in T\times S}\delta_{p(t)+ n(s)}\mathrm{d}(\upsilon\times \nu)(t,s),
\]
which is also denoted as $\upsilon\otimes \nu = \upsilon\times \nu$ and $(p,\upsilon)\otimes (n, \nu) = (p+n, \upsilon\times\nu)$. The goal of this section is to prove the following~result.
\begin{proposition}\label{prop: TensoratorForM}
The equivalence of $\mathrm{W}^*$-categories $\mathfrak{M}: \Rep(C_0(\R))\to \Rep(\Heis)$ admits a unitary natural isomorphism filling the diagram
\[\begin{tikzcd}
{\Rep(C_0(\R))\times \Rep(C_0(\R))} && {\Rep(\Heis)\times \Rep(\Heis)} \\
{\Rep(C_0(\R))} && {\Rep(\Heis)}
\arrow["{\mathfrak{M}\times \mathfrak{M}}", from=1-1, to=1-3]
\arrow["\otimes"', from=1-1, to=2-1]
\arrow["{[-,-]}"{description}, Rightarrow, from=1-3, to=2-1]
\arrow["\boxtimes", from=1-3, to=2-3]
\arrow["{\mathfrak{M}}"', from=2-1, to=2-3]
\end{tikzcd}\]
where the tensor structure on $\Rep(C_0(\R))$ is that of $\Hilb\,\R$.
\end{proposition}

To produce the tensor structure $[-,-]$ on $\mathfrak{M}$, we make use of the category of DHR endomorphisms of $\Heis$. We will produce $[-,-]$ in the following steps. First of all, the construction will depend on a choice of suitably placed intervals $I_0, I_1, I_2\in \Jcal$ and loops $g_0\in L_{I_0}\R$, $g_1\in L_{I_1}\R$ and $g_2\in L_{I_2}\R$ of charge one, meaning that $\int g_i = 1$ for $i = 0,1,2$. Each such loop $g_i$ produces another equivalence $\mathfrak{M}_{g_i}: \Rep(C_0(\R))\to \Rep(\Heis)$ naturally isomorphic to $\mathfrak{M}$, see Definition \ref{def: FunctorMg}. Let $\Rep^{\text{irr}}(C_0(\R))$ be the full subcategory of $\Rep(C_0(\R))$ on the irreducible representations. Recall that these are given by $\delta_x$ for $x\in \R$. The advantage of using $\mathfrak{M}_{g_i}$ as opposed to $\mathfrak{M}$ is that the restriction of $\mathfrak{M}_{g_i}$ to $\Rep^{\text{irr}}(C_0(\R))$ factors through the tensor equivalence $(\mathfrak{E},\Psi):\Loc_{I_i}(\Heis)\xrightarrow{\cong}\Rep(\Heis)$, see Equation \eqref{eq: diagram FME}. This will allow us to produce a unitary natural isomorphism (not necessarily compatible with associators)
\begin{equation}\label{eq: TensoratorMIrreps}\begin{tikzcd}
	{\Rep^{\text{irr}}(C_0(\R))\times \Rep^{\text{irr}}(C_0(\R))} && {\Rep(\Heis)\times\Rep(\Heis)} \\
	{\Rep^{\text{irr}}(C_0(\R))} && {\Rep(\Heis)}.
	\arrow["{\mathfrak{M}_{g_1}\times\mathfrak{M}_{g_2}}", from=1-1, to=1-3]
	\arrow["\otimes"', from=1-1, to=2-1]
	\arrow["{^{\text{irr}}[-,-]_{g_1,g_2}^{g_0}}"{description}, Rightarrow, from=1-3, to=2-1]
	\arrow["\boxtimes", from=1-3, to=2-3]
	\arrow["{\mathfrak{M}_{g_0}}"', from=2-1, to=2-3]
\end{tikzcd}\end{equation}
Denoting by $\Rep^{\text{ss}}(C_0(\R))$ the full subcategory of $\Rep(C_0(\R))$ on objects given as countable direct sums of objects in $\Rep^{\text{irr}}(C_0(\R))$, possibly with infinite multiplicity, we will extend the tensorator above directly to $\Rep^{\text{ss}}(C_0(\R))$ by linearity with respect to direct sums and multiplicity spaces. We will finally extend by continuity the tensorator from $\Rep^{\text{ss}}(C_0(\R))$ to the whole category $\Rep(C_0(\R))$, see Proposition \ref{prop: ExtendTensoratorsByContinuity}. Using the natural isomorphisms $\mathfrak{M}\cong \mathfrak{M}_{g_i}$, we transport this structure to a tensorator (not necessarily compatible with associators) for $\mathfrak{M}.$

\begin{definition}\label{def: FunctorMg}
    Let $I\in \Jcal$ be an interval and $g\in L_{I}\R$ be a loop of charge $\int g= 1$. We~write 
    \[
\mathfrak{M}_g: \Rep(C_0(\R))\to \Rep(\Heis)  
\]
for the $\mathrm{W}^*$-functor that sends $\int^\oplus_{t\in T}\delta_{p(t)}\mathrm{d}\upsilon(t)\in \Rep(C_0(\R))$ to the Hilbert space $L^2(T,\upsilon)\otimes e^{\overline{L\R/\R}}$ with action
\[
W(f)\cdot (\psi\otimes e^h) = e^{-\frac{1}{2}||f||^2-\langle f,h\rangle }\cdot \big(t\mapsto e^{ip(t)\int g f}\psi(t)\big)\otimes e^{f+h},
\]
for $J\in \Jcal$, $f\in L_J\R$ and $\psi\otimes e^h\in L^2(T,\upsilon)\otimes e^{\overline{L\R/\R}}$. At the level of morphisms, we define $\mathfrak{M}_g(F) = F\otimes\id_{e^{\overline{L\R/\R}}}$ for any morphism $F$ in $\Rep(C_0(\R))$.
\end{definition}

The tensor product $L^2(T,\upsilon)\otimes e^{\overline{L\R/\R}}$ is canonically isomorphic to the Hilbert space $L^2(T, \upsilon, e^{\overline{L\R/\R}})$ of $e^{\overline{L\R/\R}}$-valued square-integrable functions on $(T,\upsilon)$. Note that $g-1\in L\R$ is a loop of charge 0 by hypothesis, and hence there exists a loop $G\in L\R$ whose derivative is $G' = g-1$. Fix such a loop $G\in L\R$ with $G' = g-1$.

\begin{lemma}\label{Lemm: MgsFormClique}
For every $I\in\Jcal$ and $g\in L_I\R$ of charge 1, there is a canonical unitary natural isomorphism $\kappa_g: \mathfrak{M}\cong \mathfrak{M}_{g}$ given on $\int^\oplus_{t\in T}\delta_{p(t)}\mathrm{d}\upsilon(t)\in\Rep(C_0(\R))$ by
\[
\begin{array}{cccc}
\kappa_g(p,\upsilon):& L^2(T,\upsilon, e^{\overline{L\R}/\R})&\to &L^2(T,\upsilon, e^{\overline{L\R}/\R})\\
& \psi&\mapsto &\Big(t\mapsto W(p(t)\cdot G)(\psi(t))\Big) .
\end{array}
\]

\end{lemma}
\begin{proof}
By the proof of Lemma \ref{lemm: IsInner}, $\kappa_g(p,\upsilon)$ is an isomorphism $\mathfrak{M}(p,\upsilon)\cong \mathfrak{M}_{g}(p,\upsilon).$  Naturality is clear.

\end{proof}

We now restrict our attention to the full subcategory $\Rep^{\text{irr}}(C_0(\R))\subset \Rep(C_0(\R))$ on the irreducible representations. Recall that, equipping $\Rep(C_0(\R))$ with the tensor structure of $\Hilb\,\R$, the fusion rules of $\Rep^{\text{irr}}(C_0(\R))$ are given by $\delta_x\otimes \delta_y = \delta_{x+y}$ for $x,y\in \R$. Let $I_0,I_1,I_2\in \Jcal$ be intervals such that they can be covered by an interval in $J\in \Jcal_{\mathrm{p}}$, and pick $g_0\in L_{I_0}\R$, $g_1\in L_{I_1}\R$ and $g_2\in L_{I_2}\R$ all of charge 1. We will provide a unitary natural isomorphism $^{\text{irr}}[-,-]_{g_1,g_2}^{g_0}$ filling 
\[\begin{tikzcd}
{\Rep^{\text{irr}}(C_0(\R))\times \Rep^{\text{irr}}(C_0(\R))} && {\Rep(\Heis)\times \Rep(\Heis)} \\
{\Rep^{\text{irr}}(C_0(\R))} & {} & {\Rep(\Heis)}.
\arrow["{\mathfrak{M}_{g_1}\times \mathfrak{M}_{g_2}}", from=1-1, to=1-3]
\arrow["\otimes", from=1-1, to=2-1]
\arrow["{^{\text{irr}}[-,-]_{g_1, g_2}^{g_0}}"{description}, Rightarrow, from=1-3, to=2-1]
\arrow["\boxtimes", from=1-3, to=2-3]
\arrow["{\mathfrak{M}_{g_0}}"', from=2-1, to=2-3]
\end{tikzcd}\]

Recall from Section \ref{sec: CrossedBalancedRepGA} and \cite[Thm. 4.9]{GcrossedbraidedRep} applied to the trivial group $G$ that there is a tensor equivalence $(\mathfrak{E}, \Psi):\Loc_{J}(\Heis)\to \Rep(\Heis)$ given by sending an endomorphism $\rho$ to the representation $(e^{\overline{L\R/\R}}, \pi_0\circ\rho)$, and a morphism $s\in \Heis(J)$ to $\pi_0(s)$. For $g\in\{g_0,g_1,g_2\}$, the functor $\mathfrak{M}_{g}: \Rep^{\text{irr}}(C_0(\R))\to \Rep(\Heis)$ factors through the equivalence $\mathfrak{E}:\Loc_{J}(\Heis)\to \Rep(\Heis)$. Indeed, define a functor $\mathfrak{F}_g: \Rep^{\text{irr}}(C_0(\R))\to \Loc_{J}(\Heis)$ by sending $\delta_x\in\Rep^{\text{irr}}(C_0(\R))$ to the endomorphism $\rho_{xg}: W(f)\mapsto e^{ix\int fg}W(f)$ for $I\in \Jcal_\mathrm{p}$, and $f\in L_I\R$ as defined in Section \ref{Sec: ConstructingIrreps}. At the level of morphisms, we send $z\cdot\id_{\delta_x}\in \Hom_{\Rep^{\text{irr}}(C_0(\R))}(\delta_x,\delta_x)$ to $z \cdot 1\in \Heis(J)$. The triangle 
\begin{equation}\label{eq: diagram FME}\begin{tikzcd}
	& {\Loc_{J}(\Heis)} & \\
	{\Rep^{\text{irr}}(C_0(\R))} && {\Rep(\Heis)}
	\arrow["{\mathfrak{E}}", from=1-2, to=2-3]
	\arrow["{\mathfrak{F}_g}", from=2-1, to=1-2]
	\arrow["{\mathfrak{M}_g}", from=2-1, to=2-3]
\end{tikzcd}\end{equation}
commutes strictly. Indeed, given $\delta_x\in \Rep^{\text{irr}}(C_0(\R))$, the $\Heis$-representation $\mathfrak{M}_g(\delta_x)$ consists of the Hilbert space $e^{\overline{L\R/\R}}$ where the action of $W(f)$ on $e^h\in e^{\overline{L\R/\R}}$ is given by $e^{ix\int gf}\cdot W(f)(e^h)$ for all $I\in\Jcal$, $f\in L_I\R$. This is exactly the representation $\mathfrak{E}\circ \mathfrak{F}_g(\delta_x)$. Let $G_{01}\in L_{J}\R$ and $G_{02}\in L_{J}\R$ be respectively primitives of $g_0-g_1$ and $g_0-g_2$. Note that, in $\Loc_{J}(\Heis)$, we have a unitary isomorphism
\[
\mathfrak{F}_{g_1}(\delta_x)\otimes \mathfrak{F}_{g_2}(\delta_y) = \rho_{xg_1}\circ \rho_{yg_2} = \rho_{xg_1+yg_2} \xrightarrow{W(xG_{01} + yG_{02})}\rho_{(x+y)g_0} =\mathfrak{F}_{g_0}(\delta_{x+y})
\]
for $\delta_x,\delta_y\in \Rep^{\text{irr}}(C_0(\R))$, by Remark \ref{rk: IntertwineActionsSameCharge}. Therefore, we can define the unitary isomorphism
\begin{align*}
^\text{irr}[x,y]_{g_1,g_2}^{g_0}:\mathfrak{M}_{g_1}(\delta_x)\boxtimes \mathfrak{M}_{g_2}(\delta_y) = \mathfrak{E}(\rho_{xg_1})\boxtimes\mathfrak{E}(\rho_{yg_2})&\xrightarrow{\Psi_{\rho_{xg_1}, \rho_{yg_2}}}\mathfrak{E}(\rho_{xg_1+yg_2}) \\&\xrightarrow{W(xG_{01}+yG_{02})}\mathfrak{E}(\rho_{(x+y)g_0})= \mathfrak{M}_{g_0}(\delta_{x+y}),
\end{align*}
where $\Psi$ is the tensorator of the equivalence $\mathfrak{E}:\Loc_{J}(\Heis)\to \Rep(\Heis)$. These unitaries are natural in $\delta_x$ and $\delta_y$, and provide the natural unitary isomorphism $^\text{irr}[-,-]_{g_1,g_2}^{g_0}(\delta_x,\delta_y):=\, ^\text{irr}[x,y]_{g_1,g_2}^{g_0}$.

Recall that we write $\Rep^{\text{ss}}(C_0(\R))$ for the full subcategory of $\Rep(C_0(\R))$ on objects which are countable direct sums of irreducible representations of $C_0(\R)$, possibly with infinite multiplicity. By extending the unitary natural isomorphism $^{\text{irr}}[-,-]_{g_1,g_2}^{g_0}$ from $\Rep^{\text{irr}}(C_0(\R))$ to $\Rep^{\text{ss}}(C_0(\R))$ linearly with respect to direct sums and multiplicity spaces, we obtain a unitary natural isomorphism $^{\text{ss}}[-,-]_{g_1,g_2}^{g_0}$ filling the diagram
\[\begin{tikzcd}
{\Rep^{\text{ss}}(C_0(\R))\times \Rep^{\text{ss}}(C_0(\R))} && {\Rep(\Heis)\times \Rep(\Heis)} \\
{\Rep^{\text{ss}}(C_0(\R))} & {} & {\Rep(\Heis)}.
\arrow["{\mathfrak{M}_{g_1}\times \mathfrak{M}_{g_2}}", from=1-1, to=1-3]
\arrow["\otimes", from=1-1, to=2-1]
\arrow["{^{\text{ss}}[-,-]_{g_1, g_2}^{g_0}}"{description}, Rightarrow, from=1-3, to=2-1]
\arrow["\boxtimes", from=1-3, to=2-3]
\arrow["{\mathfrak{M}_{g_0}}"', from=2-1, to=2-3]
\end{tikzcd}\]

For objects $\int^\oplus_{t\in T}\delta_{p(t)}\mathrm{d}\upsilon(t)$ and $\int^\oplus_{s\in S}\delta_{n(s)}\mathrm{d}\nu(s)$ of $\Rep(C_0(\R))$, we write ${^{\text{ss}}[(p,\upsilon),(n, \nu)]_{g_1, g_2}^{g_0}}:= {^{\text{ss}}[-,-]_{g_1, g_2}^{g_0}}((p,\upsilon),(n,\nu))$. The final step is to extend the unitary natural isomorphism $^{\text{ss}}[-,-]_{g_1, g_2}^{g_0}$ by continuity to all $\Rep(C_0(\R))$.

Fix objects $\int^\oplus_{t\in T}\delta_{p(t)}\mathrm{d}\upsilon(t),\,\int^\oplus_{s\in S}\delta_{n(s)}\mathrm{d}\nu(s)\in\Rep(C_0(\R))$ and assume that we have taken the intervals $I_1,I_2\in\Jcal$ so that $I_1\subset S^1_-$ and $I_2\subset S^1_+$. We still require that $I_0\cup I_1\cup I_2\subset J$ for some $J\in \Jcal_\mathrm{p}$. In what follows, we will produce a unitary isomorphism
\[
\mathfrak{M}_{g_1}(p,\upsilon)\boxtimes \mathfrak{M}_{g_2}(n,\nu)\cong \mathfrak{M}_{g_0}(p+n,\upsilon\times\nu).
\]
Such an isomorphism is, by definition, a unitary between the underlying Hilbert spaces of $\mathfrak{M}_{g_1}(p,\upsilon)\boxtimes \mathfrak{M}_{g_2}(n,\nu)$ and $\mathfrak{M}_{g_0}(p+n, \upsilon\times\nu)$. 

We write $\text{Forget}: \Rep(\Heis)\to \Hilb$ for the functor that sends a representation to its underlying Hilbert space. By definition, the underlying Hilbert space of $\mathfrak{M}_{g_0}(p+n, \upsilon\times\nu)$ is canonically unitarily equivalent to 
\begin{equation}\label{eq: Forgetg0}\text{Forget}(\mathfrak{M}_{g_0}(p+n, \upsilon\times\nu)) = L^2(T\times S, \upsilon\times\nu)\otimes e^{\overline{L\R/\R}}\cong L^2(T\times S, \upsilon\times\nu, e^{\overline{L\R/\R}}).
\end{equation}
Recall from Equation \eqref{eq: ConnesFusionRepA} that $\text{Forget}(\mathfrak{M}_{g_1}(p, \upsilon)\boxtimes \mathfrak{M}_{g_2}(n, \nu)) := \text{Forget}\big(\mathfrak{M}_{g_1}(p, \upsilon)(S^1_-)\boxtimes \mathfrak{M}_{g_2}(n, \nu)(S^1_+)\big)$. Since we have taken $g_1$ to be supported in $I_1\subset S^1_-$, the action of $\Heis(S^1_+)$ on $\mathfrak{M}_{g_1}(p, \upsilon)$ is exactly the action of $\Heis(S^1_+)$ on the representation $L^2(T,\upsilon)\cdot H_0$, and therefore we have that $\mathfrak{M}_{g_1}(p, \upsilon)(S^1_-) = (L^2(T, \upsilon)\cdot H_0)(S^1_-)$ is independent of $p$, where ${L^2(T, \upsilon)}\cdot H_0\in \Rep(\Heis)$ denotes the vacuum representation $H_0$ with multiplicity $L^2(T, \upsilon)$. Similarly $\mathfrak{M}_{g_2}(n, \nu)(S^1_+) = (L^2(S, \nu)\cdot H_0)(S^1_+)$. Therefore, we have a canonical unitary isomorphism
\begin{align}\label{eq: CompotisionUnitorAndSo}
 \nonumber\text{Forget}(\mathfrak{M}_{g_1}(p, \upsilon)&(S^1_-)\boxtimes \mathfrak{M}_{g_2}(n, \nu)(S^1_+)) \\&=  \nonumber\text{Forget}\big((L^2(T, \upsilon)\cdot H_0)(S^1_-)\boxtimes (L^2(S, \nu)\cdot H_0)(S^1_+)\big)\\&\cong L^2(T,\upsilon)\otimes L^2(S,\nu)\otimes \text{Forget}(H_0\boxtimes H_0)  \\& \nonumber\xrightarrow[\cong]{\id\otimes \text{unitor}}  L^2(T,\upsilon)\otimes L^2(S,\nu)\otimes  \text{Forget}(H_0)\\ & \nonumber\cong L^2(T\times S, \upsilon\times\nu, e^{\overline{L\R/\R}})
\end{align}
independent of the maps $p$ and $n$. Let us denote by $U\big(\mathfrak{M}_{g_1}(p, \upsilon)\boxtimes \mathfrak{M}_{g_2}(n, \nu), \mathfrak{M}_{g_0}(p+n, \upsilon\times\nu)\big)$ the topological space of unitary maps between the Hilbert spaces $\text{Forget}(\mathfrak{M}_{g_1}(p, \upsilon)\boxtimes \mathfrak{M}_{g_2}(n, \nu))$ and $\text{Forget}(\mathfrak{M}_{g_0}(p+n, \upsilon\times\nu))$. The unitaries in Equations \eqref{eq: Forgetg0} and \eqref{eq: CompotisionUnitorAndSo} produce a canonical homeomorphism (with respect to the strong operator topologies)
\begin{equation}\label{eq: UnderlyingHilbSpaces}
\mathfrak{U}:U\big(\mathfrak{M}_{g_1}(p,\upsilon)\boxtimes \mathfrak{M}_{g_2}(n,\nu), \mathfrak{M}_{g_0}(p+n,\upsilon\times\nu)\big)\cong U\big(L^2(T\times S,\upsilon\times\nu, e^{\overline{L\R/\R}})\big)
\end{equation}
independent of the maps $p$ and $n$.

In order to extend the tensorator $^{\text{ss}}[-,-]^{g_0}_{g_1,g_2}$ from $\Rep^{\text{ss}}(C_0(\R))$ to all of $\Rep(C_0(\R))$, we need the following terminology. Let $T$ be a locally compact Hausdorff space and $\upsilon$ a measure on $T$. A \emph{partition} of $T$ is a topological space of the form $\bigsqcup_{\alpha\in A} T_\alpha$, where $A$ is a finite or countable set and $\{T_\alpha\}_{\alpha\in A}$ is a family of locally compact, disjoint subsets of $T$ such that $T = \bigsqcup_{\alpha\in A} T_\alpha$ as sets. Note that, since $T$ is locally compact and Hausdorff, the subsets $T_\alpha\subset T$ are locally compact if and only if they are locally closed. Given such a partition $\bigsqcup_{\alpha\in A} T_\alpha$, we write $\upsilon_\alpha$ for the restriction of $\upsilon$ to ${T_\alpha}\subset T$. We also set $C_{\sqcup}(\bigsqcup_{\alpha\in A}T_\alpha) = \prod_{\alpha\in A}C({T_\alpha}, \R)$, where $C({T_\alpha}, \R)$ is the space of continuous functions from ${T_\alpha}\subset T$ into $\R$. We write
\[
C_\sqcup(T):=\bigcup_{\bigsqcup_{\alpha\in A} T_\alpha} C_\sqcup\big(\bigsqcup_{\alpha\in A}T_\alpha\big),
\]
where the union runs over all partitions of $T$. There is an assignment from $C_\sqcup(T)$ to the class of objects of $\Rep(C_0(\R))$ given by
\begin{equation}\label{eq: PartitionToObject}
\begin{array}{ccc}
    C_\sqcup\big(\bigsqcup_{\alpha\in A}T_\alpha\big) &\to & \Rep(C_0(\R))  \\ p&\mapsto & \bigoplus_{\alpha\in A}\int^\oplus_{t\in{T_\alpha}}\delta_{p(t)}\mathrm{d}\upsilon_\alpha(t).
\end{array}
\end{equation}
We also denote the right-hand side by $\int_{t\in T}^\oplus \delta_{p(t)}\mathrm{d}\upsilon(t)$, and note that its underlying Hilbert space is $\bigoplus\limits_{\alpha\in A}L^2(T_\alpha, \upsilon_\alpha) = L^2(T, \upsilon)$. We write $$C_\sqcup(T)_0\subset C_\sqcup(T)$$ for the subset of maps whose image is at most a countable subset of $\R$. Elements of $C_\sqcup(T)_0$ are assigned to representations in $\Rep^{\text{ss}}(C_0(\R))$ by the map in Equation \eqref{eq: PartitionToObject}. In addition, the essential image of the restriction of the map in Equation \eqref{eq: PartitionToObject} to $C_\sqcup(T)_0$ is equivalent to $\Rep^{\text{ss}}(C_0(\R))$. Let $S$ be another locally compact Hausdorff space with a Borel measure $\nu$.

\begin{lemma}\label{lemm: TensoratorInSSSimplified}
Fix $p\in C_\sqcup(T)_0$ and $n\in C_\sqcup(S)_0$. Then, under the homeomorphism \[\mathfrak{U}: U\big(\mathfrak{M}_{g_1}(p,\upsilon)\boxtimes \mathfrak{M}_{g_2}(n,\nu), \mathfrak{M}_{g_0}(p+n,\upsilon\times\nu)\big)\cong U\big(L^2(T\times S,\upsilon\times\nu, e^{\overline{L\R/\R}})\big),\]
the unitary isomorphism $^{\text{ss}}[(p, \upsilon),(n, \nu)]_{g_1,g_2}^{g_0}\in U\big(\mathfrak{M}_{g_1}(p,\upsilon)\boxtimes \mathfrak{M}_{g_2}(n,\nu), \mathfrak{M}_{g_0}(p+n,\upsilon\times\nu)\big)$ is sent to the~unitary
\begin{align}\label{eq: TensoratorsInUForm}
\psi\mapsto \Big((t,s)\mapsto  W\big(p(t)\cdot G_{01} + n(s)\cdot G_{02}\big)\big(\psi(t,s)\big)\Big)
\end{align}
in $U(L^2(T\times S, \upsilon\times\nu, e^{\overline{L\R/\R}}))$.
\end{lemma}
\begin{proof}
Let us first assume that $T$ and $S$ are a point. Then, $\int^\oplus_{t\in T}\delta_{p(t)}\mathrm{d}\upsilon(t)$ and $\int^\oplus_{s\in S}\delta_{n(s)}\mathrm{d}\nu(s)$ are irreducible objects $\delta_x$ and $\delta_y$ respectively. We compute the composition
\[H_0\xleftarrow[\cong]{\text{unitor}} H_0\boxtimes H_0 = \mathfrak{E}(\rho_{xg_1})\boxtimes\mathfrak{E}(\rho_{yg_2})\xrightarrow{\Psi} \mathfrak{E}(\rho_{xg_1+yg_2})\xrightarrow{W(xG_{01}+yG_{02})}\mathfrak{E}(\rho_{(x+y)g_0}) = H_0.
\]
Since $\Psi$ is given by the same map of Hilbert spaces as the unitor $H_0\boxtimes H_0\to H_0$ (see \cite[Thm. 3.8 and Sec. 4]{GcrossedbraidedRep} or \cite[Sec. 2 and Rk. 6.7]{Gui21}) the claim follows when $\int^\oplus_{t\in T}\delta_{p(t)}\mathrm{d}\upsilon(t)$ and $\int^\oplus_{s\in S}\delta_{n(s)}\mathrm{d}\nu(s)$ are irreducible objects of $\Rep(C_0(\R))$. By linearly extending with respect to multiplicities, the claim also holds when $p:T\to \R$ and $n:S\to \R$ have image exactly one point in $\R$. Let us now generalize the statement to direct~sums.

Let $p$ be associated to a partition $T = \sqcup_{\alpha\in A} T_\alpha$ and $n$ be associated to a partition $S = \sqcup_{\beta\in B} S_\beta$. By taking finer partitions if necessary, we can assume that $p$ and $n$ are constant on each of the $T_\alpha$'s and $S_\beta$'s respectively, as $p$ and $n$ have at most countable image. Then, the statement holds for the pair $(p_\alpha,n_\beta)\in C_\sqcup(T)_0\times C_\sqcup(S)_0$ for all $\alpha\in A$ and $\beta\in B$. We need to show that the statement holds for the pair $(p,n)$.

By definition, the tensorator $^{\text{ss}}[(p,\upsilon),(n,\nu)]_{g_1,g_2}^{g_0}$ belongs to the subset
\begin{align*}
\prod\limits_{(\alpha,\beta)\in A\times B}U\big(\mathfrak{M}_{g_1}(p_\alpha,\upsilon_\alpha)&\boxtimes \mathfrak{M}_{g_2}(n_\beta,\nu_\beta), \mathfrak{M}_{g_0}(p_\alpha + n_\beta, \upsilon_\alpha\times\nu_\beta)\big)\\&\subset U\big(\mathfrak{M}_{g_1}(p, \upsilon)\boxtimes \mathfrak{M}_{g_2}(n, \nu), \mathfrak{M}_{g_0}(p+n, \upsilon\times\nu)\big),
\end{align*}
and $\mathfrak{U}$ restricts to a homeomorphism
\begin{align*}
\prod\limits_{(\alpha,\beta)\in A\times B}U\big(\mathfrak{M}_{g_1}(p_\alpha,\upsilon_\alpha)&\boxtimes \mathfrak{M}_{g_2}(n_\beta,\nu_\beta), \mathfrak{M}_{g_0}(p_\alpha + n_\beta, \upsilon_\alpha\times\nu_\beta)\big)\\&\cong \prod\limits_{(\alpha,\beta)\in A\times B} U\big(L^2(T_\alpha\times S_\beta, \upsilon_\alpha\times\nu_\beta, e^{\overline{L\R/\R}})\big)\subset U\big(L^2(T\times S, \upsilon\times\nu, e^{\overline{L\R/\R}})\big).
\end{align*}
In fact, $\mathfrak{U}$ preserves the grading given by $A\times B$, and the restriction \[U(\mathfrak{M}_{g_1}\big(p_\alpha, \upsilon_\alpha\big)\boxtimes \mathfrak{M}_{g_2}(n_\beta, \nu_\beta), \mathfrak{M}_{g_0}(p_\alpha + n_\beta, \upsilon_\alpha\times\nu_\beta))\cong U(L^2(T_\alpha\times S_\beta, \upsilon_\alpha\times\nu_\beta, e^{\overline{L\R/\R}}))\] is again given by $\mathfrak{U}$ in Equation \eqref{eq: UnderlyingHilbSpaces} for $\upsilon = \upsilon_\alpha$ and $\nu = \nu_\beta$. Hence, the claim follows.
\end{proof}

We now endow $C_\sqcup(T)\times C_\sqcup(S)$ with a topology and show that the formula from Lemma \ref{lemm: TensoratorInSSSimplified} extends by continuity to a dashed map
\[\begin{tikzcd}
	{C_\sqcup(T)\times C_\sqcup(S)} & {U\big(L^2(T\times S,\upsilon\times\nu, e^{\overline{L\R/\R}})\big)} \\
	{C_\sqcup(T)_0\times C_\sqcup(S)_0}
	\arrow[dashed, from=1-1, to=1-2]
	\arrow[hook, from=2-1, to=1-1]
	\arrow["{\text{Equation}\ \eqref{eq: TensoratorsInUForm}}"', from=2-1, to=1-2]
\end{tikzcd}\]
Given $p_1,p_2\in C_\sqcup(T)$, we define the extended pseudo-metric
\begin{equation}\label{eq: MetricContinuous}
\text{dist}(p_1,p_2) := ||p_1-p_2||_\infty = \text{ess}\sup\limits_{t\in T}|p_1(t)-p_2(t)|,
\end{equation}
where the essential supremum is taken with respect to the measure $\upsilon$ on $T$.
\begin{lemma}\label{lemm: dense}
With the extended pseudo-metric \eqref{eq: MetricContinuous}, the subspace $C_\sqcup(T)_0$ is dense in $C_\sqcup(T)$.
\end{lemma}
\begin{proof}
Let $p\in C_\sqcup(T)$ and $\varepsilon>0$. We may assume, without loss of generality, that $p$ is continuous on all of $T$. Consider the partition of $T$ given by the sets $$\{p^{-1}\Big(\big(n\varepsilon, (n+1)\varepsilon\big]\Big)\}_{n\in\mathbb{Z}},$$ and define $p': \bigsqcup\limits_{n\in \mathbb{Z}} T_n\to \R$ by sending $t\in T_n$ to $(n+1/2)\varepsilon$. Note that each $\big(n\varepsilon, (n+1)\varepsilon\big]$ is locally closed, hence its preimage under $p$ is locally closed. Since $T$ is locally compact and Hausdorff, $p^{-1}\Big(\big(n\varepsilon, (n+1)\varepsilon\big]\Big)$ is locally compact. Then, $p'\in C_{\sqcup}(T)_0$ and it holds by construction that $|p(t)-p'(t)|\leq\varepsilon/2$ for every $t\in T$. Hence, the claim follows.
\end{proof}

We define on the product $C_\sqcup(T)\times C_\sqcup(S)$ the extended pseudo-metric $$\text{dist}\big((p_1,n_1), (p_2,n_2)\big) = \max\{\text{dist}(p_1,p_2), \text{dist}(n_1,n_2)\},$$ where the metric on $C_\sqcup(S)$ is defined analogously to that in $C_\sqcup(T)$. Then, it holds that $C_\sqcup(T)_0\times C_\sqcup(S)_0\subset C_\sqcup(T)\times C_\sqcup(S)$ is dense.

Recall that, given loops $g_i$ and $g_j$ in $L\R$ of the same charge, we write $G_{ij}\in L\R$ for a primitive of $g_i-g_j$.

\begin{definition}\label{def: Family[-,-]g0g1g2}
Let $T$ and $S$ be locally compact Hausdorff spaces with Borel measures $\upsilon$ and $\nu$ respectively. Let $I_0,I_1,I_2\in \Jcal$ be intervals such that $I_1\subset S_-^1$, $I_2\subset S_+^1$, and $I_0\cup I_1\cup I_2\subset J$ for some $J\in\Jcal_\mathrm{p}$. Let $g_0\in L_{I_0}\R$, $g_1\in L_{I_1}\R$ and $g_2\in L_{I_2}\R$ be loops of charge 1. Let $(p,n)\in C_\sqcup(T)\times C_\sqcup(S)$. We write $[(p, \upsilon), (n, \nu)]_{g_1,g_2}^{g_0}\in U\big(\mathfrak{M}_{g_1}(p,\upsilon)\boxtimes \mathfrak{M}_{g_2}(n,\nu), \mathfrak{M}_{g_0}(p+n,\upsilon\times\nu)\big)$ for the image under the homeomorphism
\[\mathfrak{U}^{-1}:U\big(L^2(T\times S, \upsilon\times\nu, e^{\overline{L\R/\R}})\big)\xrightarrow{\cong} U\big(\mathfrak{M}_{g_1}(p,\upsilon)\boxtimes \mathfrak{M}_{g_2}(n,\nu), \mathfrak{M}_{g_0}(p+n,\upsilon\times\nu)\big) 
\]
of the unitary sending $\psi\in L^2(T\times S,\upsilon\times\nu, e^{\overline{L\R/\R}})$ to
\begin{align*}
\Big((t,s)\mapsto & W\big(p(t)\cdot G_{01} + n(s)\cdot G_{02}\big)\big(\psi(t,s)\big)\Big) \in L^2(T\times S,\upsilon\times\nu, e^{\overline{L\R/\R}}).
\end{align*}
\end{definition}

We can now show that the collection of unitaries $[(p, \upsilon), (n, \nu)]_{g_1,g_2}^{g_0}$ provides a unitary natural isomorphism $[-,-]_{g_1,g_2}^{g_0}$ filling 
\[\begin{tikzcd}
{\Rep(C_0(\R))\times \Rep(C_0(\R))} && {\Rep(\Heis)\times \Rep(\Heis)} \\
{\Rep(C_0(\R))} & {} & {\Rep(\Heis)}.
\arrow["{\mathfrak{M}_{g_1}\times \mathfrak{M}_{g_2}}", from=1-1, to=1-3]
\arrow["\otimes", from=1-1, to=2-1]
\arrow["{[-,-]_{g_1, g_2}^{g_0}}"{description}, Rightarrow, from=1-3, to=2-1]
\arrow["\boxtimes", from=1-3, to=2-3]
\arrow["{\mathfrak{M}_{g_0}}"', from=2-1, to=2-3]
\end{tikzcd}\]
We will need the following lemma. Given a representation $H$ of $\Heis$, we write $U(H)$ for $U(\text{Forget}(H))$, the space of unitaries on the underlying Hilbert space of $H$. Given objects $\int^{\oplus}_{t\in T}\delta_{p(t)}\mathrm{d}\upsilon(t),\,\int^{\oplus}_{s\in S}\delta_{n(s)}\mathrm{d}\nu(s)\in\Rep(C_0(\R))$, intervals $I_1, I_2\in\Jcal$ with $I_1\subset S^1_-$ and $I_2\subset S^1_+$ and localized loops $g_1\in L_{I_1}\R$ and $g_2\in L_{I_2}\R$ of charge 1, the underlying Hilbert spaces of $\mathfrak{M}_{g_1}(p,\upsilon)$ and $\mathfrak{M}_{g_1}(p,\upsilon)\boxtimes \mathfrak{M}_{g_2}(n,\nu)$ are independent of $p$ and $n$. Hence, we may write $U(\mathfrak{M}_{g_1}(\upsilon))$ and $U\big(\mathfrak{M}_{g_1}(\upsilon)\boxtimes \mathfrak{M}_{g_2}(\nu)\big)$, without making reference to the maps $p$ and $n$.

\begin{lemma}\label{lemma: HeisActionsDependContinuously}
Let $(T,\upsilon)$ be a Borel measure space, $I_1\in \Jcal$ be an interval and $g_1\in L_{I_1}\R$ be a loop of charge 1. Given a function $p\in C_\sqcup(T)$, we write $\pi_1^p$ for the action of $\Heis$ on $\mathfrak{M}_{g_1}(p,\upsilon).$ Let $f\in L_I\R$ be a loop supported on an interval $I\in \Jcal$. Then, the map
\[
\begin{array}{ccc}
C_\sqcup(T) &\to &U(\mathfrak{M}_{g_1}(\upsilon))  \\
p&\mapsto &\pi_{1, I}^{p}(W(f)) 
\end{array}
\]
is strongly continuous. Let $(S,\nu)$ be another Borel space. Assume that $I_1\subset S_-^1$ and let $I_2\in \Jcal$ be another interval with $I_2\subset S_+^1$. Let $g_2\in L_{I_2}\R$ be another loop of charge 1. For a pair $(p,n)\in C_\sqcup(T)\times C_\sqcup(S)$ we write $\pi_{1\boxtimes 2}^{p\boxtimes n}$ for the action of $\Heis$ on $\mathfrak{M}_{g_1}(p,\upsilon)\boxtimes\mathfrak{M}_{g_2}(n,\nu)$. Then, the map  
\[
\begin{array}{ccc}
C_\sqcup(T)\times C_\sqcup(S) &\to &U(\mathfrak{M}_{g_1}(p,\upsilon)\boxtimes \mathfrak{M}_{g_2}(n,\nu))  \\
(p,n)&\mapsto &\pi_{1\boxtimes 2, I}^{p\boxtimes n}(W(f)) 
\end{array}
\]
is strongly continuous.
\end{lemma}
\begin{proof}
The unitary $\pi_{1, I}^{p}(W(f))\in U(\mathfrak{M}_{g_1}(\upsilon))$ acts on $\psi\otimes e^h\in \mathfrak{M}_g(\upsilon) = L^2(T,\upsilon)\otimes e^{\overline{L\R/\R}}$ by
\[
\pi_{1, I}^{p}(W(f))(\psi\otimes e^h)
= (t\mapsto e^{ip(t)\int g_1f}\psi(t))\otimes W(f)(e^h).
\]
Hence, it is enough to argue that the map $C_\sqcup(T)\to U(L^2(T,\upsilon))$ given by $$p\mapsto e^{ip(t)\int g_1f}\cdot\id_{L^2(T,\upsilon)}$$ is strongly continuous. This is clear because $|e^{i p(t)\int g_1f} - e^{ip_n(t)\int g_1f}|\leq |p(t)-p_n(t)|\cdot |\int g_1f| $ for all $p,p_n\in C_\sqcup(T)$ and all $t\in T$.

We now show the second statement. We use $\pi_{1,I}^p(W(f))\boxtimes \id\in B\big( \mathfrak{M}_{g_1}(p, \upsilon)(I)\boxtimes \mathfrak{M}_{g_2}(n, \nu)(J)\big)$ to denote the map induced by $ \pi_{1}^p(W(f))\otimes \id$. Recall from Section \ref{Sec: ConfNets} that the unitary $\pi_{1\boxtimes 2, I}^{p\boxtimes n}(W(f))\in U\big(\mathfrak{M}_{g_1}(p,\upsilon)\boxtimes\mathfrak{M}_{g_2}(n,\nu)\big)$ is given by the composition \cite[Sec. 3.2]{GcrossedbraidedRep}
\begin{align*}\hspace{-.3cm}
\mathfrak{M}_{g_1}(p,\upsilon)(S^1_-)\boxtimes\mathfrak{M}_{g_2}(n,\nu)(S^1_+)&\fixedxrightarrow{\gamma^\bullet} \mathfrak{M}_{g_1}(p,\upsilon)(I)\boxtimes\mathfrak{M}_{g_2}(n,\nu)(J)\\&\fixedxrightarrow{\pi_{1,I}^p(W(f))\boxtimes \id}\mathfrak{M}_{g_1}(p,\upsilon)(I)\boxtimes\mathfrak{M}_{g_2}(n,\nu)(J)\\&\fixedxrightarrow{(\gamma^\bullet)^{-1}} \mathfrak{M}_{g_1}(p,\upsilon)(S^1_-)\boxtimes\mathfrak{M}_{g_2}(n,\nu)(S^1_+)
\end{align*}
for $J\in \Jcal$ any interval disjoint from $I$ and $\gamma^\bullet$ the path continuation of a path $\gamma:[0,1]\to \text{Conf}_2(S^1)$ from $S_-^1\times S_+^1$ to $I\times J$. Since the assignment $p\in C_\sqcup(T)\mapsto \pi_{1}^p(W(f))\otimes \id\in U\big(\mathfrak{M}_{g_1}(p,\upsilon)(I)\otimes \mathfrak{M}_{g_2}(n,\nu)(J)\big)$ is strongly continuous by the discussion above and the map 
\[
U\big(\mathfrak{M}_{g_1}(p,\upsilon)(I)\boxtimes\mathfrak{M}_{g_2}(n,\nu)(J)\big)\xrightarrow{(\gamma^\bullet)^{-1}\circ (-) \circ\gamma^\bullet}U\big(\mathfrak{M}_{g_1}(p,\upsilon)(S^1_-)\boxtimes\mathfrak{M}_{g_2}(n,\nu)(S^1_+)\big)
\]
is a homeomorphism, the claim follows.
\end{proof}

We can now show that the family $[-,-]^{g_0}_{g_1,g_2}$ constructed in Definition \ref{def: Family[-,-]g0g1g2} provides a unitary natural isomorphism $\mathfrak{M}_{g_1}(-)\boxtimes \mathfrak{M}_{g_2}(-)\cong \mathfrak{M}_{g_0}(-\otimes -)$, where $\otimes$ denotes the tensor structure of $\Hilb\,\R$.

\begin{proposition}\label{prop: ExtendTensoratorsByContinuity}
Let $T$ and $S$ be locally compact Hausdorff spaces with Borel measures $\upsilon$ and $\nu$ respectively. Let $I_0, I_1,I_2\in \Jcal$ be intervals with $I_1\subset S^1_-$ and $I_2\subset S^1_+$ and such that $I_0\cup I_1\cup I_2$ can be covered by an interval $J\in \Jcal_\mathrm{p}$. Let $g_0\in L_{I_0}\R$, $g_1\in L_{I_1}\R$ and $g_2\in L_{I_2}\R$ be loops of charge 1. Let $(p,n)\in C_\sqcup(T)\times C_\sqcup(S)$ be a pair of functions. Then, 
\begin{enumerate}
\item $[(p, \upsilon),(n, \nu)]_{g_1,g_2}^{g_0}\in U\big(\mathfrak{M}_{g_1}(p,\upsilon)\boxtimes \mathfrak{M}_{g_2}(n,\nu), \mathfrak{M}_{g_0}(p+n,\upsilon\times\nu)\big)$ intertwines the actions of $\Heis$ on $\mathfrak{M}_{g_1}(p,\upsilon)\boxtimes \mathfrak{M}_{g_2}(n,\nu)$ and $\mathfrak{M}_{g_0}(p+n,\upsilon\times\nu)$;
\item The family $[-,-]_{g_1,g_2}^{g_0}$ is natural with respect to morphisms in $\Rep(C_0(\R))$.
\end{enumerate}
\end{proposition}
\begin{proof}

We write\footnote{Note that the subscript $0$ on $\pi_{0}^{p+n}$ refers to the subscript $0$ on the loop $g_0$, and not to the vacuum representation of $\Heis$.} $\pi_{0}^{p+n}$ for the action of $\Heis$ on $\mathfrak{M}_{g_0}(p+n,\upsilon\times\nu)$, and $\pi_{1\boxtimes 2}^{p\boxtimes n}$  for the action of $\Heis$ on $\mathfrak{M}_{g_1}(p,\upsilon)\boxtimes \mathfrak{M}_{g_2}(n, \nu)$. For every $I\in \Jcal$ and $a\in \Heis(I)$, we need to show that
\[
\pi_{0, I}^{p+ n}(a)\circ [(p, \upsilon),(n, \nu)]_{g_1,g_2}^{g_0} = [(p, \upsilon),(n, \nu)]_{g_1,g_2}^{g_0}\circ \pi_{1\boxtimes 2}^{p\boxtimes n}(a).
\]
Actually, since the representations $\pi_{0}^{p+n}$ and $\pi_{1\boxtimes 2}^{p\boxtimes n}$ of $\Heis(I)$ are normal, it is enough to show the equality for $a \in \langle W(f)\rangle_{f\in L_I\R}$. Fix $a = W(f)$ for some $f\in L_I\R$. Then, we know that the equality holds if $(p,n)\in C_\sqcup(T)_0\times C_\sqcup(S)_0$, since $[(p, \upsilon),(n, \nu)]_{g_1,g_2}^{g_0} =\, ^{\text{ss}}[(p, \upsilon),(n, \nu)]_{g_1,g_2}^{g_0}$ by Lemma \ref{lemm: TensoratorInSSSimplified}, and $^{\text{ss}}[(p, \upsilon),(n, \nu)]_{g_1,g_2}^{g_0}$ is by construction a morphism in $\Rep(\Heis)$. Recall that $C_\sqcup(T)_0\times C_\sqcup(S)_0$ is dense in $C_\sqcup(T)\times C_\sqcup(S)$ by Lemma \ref{lemm: dense} and note that the map
\[\hspace{-1cm}
\begin{array}{ccccc}
C_\sqcup(T)\times C_\sqcup(S) &\to &U(L^2(T\times S, \upsilon\times\nu, e^{\overline{L\R/\R}})) &\xrightarrow{\mathfrak{U}^{-1}}&U(\mathfrak{M}_{g_1}(\upsilon)\boxtimes \mathfrak{M}_{g_2}(\nu), \mathfrak{M}_{g_0}(\upsilon\times\nu)) \\
(p,n)&\mapsto & \big((t,s)\mapsto W(p(t)\cdot G_{01}+n(s)\cdot G_{02})\big)&\mapsto & [(p, \upsilon),(n, \nu)]_{g_1,g_2}^{g_0}
\end{array}
\]
is strongly continuous, as it is the composition of strongly continuous maps. Therefore, since
\[
\begin{array}{ccc}
C_\sqcup(T)\times C_\sqcup(S) &\to &U(\mathfrak{M}_{g_1}(\upsilon)\boxtimes \mathfrak{M}_{g_2}(\nu))  \\
(p,n)&\mapsto &\pi_{1\boxtimes 2}^{p\boxtimes n}(W(f)) 
\end{array}
\]
and 
\[
\begin{array}{ccc}
C_\sqcup(T)\times C_\sqcup(S) &\to &U(\mathfrak{M}_{g_0}(\upsilon\times \nu))  \\
(p,n)&\mapsto &\pi_{0}^{p+ n}(W(f)) 
\end{array}
\]
are strongly continuous by Lemma \ref{lemma: HeisActionsDependContinuously}, the first claim follows. 

We now prove the second part. Fix objects $\int^\oplus_{t\in T_1}\delta_{p_1(t)}\mathrm{d}\upsilon_1(t), \, \int^\oplus_{t\in T_2}\delta_{p_2(t)}\mathrm{d}\upsilon_2(t),\, \int^\oplus_{s\in S_1}\delta_{n_1(s)}\mathrm{d}\nu_1(s)$ and $\int^\oplus_{s\in S_2}\delta_{n_2(s)}\mathrm{d}\nu_2(s)$ of $\Rep(C_0(\R))$, and morphisms
\begin{align*}
F&\in \Hom_{\Rep(C_0(\R))}\bigg(\int^\oplus_{t\in T_1}\delta_{p_1(t)}\mathrm{d}\upsilon_1(t), \int^\oplus_{t\in T_2}\delta_{p_2(t)}\mathrm{d}\upsilon_2(t)\bigg)\\
G&\in \Hom_{\Rep(C_0(\R))}\bigg(\int^\oplus_{s\in S_1}\delta_{n_1(s)}\mathrm{d}\nu_1(s), \int^\oplus_{s\in S_2}\delta_{n_2(s)}\mathrm{d}\nu_2(s)\bigg).
\end{align*}
We need to show that
\begin{equation}\label{eq: naturality}
\mathfrak{M}_{g_0}(F\otimes G)\circ [(p_1,\upsilon_1), (n_1, \nu_1)]^{g_0}_{g_1, g_2} = [(p_2,\upsilon_2), (n_2, \nu_2)]^{g_0}_{g_1, g_2}\circ \big(\mathfrak{M}_{g_1}(F)\boxtimes \mathfrak{M}_{g_2}(G)\big).
\end{equation}
Fix $\varepsilon>0$ and let $\mathfrak{X}_\varepsilon: \R\to \R$ denote the Borel map that sends $x\in (k\varepsilon, [k+1]\varepsilon]$ to $\mathfrak{X}_{\varepsilon}(x) = (k+\frac{1}{2})\varepsilon$. We define $p_{i,\epsilon} : =\mathfrak{X}_\varepsilon\circ  p_i$ and $n_{i, \varepsilon}: = \mathfrak{X}_\varepsilon\circ n_i$ for $i = 1,2$. By the same arguments as Lemma \ref{lemm: dense}, $p_{i,\varepsilon}$ and $n_{i,\varepsilon}$ tend to $p_i$ and $n_i$ respectively as $\varepsilon\to 0$ for $i  =1,2$. By definition, for every function $f\in C_0(\R)$, the map $F$ intertwines the actions of pointwise multiplication by $f\circ p_1$ on $L^2(T_1, \upsilon_1)$ and by $f\circ p_2$ on $L^2(T_2, \upsilon_2)$. Let $B_b(T_i)$ be the space of bounded Borel maps from $T_i$ into $\C$, for $i = 1,2$. Then, $B_b(T_i)$ acts on $L^2(T_i, \upsilon)$ by pointwise multiplication, and the map $B_b(T_i)\to B(L^2(T_i,\upsilon))$ is strongly continuous with respect to bounded pointwise convergence on the domain, by dominated convergence. Since any function $f = g\circ \mathfrak{X}_\varepsilon\in B_b(\R)$ for $g\in C_0(\R)$ can be approximated by functions in $C_0(\R)$ under bounded pointwise convergence, the map $F$ intertwines the actions of pointwise multiplication by $f = g\circ \mathfrak{X}_\varepsilon$ for any function $g\in C_0(\R).$ Therefore, $F: L^2(T_1, \upsilon_1)\to L^2(T_2, \upsilon_2)$ intertwines the actions of $C_0(\R)$ given by pointwise multiplication by the pullback along $p_{1, \varepsilon}$ and $p_{2, \varepsilon}$. Similarly, $G: L^2(S_1, \nu_1)\to L^2(S_2, \nu_2)$ intertwines the actions of $C_0(\R)$ given by pointwise multiplication by the pullback along $n_{1, \varepsilon}$ and $n_{2, \varepsilon}$. 

Since $p_{i, \varepsilon}\in C_{\sqcup}(T)_{0}$ and $n_{i, \varepsilon}\in C_{\sqcup}(S)_0$, by Lemma \ref{lemm: TensoratorInSSSimplified} it holds that $[(p_{1,\varepsilon},\upsilon_1), (n_{1,\varepsilon}, \nu_1)]^{g_0}_{g_1, g_2} = \,^{\text{ss}}[(p_{1,\varepsilon},\upsilon_1), (n_{1,\varepsilon}, \nu_1)]^{g_0}_{g_1, g_2}$ and $[(p_{2,\varepsilon},\upsilon_2), (n_{2,\varepsilon}, \nu_2)]^{g_0}_{g_1, g_2} = \,^{\text{ss}}[(p_{2,\varepsilon},\upsilon_2), (n_{2,\varepsilon}, \nu_2)]^{g_0}_{g_1, g_2}$. By the discussion above, we have
\begin{align*}
F&\in \Hom_{\Rep^{\text{ss}}(C_0(\R))}\bigg(\int^\oplus_{t\in T_1}\delta_{p_{1,\varepsilon}(t)}\mathrm{d}\upsilon_1(t), \int^\oplus_{t\in T_2}\delta_{p_{2,\varepsilon}(t)}\mathrm{d}\upsilon_2(t)\bigg)\\
G&\in \Hom_{\Rep^{\text{ss}}(C_0(\R))}\bigg(\int^\oplus_{s\in S_1}\delta_{n_{1,\varepsilon}(s)}\mathrm{d}\nu_1(s), \int^\oplus_{s\in S_2}\delta_{n_{2, \varepsilon}(s)}\mathrm{d}\nu_2(s)\bigg),
\end{align*}
and since $^{\text{ss}}[-,-]^{g_0}_{g_1,g_2}$ is natural by construction, it holds that
\[
\mathfrak{M}_{g_0}(F\otimes G)\circ [(p_{1,\varepsilon},\upsilon_1), (n_{1,\varepsilon}, \nu_1)]^{g_0}_{g_1, g_2} = [(p_{2,\varepsilon},\upsilon_2), (n_{2,\varepsilon}, \nu_2)]^{g_0}_{g_1, g_2}\circ \big(\mathfrak{M}_{g_1}(F)\boxtimes \mathfrak{M}_{g_2}(G)\big).
\]
By taking the limit as $\varepsilon\to 0$, Equation \eqref{eq: naturality} holds. 
\end{proof}

We can now prove the main result of this section, Proposition \ref{prop: TensoratorForM}. Let us recall the statement here.

\begin{proposition*}
The equivalence of $\mathrm{W}^*$-categories $\mathfrak{M}: \Rep(C_0(\R))\to \Rep(\Heis)$ admits a unitary natural isomorphism $[-,-]$ filling the diagram
\[\begin{tikzcd}
{\Rep(C_0(\R))\times \Rep(C_0(\R))} && {\Rep(\Heis)\times \Rep(\Heis)} \\
{\Rep(C_0(\R))} && {\Rep(\Heis)},
\arrow["{\mathfrak{M}\times \mathfrak{M}}", from=1-1, to=1-3]
\arrow["\otimes"', from=1-1, to=2-1]
\arrow["{[-,-]}"{description}, Rightarrow, from=1-3, to=2-1]
\arrow["\boxtimes", from=1-3, to=2-3]
\arrow["{\mathfrak{M}}"', from=2-1, to=2-3]
\end{tikzcd}\]
where the tensor structure on $\Rep(C_0(\R))$ is that of $\Hilb\,\R$.
\end{proposition*}

\begin{proof}[Proof of Proposition \ref{prop: TensoratorForM}]
Let $I_0,I_1,I_2\in \Jcal$ be intervals such that $I_1\subset S_-^1$ and $I_2\subset S_+^1$ and $I_0,I_1,I_2$ can be covered by an interval $J\in \Jcal_\mathrm{p}$. Let $g_0\in L_{I_0}\R$, $g_1\in L_{I_1}\R$ and $g_2\in L_{I_2}\R$ be loops of charge 1. Given $\int^{\oplus}_{t\in T}\delta_{p(t)}\mathrm{d}\upsilon(t), \int^\oplus_{s\in S}\delta_{n(s)}\mathrm{d}\nu(s)\in \Rep(C_0(\R))$, we define
\begin{align*}
[(p,\upsilon)&, (n, \nu)]:= \\ &\kappa_{g_0}(p+n, \upsilon\times\nu)^{-1} \circ [(p,\upsilon), (n, \nu)]_{g_1, g_2}^{g_0}\circ (\kappa_{g_1}(p,\upsilon)\boxtimes \kappa_{g_2}(n,\nu)).
\end{align*}
By Lemma \ref{Lemm: MgsFormClique} and Proposition \ref{prop: ExtendTensoratorsByContinuity}, this provides a well-defined unitary isomorphism $\mathfrak{M}(p, \upsilon)\boxtimes\mathfrak{M}(n, \nu)\xrightarrow{\cong} \mathfrak{M}(p+n, \upsilon\times\nu)$ natural in $\int^{\oplus}_{t\in T}\delta_{p(t)}\mathrm{d}\upsilon(t)$ and $ \int^\oplus_{s\in S}\delta_{n(s)}\mathrm{d}\nu(s)$.
\end{proof}

\section{The category of twisted representations of the Heisenberg conformal net}
We now focus on the category $\Rep^{\sigma}(\Heis)$ of $\sigma$-twisted representations of $\Heis$, where the action of $\sigma$ is given by $\sigma W(f) = W(-f)$ for all $I\in\Jcal$ and $f\in L_I\R$. We first produce a fully faithful functor from $\Rep^\sigma(\Heis)$ into the $\mathrm{W}^*$-category $\Rep^{pe}(\widetilde{L^\sigma\R})$ of positive energy representations of the centrally extended twisted loop group $\widetilde{L^\sigma\R}$, see Proposition \ref{prop: Rsigmafactors}. The category $\Rep^{pe}(\widetilde{L^\sigma\R})$ is then identified with $\Hilb$ in Proposition \ref{prop: ReppeLSigmaIsHIlb}, implying that there exists at most one irreducible $\sigma$-twisted representation of $\Heis$ up to equivalence. The fact that there exists a nonzero $\sigma$-twisted representation of $\Heis$ \cite{fredenhagen} implies that the constructed functor $\Rep^\sigma(\Heis)\to \Rep^{pe}(\widetilde{L^\sigma\R})$ is an equivalence. We then characterize the fusion rules of honest representations of $\Heis$ with the unique $\sigma$-twisted representation, see Proposition \ref{prop: TensoratorForMSigma}. We use the terminology \emph{twisted representations} to refer to $\sigma$-twisted representations of $\Heis$.

\subsection{The restriction functor $\Rep^\sigma(\Heis)\to \Rep^{pe}(\widetilde{L^\sigma\R})$}
\label{Sec: CatOfTwistedReps}

Let $\mathfrak{t}: \mathrm{T}:= \R/2\Z\to S^1$, $t\mapsto e^{2\pi i t}$ be a double cover of the circle. We define the twisted loop group of $\R$ by
\[
L^\sigma \R:= \{f\in C^\infty(\mathrm{T}, \R)\ |\ f(x) = -f(x+1)\}.
\]
Analogously to the untwisted case, the twisted loop group $L^\sigma\R$ admits a cocycle $c(f,g) = e^{-\frac{i}{2}\int_0^{1} f(\theta)g'(\theta)\mathrm{d}\theta}$ for $f,g\in L^\sigma\R$, which provides a central extension 
\[
1\to U(1)\to \widetilde{L^\sigma\R}\to L^\sigma\R\to 1.
\]
For any interval $I\in\Jcal$, we write $L^\sigma_I\R$ for the subgroup of loops $f\in L^\sigma\R$ such that $\mathfrak{t}(\text{supp}(f))\subset I$ and we denote by $\widetilde{L^\sigma_I\R}:=\widetilde{L^\sigma\R}\times_{L^\sigma\R}L_I^\sigma\R$ the restriction of the central extension $\widetilde{L^\sigma\R}$ to $L^\sigma_I\R$, for $I\in\Jcal$.

Let $\Rep(\widetilde{L^\sigma\R})$ be the category of strongly continuous unitary representations of $\widetilde{L^\sigma\R}$ on separable Hilbert spaces. We now construct a restriction functor
\[
\mathfrak{R}_0^\sigma: \Rep^\sigma(\Heis)\to \Rep(\widetilde{L^\sigma\R}).
\]
Given an interval $I\in \Jcal$, we write $\hat{I}$ for the lift of $I$ to $\R$ as close to $0$ as possible so that $\text{cl}(\hat{I})\subset \R_{>0}$. We continue denoting by $\hat{I}$ the image of $\hat{I}\subset\R$ in $\mathrm{T}$.  For an element $f\in L_I\R$, we denote by $\hat{f}$ the unique element of $L_{I}^\sigma\R$ such that $\hat{f}$ restricted to $\hat{I}$ equals $f$. We write $\hat{f}_I$ when we want to make the interval $I$ explicit. On the other hand, given another interval $J\in \Jcal$ and $g\in L^\sigma_{J}\R$, we denote by $g_0 = g_{0,J}\in L_J\R$ the unique loop such that $\hat{g_0}_J = g$. See Figure~\ref{fig: TtoS^1}.
\begin{figure}[htb]
\resizebox{!}{4.5cm}{
\begin{tikzpicture}[
    line cap=round,
    line join=round,
    scale=1.15,
    cover/.style={black, line width=1.25pt},
    point/.style={circle, fill=black, inner sep=2.4pt},
    bump/.style={line width=1.55pt},
    lab/.style={font=\Large},
    interval/.style n args={1}{
        #1,
        line width=4.8pt,
        postaction={decorate},
        decoration={
            }
        }
]

\def\R{3.55}
\def\r{0.55}
\def\Y{2.35}
\def\pitch{1.22}

\def\Ja{-158}
\def\Jb{-98}
\def\Ia{-62}
\def\Ib{-18}

\def\gamp{-0.68}
\def\famp{0.55}


\newcommand{\BaseInterval}[3]{%
    \draw[interval={#3}, domain=#1:#2, samples=80, variable=\ang]
        plot ({\R*cos(\ang)}, {\r*sin(\ang)});
}

\newcommand{\LiftInterval}[3]{%
    \draw[interval={#3}, domain=#1:#2, samples=80, variable=\ang]
        plot (
            {\R*cos(\ang-90)},
            {\Y + \pitch*\ang/360 + \r*sin(\ang-90)}
        );
}

\newcommand{\BaseBump}[4]{%
    \draw[#4,bump, domain=#1:#2, samples=120, variable=\ang]
        plot (
            {\R*cos(\ang)},
            {\r*sin(\ang)
             + (#3)*sin(180*((\ang)-(#1))/((#2)-(#1)))}
        );
}

\newcommand{\LiftBump}[4]{%
    \draw[#4,bump, domain=#1:#2, samples=120, variable=\ang]
        plot (
            {\R*cos(\ang-90)},
            {\Y + \pitch*\ang/360 + \r*sin(\ang-90)
             + (#3)*sin(180*((\ang)-(#1))/((#2)-(#1)))}
        );
}


\draw[cover] (0,0) ellipse[x radius=\R, y radius=\r];

\draw[cover, domain=0:720, samples=420, variable=\ang]
    plot (
        {\R*cos(\ang-90)},
        {\Y + \pitch*\ang/360 + \r*sin(\ang-90)}
    );


\BaseBump{\Ja}{\Jb}{\gamp}{RedOrange}
\BaseBump{\Ia}{\Ib}{\famp}{RoyalBlue}


\LiftBump{292}{352}{\gamp}{RedOrange}
\LiftBump{652}{712}{0.68}{RedOrange}

\LiftBump{28}{72}{\famp}{RoyalBlue}
\LiftBump{388}{432}{-0.55}{RoyalBlue}


\BaseInterval{\Ja}{\Jb}{Red}
\BaseInterval{\Ia}{\Ib}{Blue}

\LiftInterval{292}{352}{Red}
\LiftInterval{28}{72}{Blue}


\node[point] at ({\R*cos(-90)}, {\r*sin(-90)}) {};
\node[point] at ({\R*cos(-90)}, {\Y+\r*sin(-90)}) {};


\node[lab] at (0.3,1.42) {\LARGE$0$};
\node[lab, below] at (.3,-.7) {\LARGE$\mathrm{p} = 1$};

\node[lab, black] at (-4.4,-0.15) {\huge$S^1$};
\node[lab, black] at (-4.4,4.65) {\huge$\mathrm{T}$};

\node[lab, Red] at (-2,0.1) {\LARGE$J$};
\node[lab, Blue] at (2.70,-1.05) {\LARGE$I$};

\node[lab, Red] at (-3.75,2.85) {\LARGE$\hat J$};
\node[lab, Blue] at (3.35,1.95) {\LARGE$\hat I$};

\node[lab, RedOrange] at (-1.25,-1.45) {\LARGE$g_0$};
\node[lab, RoyalBlue] at (1.10,0) {\LARGE$f$};

\node[lab, RedOrange] at (-2.0,3.8) {\LARGE$g$};
\node[lab, RoyalBlue] at (4,3.25) {{\LARGE$\hat f$}};
\end{tikzpicture}}
\caption{}
\label{fig: TtoS^1}
\end{figure}

Let $(K, \pi^K)\in \Rep^\sigma(\Heis)$ be a twisted representation. For every $I\in \Jcal$, we obtain an inclusion of groups
\[
\begin{array}{ccc}
    \widetilde{L^\sigma_I\R} & \to & U(\Heis(I)) \\
     (f, z) &\mapsto &z\cdot W(f_{0,I})
\end{array}
\]
satisfying the following compatibility condition. Recall from Section \ref{sec: CrossedBalancedRepGA} that, given an inclusion of intervals $J\subset I$ in $\Jcal$, we write $\delta_{J\subset I}^\sigma: \Heis(J)\hookrightarrow\Heis(I)$ for the inclusion $\Heis(J)\hookrightarrow\Heis(I)$ if $J\subset I$ is standard and for the composition $\Heis(J)\hookrightarrow\Heis(I)\xrightarrow{\sigma_I^{-1}}\Heis(I)$ if $J\subset I$ is special. It is clear that, given any inclusion $J\subset I$, the following square~commutes,
\[\begin{tikzcd}
	{\widetilde{L^\sigma_J\R}} &&& {U(\Heis(J))} \\
	{\widetilde{L^\sigma_I\R}} &&& {U(\Heis(I))},
	\arrow["{(f,z)\mapsto z\cdot W(f_{0,J})}", from=1-1, to=1-4]
	\arrow[hook, from=1-1, to=2-1]
	\arrow["{\delta^\sigma_{J\subset I}}", hook, from=1-4, to=2-4]
	\arrow["{(f, z)\mapsto z\cdot W(f_{0,I})}"', from=2-1, to=2-4]
\end{tikzcd}\]
where $\widetilde{L^\sigma_J\R}\hookrightarrow \widetilde{L^\sigma_I\R}$ is the obvious inclusion. Therefore, we obtain homomorphisms $\pi_I: \widetilde{L^\sigma_I\R}\xrightarrow{(z,f)\mapsto z\cdot W(f_{0,I})} U(\Heis(I))\xrightarrow{\pi^K_I} U(K)$ parametrized by $I\in\Jcal$ so that the outer diagram
\[\begin{tikzcd}
	{\widetilde{L^\sigma_J\R}} &&& {U(\Heis(J))} && \\
	&&&&& {U(K)} \\
	{\widetilde{L^\sigma_I\R}} &&& {U(\Heis(I))}
	\arrow["{(f,z)\mapsto z\cdot W(f_{0,J})}", from=1-1, to=1-4]
	\arrow[hook, from=1-1, to=3-1]
	\arrow["{\pi^K_J}", from=1-4, to=2-6]
	\arrow["{\delta^\sigma_{J\subset I}}", hook, from=1-4, to=3-4]
	\arrow["{(f,z)\mapsto z\cdot W(f_{0,I})}"', from=3-1, to=3-4]
	\arrow["{\pi^K_I}"', from=3-4, to=2-6]
\end{tikzcd}\]
commutes, as the square commutes by the discussion above and the triangle commutes by the definition of a $\sigma$-twisted representation. Therefore, we obtain a strongly continuous unitary representation of the colimit
\[
\colim_{I\in \Jcal} \widetilde{L^\sigma\R}
\]
on the Hilbert space $K$. By the same arguments as in \cite[Thm. 4]{Hen19}, the colimit $\colim_{I\in \Jcal} \widetilde{L^\sigma\R}$ is canonically equivalent to $\widetilde{L^\sigma\R}$, and hence we obtain a strongly continuous unitary representation of $\widetilde{L^\sigma\R}$ on $K$ so that $\widetilde{L^\sigma_I\R}\subset \widetilde{L^\sigma\R}$ acts by the composition $\widetilde{L^\sigma_I\R}\xrightarrow{(f, z)\mapsto z\cdot W(f_{0,I})}U(\Heis(I))\xrightarrow{\pi^K_I}U(K)$. Extending this assignment trivially at the level of morphisms, we obtain a functor
\[
\mathfrak{R}^\sigma_0: \Rep^\sigma(\Heis)\to \Rep(\widetilde{L^\sigma\R}).
\]
Since the $*$-algebra generated by $\{W(f_{0,I})\}_{f\in L^\sigma_I\R}$ in $\Heis(I)$ is dense, the functor $\mathfrak{R}^\sigma_0$ is fully faithful. We will next show that, as in the untwisted case, the essential image of $\mathfrak{R}^\sigma_0$ is the category of positive energy representations of $\widetilde{L^\sigma\R}$. Note that the double cover $U(1)^{(2)}$ of $U(1)$, seen as the group of rotations on $\mathrm{T}$, acts on $\widetilde{L^\sigma\R}.$

\begin{definition}
    A strongly continuous unitary representation $(K, \pi)$ of $\widetilde{L^\sigma\R}$ is \emph{a positive energy representation} if the action of $\widetilde{L^\sigma\R}$ on $K$ extends to an action of $U(1)^{(2)}\ltimes \widetilde{L^\sigma\R}$ so that the generator of the action of $U(1)^{(2)}$ has positive spectrum. We write $\Rep^{pe}(\widetilde{L^\sigma\R})\subset \Rep(\widetilde{L^\sigma\R})$ for the full subcategory on the positive energy representations.
\end{definition}
\begin{proposition}\label{prop: Rsigmafactors}
    The fully faithful $\mathrm{W}^*$-functor $\mathfrak{R}^\sigma_0: \Rep^\sigma(\Heis)\to \Rep(\widetilde{L^\sigma\R})$ factors through a fully faithful $\mathrm{W}^*$-functor $\mathfrak{R}^\sigma: \Rep^\sigma(\Heis)\to \Rep^{pe}(\widetilde{L^\sigma\R})$.
\end{proposition}
\begin{proof}
    Let $H\in \Rep^\sigma(\Heis)$ be a twisted representation. The same arguments as those leading to Proposition \ref{prop: R0Factors} imply that $\mathfrak{R}^\sigma(H)$ extends to a unitary representation of $U(1)^{(2)}\ltimes \widetilde{L^\sigma\R}$. The main difference with Proposition \ref{prop: R0Factors} is that the result in \cite{weiner} is only stated for honest representations of the conformal net. However, restricting $H$ to a representation of the subnet $\Heis^{\Z/2}\subset \Heis$ provides an honest representation of the conformal net $\Heis^{\Z/2}$. Then, we may apply the result in \cite{weiner} to this honest representation of $\Heis^{\Z/2}$ to show that the constructed action of $U(1)^{(2)}\ltimes \widetilde{L^\sigma\R}$ on $H$ is of positive energy, as in the untwisted case.
\end{proof}

\subsection{Characterizing $\Rep^\sigma(\Heis)$}

In the current section, we use the fully faithful $\mathrm{W}^*$-functor $\mathfrak{R}^\sigma: \Rep^\sigma(\Heis)\to \Rep^{pe}(\widetilde{L^\sigma\R})$ and our understanding of positive energy representations of twisted loop groups (\cite{LG}) to characterize $\Rep^\sigma(\Heis)$, see Theorem \ref{thm: RepSigmaHeisIsHilb}.

We claim that there is a unique irreducible positive energy representation of $\widetilde{L^\sigma\R}$ up to equivalence. We first construct this representation. Given a twisted loop $f\in L^\sigma\R$, we can~define 
\begin{equation}\label{eq: normTwisted}
||f||^2 = 2\pi \sum\limits_{k \in\mathbb{Z}_{\geq0} + 1/2} k\cdot ||\hat{f}_k||^2,
\end{equation}
where $\hat{f_k} = \int_0^{1} e^{-2\pi i k \theta}f(\theta){\text{d}\theta}$. Together with the complex structure $\mathcal{J}(\hat{f}_k) = -i\text{sgn}(k)\hat{f}_k$, the vector space $L^\sigma\R$ can be completed with respect to the norm in Equation \eqref{eq: normTwisted} to obtain the Hilbert space $\overline{L^\sigma\R}$. The inner product is obtained by polarization. For every twisted loop $f\in \widetilde{L^\sigma\R}$, we write $W^\sigma(f)\in U(e^{\overline{L^\sigma\R}})$ for the Weyl operator
\[
W^\sigma(f)(e^h) = e^{-\frac{1}{2}||f||^2-\langle f, h\rangle}e^{f+h}.
\]
It holds that $W^\sigma(f)W^\sigma(g) = e^{-i\omega(f,g)}W^\sigma(f+g)$, where 
\[
\omega(f,g) = \text{Im}\langle f, g\rangle = \frac{1}{2}\int_0^{1}f(\theta)g'(\theta) {\mathrm{d}\theta}=:\frac{1}{2}\int fg'.
\]
\begin{proposition}\label{prop: ReppeLSigmaIsHIlb}
    The pair $(e^{\overline{L^\sigma\R}}, W^\sigma(-))$ is an irreducible positive energy unitary representation of $\widetilde{L^\sigma\R}$. In addition, it is the only irreducible positive energy unitary representation of $\widetilde{L^\sigma\R}$. Therefore, the $\mathrm{\mathrm{W}^*}$-functor
    \[
    \mathfrak{L}^\sigma: \Hilb\cdot \tau\to \Rep^{pe}(\widetilde{L^\sigma\R})
    \]
    sending $\tau$ to $(e^{\overline{L^\sigma\R}}, W^\sigma(-))$ is an equivalence of $\mathrm{W}^*$-categories.
\end{proposition}
\begin{proof}
Note that $L^\sigma\R$ is a topological vector space and the form
\[
\omega(f,g) = \frac{1}{2}\int_0^{1} f(\theta)g'(\theta){\mathrm{d}\theta}
\]
is skew bilinear and nondegenerate. Indeed, if $g' \neq 0$, there exists a twisted loop $f\in L^\sigma\R$ such that $\int_0^{1} fg' \neq 0$. Hence, if $\omega(f,g) = 0$ for all $f\in L^\sigma\R$, then $g$ is necessarily constant in $(0,1)$, hence constant on $\mathrm{T}$. But the only constant function in $L^\sigma\R$ is the zero function. Thus, the central extension $\widetilde{L^\sigma\R}$ is a Heisenberg group in the sense of \cite[Def. 9.5.1]{LG} and by \cite[Prop. 9.5.10]{LG} it has a unique irreducible positive energy representation. Moreover, the unique irreducible positive energy representation of $\widetilde{L^\sigma\R}$ is the defining representation $(e^{\overline{L^\sigma\R}}, W^\sigma(-))$ \cite[Prop. 9.5.6 and 9.5.10]{LG}.
\end{proof}

We can now characterize the $\mathrm{W}^*$-category $\Rep^\sigma(\Heis)$. The representation $(e^{\overline{L^\sigma\R}}, W^\sigma(-))$ of $\widetilde{L^\sigma\R}$ induces $*$-actions of all the dense subalgebras $\langle W(f)\rangle_{f\in L_I\R}\subset \Heis(I)$ for $I\in\Jcal$. Indeed, given $I\in\Jcal$ and $f\in L_I\R$, we define the action of $W(f)$ on $e^{\overline{L^\sigma\R}}$ as $\pi^\sigma_I(W(f)):=W^\sigma(\hat{f}_{\hat{I}})$. These actions satisfy that, if $J\subset I$ is an inclusion in $\Jcal$ and $f\in L_J\R$ is a twisted loop, $\pi^\sigma_J(W(f)) = \pi_I^\sigma(\delta^\sigma_{J\subset I}W(f))$. Hence, if the actions of $\langle W(f)\rangle_{f\in L_I\R}$ extend to actions $\pi^\sigma_I$ of $\Heis(I)$ for all $I$, then we obtain a $\sigma$-twisted representation $\pi^\sigma$ of $\Heis$ on $e^{\overline{L^\sigma\R}}$.

\begin{theorem}\label{thm: RepSigmaHeisIsHilb}
The positive energy representation $(e^{\overline{L^\sigma\R}}, W^\sigma(-))$ of $\widetilde{L^\sigma\R}$ extends to a $\sigma$-twisted representation $H_0^\sigma = (e^{\overline{L^\sigma\R}}, \pi^\sigma)$ of $\Heis$. Therefore, the functor
\[
\mathfrak{M}^\sigma:\Hilb\cdot\tau\to \Rep^\sigma(\Heis)
\]
sending $\tau$ to $H_0^\sigma$ is an equivalence of $\mathrm{W}^*$-categories.
\end{theorem}
\begin{proof}
    Let $I_0\in\Jcal_\mathrm{p}$. By \cite[Sec. 5]{fredenhagen}, there exists a nonzero object in the category $\sigma-\Loc_{I_0}(\Heis)$. Since there is an equivalence of categories $\sigma-\Loc_{I_0}(\Heis)\cong \Rep^\sigma(\Heis)$ by \cite[Prop. 4.6]{GcrossedbraidedRep}, there exists a nonzero object $H^\sigma$ in $\Rep^\sigma(\Heis)$. By Proposition \ref{prop: Rsigmafactors}, we have a fully faithful restriction functor $\mathfrak{R}^\sigma: \Rep^\sigma(\Heis)\to \Rep^{pe}(\widetilde{L^\sigma\R})$, and $\mathfrak{R}^\sigma(H^\sigma)\in \Rep^{pe}(\widetilde{L^\sigma\R})$ is nonzero. Since $\Rep^{pe}(\widetilde{L^\sigma\R})\cong\Hilb$ by Proposition \ref{prop: ReppeLSigmaIsHIlb}, there exists a projection $\mathfrak{p}: \mathfrak{R}^\sigma(H^\sigma)\to \mathfrak{R}^\sigma(H^\sigma)$ in $\Rep^{pe}(\widetilde{L^\sigma\R})$ whose splitting is unitarily isomorphic to $(e^{\overline{L^\sigma\R}}, W^\sigma(-))\in \Rep^{pe}(\widetilde{L^\sigma\R})$. By fully faithfulness of $\mathfrak{R}^\sigma$, we can lift $\mathfrak{p}$ to an idempotent $\tilde{\mathfrak{p}}:  H^\sigma\to  H^\sigma$, whose splitting (which exists because $\Rep^\sigma(\Heis)$ is idempotent complete) produces an irreducible object $H^\sigma_0\in \Rep^\sigma(\Heis)$ whose image under $\mathfrak{R}^\sigma$ is unitarily isomorphic to $(e^{\overline{L^\sigma\R}}, W^\sigma(-))$. Therefore, the first claim follows. Since $\Rep^{pe}(\widetilde{L^\sigma\R})\cong \Hilb$ by Proposition \ref{prop: ReppeLSigmaIsHIlb}, the second statement also follows.
\end{proof}

\subsection{Tensoring with the twisted representation}

Recall that we denote by $H^\sigma_0 = (e^{\overline{L^\sigma\R}}, \pi^\sigma)$ the unique (up to isomorphism) irreducible $\sigma$-twisted representation of $\Heis$, which is the extension to a $\sigma$-twisted representation of $\Heis$ of the standard positive energy representation $(e^{\overline{L^\sigma\R}}, W^\sigma(-))$ of $\widetilde{L^\sigma\R}.$ The twisted representation $H^\sigma_0$ is also the image of $\tau\in \Hilb\cdot \tau$ under the functor $\mathfrak{M}^\sigma: \Hilb\cdot \tau\to \Rep^\sigma(\Heis)$ from Theorem \ref{thm: RepSigmaHeisIsHilb}, by construction. Since $\Rep^{\Z/2}(\Heis)$ is a $\Z/2$-graded tensor category, there exist functors $-\boxtimes H^\sigma_0: \Rep(\Heis)\to \Rep^\sigma(\Heis)$ and $ H^\sigma_0\boxtimes -: \Rep(\Heis)\to \Rep^\sigma(\Heis)$ given by tensoring with $H^\sigma_0$. We denote by $\text{Forget}: \Rep(C_0(\R))\to \Hilb$ the functor that sends a representation of $C_0(\R)$ to its underlying Hilbert space. The goal of this section is to prove the following proposition, which characterizes the fusion rules of honest representations of $\Heis$ with the twisted representation $H^\sigma_0$.

\begin{proposition}\label{prop: TensoratorForMSigma}
The equivalence of $\mathrm{W}^*$-categories $\mathfrak{M}\oplus\mathfrak{M}^\sigma: \Rep(C_0(\R))\oplus\Hilb\to \Rep^{\mathbb{Z}/2}(\Heis)$ admits unitary natural isomorphisms filling the diagram
\[\begin{tikzcd}
&& {\Rep(\Heis)} && \\
{\Rep^\sigma(\Heis)} && {\Rep(C_0(\R))} && {\Rep^\sigma(\Heis)} \\
&& \Hilb
\arrow[""{name=0, anchor=center, inner sep=0}, "{H^\sigma_0\boxtimes -}"', from=1-3, to=2-1]
\arrow[""{name=1, anchor=center, inner sep=0}, "{-\boxtimes H^\sigma_0}", from=1-3, to=2-5]
\arrow["{\mathfrak{M}}"', from=2-3, to=1-3]
\arrow["{\text{Forget}}", from=2-3, to=3-3]
\arrow[""{name=2, anchor=center, inner sep=0}, "{\mathfrak{M}^\sigma}", from=3-3, to=2-1]
\arrow[""{name=3, anchor=center, inner sep=0}, "{\mathfrak{M}^\sigma}"', from=3-3, to=2-5]
\arrow["{[H^\sigma_0,-]}"'{pos=0.3}, shift left=3, between={0.2}{0.8}, Rightarrow, from=0, to=2]
\arrow["{[-,H^\sigma_0]}"{pos=0.3}, shift right=3, between={0.2}{0.8}, Rightarrow, from=1, to=3]
\end{tikzcd}\]
\end{proposition}

We follow a similar strategy to Section \ref{sec: ProvidingTensorator}. We will first work towards providing the unitary natural isomorphism $[H_0^\sigma,-]$. The existence of $[-,H^\sigma_0]$ will then follow using the $\Z/2$-crossed braiding. We fix an interval $I_0\in \Jcal_{\mathrm{p}}$ and a loop $g\in L_{I_0}\R$ of charge 1. Recall from Definition \ref{def: FunctorMg} and Lemma \ref{Lemm: MgsFormClique} that we obtain a functor $\mathfrak{M}_g: \Rep(C_0(\R))\to \Rep(\Heis)$ naturally isomorphic to $\mathfrak{M}$. We first provide a unitary natural isomorphism $^{\text{irr}}[H_0^\sigma, -]$ filling
\begin{equation}\begin{tikzcd}\label{eq: 2cellIrrTwisted}
{\Rep^{\text{irr}}(C_0(\R))} && {\Rep^\sigma(\Heis)}.
\arrow[""{name=0, anchor=center, inner sep=0}, "{H_0^\sigma\boxtimes  \mathfrak{M}_{g}(-)}", curve={height=-18pt}, from=1-1, to=1-3]
\arrow[""{name=1, anchor=center, inner sep=0}, "{H_0^\sigma}"', curve={height=18pt}, from=1-1, to=1-3]
\arrow["{^{\text{irr}}[H_0^\sigma,-]_g}", shift right=4, between={0.2}{0.8}, Rightarrow, from=0, to=1]
\end{tikzcd}\end{equation}
Here, the bottom functor is the constant $\mathbb{C}$-linear functor with value $H_0^\sigma$.

Recall that there is a tensor equivalence $\mathfrak{E}: \mathbb{Z}/2-\Loc_{I_0}(\Heis)\to \Rep^{\mathbb{Z}/2}(\Heis)$, see Section \ref{sec: CrossedBalancedRepGA}. On the tensor subcategory $\Loc_{I_0}(\Heis)\subset \mathbb{Z}/2-\Loc_{I_0}(\Heis)$, this functor restricts to the equivalence (also denoted by $\mathfrak{E}$) that we have used in Section \ref{sec: ProvidingTensorator}. On the twisted localized endomorphisms $\sigma-\Loc_{I_0}(\Heis)\subset \mathbb{Z}/2-\Loc_{I_0}(\Heis)$, the functor sends a $\sigma$-twisted endomorphism $\rho$ to the $\sigma$-twisted representation $\mathfrak{E}(\rho)$ of $\Heis$ on $e^{\overline{L\R/\R}}$ where $W(f)\in \Heis(I)$ acts by $\pi_0\circ \rho(W(f))$ for $I\in \Jcal_{\mathrm{p}}$. Let us fix an endomorphism $\rho_\tau\in \sigma-\Loc_{I_0}(\Heis)$ such that $\mathfrak{E}(\rho_\tau)\cong H_0^\sigma$ and a unitary isomorphism
\[
\phi: \mathfrak{E}(\rho_\tau)\xrightarrow{\cong}H_0^\sigma
\]
of $\sigma$-twisted representations of $\Heis$. We write $(e^{\overline{L\R/\R}}, \pi_\tau):=\mathfrak{E}(\rho_\tau)$. The tensorator $\Psi$ for $\mathfrak{E}$ and the isomorphism $\phi$ allow us to construct unitary natural isomorphisms 
\begin{equation}\label{eq: 2Cells1}\begin{tikzcd}
& {\Loc_{I_0}(\Heis)} && {\sigma-\Loc_{I_0}(\Heis)} \\
	{\Rep^{\text{irr}}(C_0(\R))} \\
	& {\Rep(\Heis)} && {\Rep^\sigma(\Heis)}.
	\arrow["{\rho_\tau \circ-}", from=1-2, to=1-4]
	\arrow["{\mathfrak{E}}", from=1-2, to=3-2]
	\arrow["\Psi_{\rho_\tau,-}"', shift right=4, between={0.3}{0.7}, Rightarrow, from=1-4, to=3-2]
	\arrow["{\mathfrak{E}}"', from=1-4, to=3-4]
	\arrow["{\mathfrak{F}_g}", from=2-1, to=1-2]
	\arrow[""{name=0, anchor=center, inner sep=0}, "{\mathfrak{M}_g}"', from=2-1, to=3-2]
	\arrow[""{name=1, anchor=center, inner sep=0}, "{H_0^\sigma\boxtimes -}"', curve={height=12pt}, from=3-2, to=3-4]
	\arrow[""{name=2, anchor=center, inner sep=0}, "{\mathfrak{E}(\rho_\tau)\boxtimes -}", curve={height=-12pt}, from=3-2, to=3-4]
	\arrow["{\text{Id}}"', between={0.2}{0.7}, Rightarrow, from=0, to=1-2]
	\arrow["{\phi\boxtimes -}"', between={0.2}{0.8}, Rightarrow, from=2, to=1]
\end{tikzcd}\end{equation}
Therefore, to produce the unitary natural isomorphism $^{\text{irr}}[H_0^\sigma, -]_g: H_0^\sigma\boxtimes \mathfrak{M}_g(-)\xrightarrow{\cong} H_0^\sigma$ in Diagram \eqref{eq: 2cellIrrTwisted}, it is enough to produce a unitary natural isomorphism
\[\begin{tikzcd}
	{\Rep^{\text{irr}}(C_0(\R))} && {\Rep^\sigma(\Heis)} \\
	{\Loc_{I_0}(\Heis)} && {\sigma-\Loc_{I_0}(\Heis)}
	\arrow[""{name=0, anchor=center, inner sep=0}, "{H_0^\sigma}", from=1-1, to=1-3]
	\arrow[""{name=0p, anchor=center, inner sep=0}, phantom, from=1-1, to=1-3, start anchor=center, end anchor=center]
	\arrow["{\mathfrak{F}_g}", from=1-1, to=2-1]
	\arrow[""{name=1, anchor=center, inner sep=0}, "{\rho_\tau \circ-}"', from=2-1, to=2-3]
	\arrow[""{name=1p, anchor=center, inner sep=0}, phantom, from=2-1, to=2-3, start anchor=center, end anchor=center]
	\arrow["{\mathfrak{E}}"', from=2-3, to=1-3]
	\arrow[between={0.2}{0.8}, Rightarrow, from=1p, to=0p]
\end{tikzcd}\]
We will do this in two steps. Let us write $H^\sigma_g:\Rep^{\text{irr}}(C_0(\R))\to \Rep^\sigma(\Heis)$ for the $\mathbb{C}$-linear functor sending $\delta_x\in \Rep^\text{irr}(C_0(\R))$ to the twisted representation $H_0^\sigma$ precomposed with the family of automorphisms $\rho_{xg,I}:W(f)\mapsto e^{ix\int gf}W(f)$ for $I\in\Jcal$ and $f\in L_I\R$, as defined in Section \ref{Sec: ConstructingIrreps}. Then, the $\sigma$-twisted representation $H^\sigma_g(\delta_x)$ of $\Heis$ has underlying Hilbert space $e^{\overline{L^\sigma\R}}$ and, given $I\in \Jcal_{\mathrm{p}}$, an element $W(f)\in \Heis(I)$ acts on $e^{\overline{L^\sigma\R}}$ via 
\begin{equation}\label{eq: twistedRepEq1}
e^{i x \int f g}\cdot W^\sigma(\hat{f}_{I}).
\end{equation}

We will produce two unitary natural isomorphisms, as follows
\begin{equation}\label{eq: 2Cells2}\begin{tikzcd}
{\Rep^{\text{irr}}(C_0(\R))} && {\Rep^\sigma(\Heis)} \\
\\
{\Loc_{I_0}(\Heis)} && {\sigma-\Loc_{I_0}(\Heis)}.
\arrow[""{name=0, anchor=center, inner sep=0}, "{H_0^\sigma}", curve={height=-18pt}, from=1-1, to=1-3]
\arrow[""{name=1, anchor=center, inner sep=0}, "{H_g^\sigma}"', curve={height=18pt}, from=1-1, to=1-3]
\arrow["{\mathfrak{F}_g}", from=1-1, to=3-1]
\arrow[""{name=2, anchor=center, inner sep=0}, "{\rho_\tau \circ-}"', curve={height=12pt}, from=3-1, to=3-3]
\arrow["{\mathfrak{E}}"', from=3-3, to=1-3]
\arrow["{w_g}"', between={0.2}{0.8}, Rightarrow, from=1, to=0]
\arrow["{v_g}"', between={0.2}{0.7}, Rightarrow, from=2, to=1]
\end{tikzcd}\end{equation}
Given $\delta_x\in \Rep^{\text{irr}}(C_0(\R))$, we write
\[
v_g(\delta_x):\mathfrak{E}(\rho_\tau\circ \rho_{xg}) = (e^{\overline{L\R/\R}}, \pi_\tau\circ \rho_{xg})\xrightarrow{\phi}H^\sigma_g(\delta_x) = (e^{\overline{L^\sigma\R}}, \pi^\sigma \circ \rho_{xg}),
\]
which provides the unitary natural isomorphism $v_g$. It is only left to produce the unitary natural isomorphism $w$. Note that there is an alternative description of the formula \eqref{eq: twistedRepEq1} which will be useful. Let $\hat{g}\in L_{I_0}^\sigma \R$ be the unique twisted loop such that, when restricted to $\hat{I_0}$, it is given by $g$. Then, since both $I$ and $I_0$ are intervals in $\Jcal_\mathrm{p}$, we have
\[
\int \hat{f}\hat{g} = \int_0^{1} {f}(\theta)g(\theta){\mathrm{d}\theta}= \int fg,
\]
and therefore Equation \eqref{eq: twistedRepEq1} can be rewritten as
\begin{equation}\label{eq: twistedRepEq2}
e^{i x \int \hat{f} \hat{g}}\cdot W^\sigma(\hat{f}).
\end{equation}
We note that $\hat{g}$ admits a primitive $\hat{G}\in L^\sigma\R$, given by
\[
\hat{G}(\theta):=\int_0^\theta g(z)\mathrm{d}z -\frac{1}{2}\int_0^{1} g(z)\mathrm{d}z.
\]
Since $W^\sigma(-x\hat{G})W^\sigma(\hat{f}) = e^{ix\int\hat{G}\hat{f}'}W^\sigma(\hat{f})W^\sigma(-x\hat{G}) = e^{-ix\int\hat{g}\hat{f}}W^\sigma(\hat{f})W^\sigma(-x\hat{G})$, we deduce that $w_g(\delta_x):=W^\sigma(-x\hat{G})$ is a unitary equivalence between the $\sigma$-twisted $\Heis$-representations $H^\sigma_g(\delta_x)$ and $H_0^\sigma$, as needed. We have proved the following lemma.

\begin{lemma}
There is a natural unitary isomorphism $^{\text{irr}}[H_0^\sigma,-]_{g}$ between the functors
\[\begin{tikzcd}
{\Rep^{\text{irr}}(C_0(\R))} && {\Rep^\sigma(\Heis)}
\arrow[""{name=0, anchor=center, inner sep=0}, "{H_0^\sigma\boxtimes  \mathfrak{M}_{g}(-)}", curve={height=-18pt}, from=1-1, to=1-3]
\arrow[""{name=1, anchor=center, inner sep=0}, "{H_0^\sigma}"', curve={height=18pt}, from=1-1, to=1-3]
\arrow["{^{\text{irr}}[H_0^\sigma,-]_g}", shift right=5, between={0.2}{0.8}, Rightarrow, from=0, to=1]
\end{tikzcd}\]
which, on $\delta_x\in \Rep^{\text{irr}}(C_0(\R))$ is given by the composition
\[
\begin{array}{rcl}
^{\text{irr}}[H_0^\sigma,-]_g(\delta_x):=\ ^\text{irr}[H_0^\sigma,\delta_x]_g: H_0^\sigma\boxtimes \mathfrak{M}_g(\delta_x)&\fixedxrightarrow{\phi^{-1}\boxtimes\id}& \mathfrak{E}(\rho_\tau)\boxtimes \mathfrak{E}\circ\mathfrak{F}_g(\delta_x)\\
&\fixedxrightarrow{\Psi_{\rho_\tau, \rho_{xg}}} & \mathfrak{E}(\rho_\tau\circ \rho_{xg})\\
&\fixedxrightarrow{\phi}& H_g^\sigma(\delta_x)\\
&\fixedxrightarrow{W^\sigma(-x\hat{G})} &H_0^\sigma,
\end{array}
\]
where $\hat{G}\in L^\sigma \R$ is the twisted loop $\hat{G}(\theta) = \int_0^\theta g(z)\mathrm{d}z-\frac{1}{2}\int_0^{1}g(z)\mathrm{d}z$.
\end{lemma}
\begin{proof}
We define $^\text{irr}[H_0^\sigma,\delta_x]_g$ to be the compositions of the unitary natural isomorphisms in Diagrams \eqref{eq: 2Cells1} and \eqref{eq: 2Cells2}.
\end{proof}

By extending linearly by direct sums and multiplicities, we obtain a natural unitary isomorphism $^{\text{ss}}[-,-]_{g}$ filling the diagram
\[\begin{tikzcd}
{\Rep^{\text{ss}}(C_0(\R))} && {\Rep^\sigma(\Heis)}.
\arrow[""{name=0, anchor=center, inner sep=0}, "{H_0^\sigma\boxtimes  \mathfrak{M}_{g}(-)}", curve={height=-18pt}, from=1-1, to=1-3]
\arrow[""{name=1, anchor=center, inner sep=0}, "{H_0^\sigma}"', curve={height=18pt}, from=1-1, to=1-3]
\arrow["{^{\text{ss}}[H_0^\sigma,-]_g}", shift right=5, between={0.2}{0.8}, Rightarrow, from=0, to=1]
\end{tikzcd}\]
We now extend this unitary natural isomorphism by continuity to all of $\Rep(C_0(\R))$. Assume that $I_0\subset S_+^1$. Then, for any object $ \int^\oplus_{t\in T}\delta_{p(t)}\mathrm{d}\upsilon(t)\in \Rep(C_0(\R))$, there is a canonical unitary isomorphism of Hilbert spaces
\begin{align}
\text{Forget}(H_0^\sigma\boxtimes \mathfrak{M}_{g}(p, \upsilon)) &\fixedxrightarrow{=} \text{Forget}(H_0^\sigma(S^1_-)\boxtimes \mathfrak{M}_g(p,\upsilon)(S^1_+))\nonumber\\& \fixedxrightarrow{=} \text{Forget}(H_0^\sigma(S_-^1)\boxtimes (L^2(\upsilon)\cdot H_0)(S^1_+))\\&\nonumber\fixedxrightarrow{\cong} L^2(\upsilon)\otimes \text{Forget}(H_0^\sigma\boxtimes H_0)\\&\nonumber\fixedxrightarrow{\id\otimes\text{unitor}} L^2(\upsilon)\otimes \text{Forget}(H_0^\sigma)
\end{align}
between the underlying Hilbert space of $ H_0^\sigma\boxtimes \mathfrak{M}_g(p,\upsilon)$ and the underlying Hilbert space of $L^2(\upsilon)\cdot H_0^\sigma$. Moreover, this unitary does not depend on $p$. Hence, there is a homeomorphism
\begin{equation}\label{eq: UnderlyingHilbSpacesTwisted}
\mathfrak{U}^\sigma: U\big(H_0^\sigma\boxtimes \mathfrak{M}_g(p,\upsilon), L^2(T, \upsilon)\cdot H_0^\sigma\big) = U\big(L^2(T,\upsilon, e^{\overline{L^\sigma\R}})\big)
\end{equation}
independent of $p$, using the strong operator topologies on the source and the target. Recall that we write $\hat{G}\in L^\sigma\R$ for the twisted loop $\hat{G}(\theta) = \int_0^\theta g(z)\mathrm{d}z-\frac{1}{2}\int_0^{1}g(z)\mathrm{d}z$.

\begin{lemma}\label{lemm: TwistedTensoratorInSSSimplified}
Fix $p\in C_\sqcup(T)_0$. Under the homeomorphism
\begin{equation}\nonumber
\mathfrak{U}^\sigma: U\big(H_0^\sigma\boxtimes \mathfrak{M}_g(p,\upsilon), L^2(T, \upsilon)\cdot H_0^\sigma\big) = U\big(L^2(T,\upsilon, e^{\overline{L^\sigma\R}})\big),
\end{equation}
the unitary $^{\text{ss}}[H_0^\sigma,-]_g(p, \upsilon)\in U\big(H^\sigma_0\boxtimes \mathfrak{M}_g(p,\upsilon),L^2(T,\upsilon)\cdot H_0^\sigma\big)$ maps to the unitary on $L^2(T,\upsilon, e^{\overline{L^\sigma\R}})$ that sends a function $\psi\in L^2(T,\upsilon, e^{\overline{L^\sigma\R}})$ to 
\begin{align*}
\Big(t\mapsto &W^\sigma\big(-p(t)\cdot \hat{G}\big)\big(\psi(t)\big)\Big) \in L^2(T,\upsilon, e^{\overline{L^\sigma\R}}).
\end{align*}
\end{lemma}
\begin{proof} The proof is analogous to that of Lemma \ref{lemm: TensoratorInSSSimplified}. Let us first discuss the case where $ \int^\oplus_{t\in T}\delta_{p(t)}\mathrm{d}\upsilon(t)$ is an irreducible representation $\delta_x$. We need to compute the composition of maps of the underlying Hilbert spaces
\begin{equation}\label{eq: ProofLemmTwistedTensoratorInSSSimplified}
H_0^\sigma\xleftarrow{\text{unitor}} H_0^\sigma\boxtimes H_0 \xrightarrow{\phi^{-1}\boxtimes\id} \mathfrak{E}(\rho_{\tau})\boxtimes\mathfrak{E}(\rho_{xg})\xrightarrow{\Psi_{\rho_\tau, \rho_{xg}}} \mathfrak{E}(\rho_\tau\circ \rho_{xg})\xrightarrow{\phi}H^\sigma_g(\delta_x)\xrightarrow{W^\sigma(-x\hat{G})} H_0^\sigma.
\end{equation}
Here, $\Psi_{\rho_\tau, \rho_{xg}}$ is the same map of Hilbert spaces as the unitor $\mathfrak{E}(\rho_\tau)\boxtimes H_0\xrightarrow{\text{unitor}} \mathfrak{E}(\rho_\tau)$, see \cite[Thm. 3.8 and Sec. 4]{GcrossedbraidedRep}. By naturality of the unitor, the diagram
\[\begin{tikzcd}
{H_0^\sigma\boxtimes H_0} && {H_0^\sigma} \\
{\mathfrak{E}(\rho_\tau)\boxtimes H_0} && {\mathfrak{E}(\rho_\tau)}
\arrow["{\text{unitor}}", from=1-1, to=1-3]
\arrow["{\phi^{-1}\boxtimes \id}"', from=1-1, to=2-1]
\arrow["{\phi^{-1}}", from=1-3, to=2-3]
\arrow["{\text{unitor}}"', from=2-1, to=2-3]
\end{tikzcd}\]
commutes, and the composition \eqref{eq: ProofLemmTwistedTensoratorInSSSimplified} is exactly given by $W^\sigma(-x\hat{G}): e^{\overline{L^\sigma\R}}\to e^{\overline{L^\sigma\R}}$, as needed. Hence, the claim follows when $ \int^\oplus_{t\in T}\delta_{p(t)}\mathrm{d}\upsilon(t)$ is an irreducible object of $\Rep(C_0(\R))$. Extending the result linearly by multiplicities, the statement also holds when $p:T\to\R$ has image exactly one point in $\R$. Let us now extend it to direct sums.

Let $p\in C_\sqcup(T)_0$ be a map associated to a partition $T = \bigsqcup\limits_{\alpha\in A}T_\alpha$. By making the partition finer if necessary, we can assume that $p$ is constant on each $T_\alpha$. We write $p_\alpha$ for the restriction of $p$ to $T_\alpha$. Then, the statement holds for $p_\alpha$ for all $\alpha\in A$.

By definition, the tensorator $^{\text{ss}}[H_0^\sigma,-]_g(\upsilon)$ belongs to the subset
\[
\prod\limits_{\alpha\in A}U\big(H_0^\sigma\boxtimes \mathfrak{M}_{g}(p_\alpha, \upsilon_\alpha), L^2(T_\alpha, \upsilon_\alpha)\cdot H_0^\sigma\big)\subset U\big(H_0^\sigma\boxtimes \mathfrak{M}_{g}(\upsilon), L^2(T, \upsilon)\cdot H_0^\sigma\big)
\]
and $\mathfrak{U}^\sigma$ restricts to a homeomorphism
\[
\prod\limits_{\alpha\in A}U\big(H_0^\sigma\boxtimes\mathfrak{M}_{g}(p_\alpha, \upsilon_\alpha), L^2(T_\alpha, \upsilon_\alpha)\cdot H_0^\sigma\big)\cong \prod\limits_{\alpha\in A} U\big(L^2(T_\alpha, \upsilon_\alpha, e^{\overline{L^\sigma\R}})\big)\subset U\big(L^2(T, \upsilon, e^{\overline{L^\sigma\R}})\big).
\]
In fact, $\mathfrak{U}^\sigma$ preserves the grading given by $A$, and the restriction \[U\big(H_0^\sigma\boxtimes\mathfrak{M}_{g}(p_\alpha, \upsilon_\alpha), L^2(T_\alpha, \upsilon_\alpha)\cdot H_0^\sigma\big)\cong U\big(L^2(T_\alpha, \upsilon_\alpha, e^{\overline{L^\sigma\R}})\big)\] is again given by $\mathfrak{U}^\sigma$ in Equation \eqref{eq: UnderlyingHilbSpacesTwisted} for $\upsilon = \upsilon_\alpha$. Hence, the claim follows.
\end{proof}

Let $ \int^\oplus_{t\in T}\delta_{p(t)}\mathrm{d}\upsilon(t)\in\Rep(C_0(\R))$ be a representation and let $I_0\in \Jcal_{\mathrm{p}}$ be an interval with $I_0\subset S^1_+$. Fix $g\in L_{I_0}\R$ a loop of charge 1. Since the underlying Hilbert space of $\mathfrak{M}_g(p, \upsilon)\boxtimes H^\sigma_0$ is independent of $p:T\to \R$, we may write it as $\text{Forget}(\mathfrak{M}_g(\upsilon)\boxtimes H^\sigma_0)$. We also write $U\big(\mathfrak{M}_g(\upsilon)\boxtimes H^\sigma_0, L^2(T, \upsilon, e^{\overline{L^\sigma\R}})\big)$ for the space of unitaries between $\text{Forget}(\mathfrak{M}_g(\upsilon)\boxtimes H^\sigma_0)$ and $L^2(T, \upsilon, e^{\overline{L^\sigma\R}})$.

\begin{definition}
Let $T$ be a locally compact Hausdorff space with a Borel measure $\upsilon$. Let $I_0\in \Jcal_{\mathrm{p}}$ be an interval with $I_0\subset S^1_+$ and $g\in L_{I_0}\R$ be a loop of charge 1. Let $p\in C_\sqcup(T)$ be a function. We write $[H_0^\sigma,(p, \upsilon)]_g\in U\big(H_0^\sigma\boxtimes \mathfrak{M}_{g}(\upsilon), L^2(T, \upsilon)\cdot H_0^\sigma\big)$ for the image under the homeomorphism
\[
(\mathfrak{U}^\sigma)^{-1}:U(L^2(T, \upsilon, e^{\overline{L^\sigma\R}}))\cong U\big(H_0^\sigma\boxtimes \mathfrak{M}_{g}(\upsilon), L^2(T, \upsilon)\cdot H_0^\sigma\big)
\]
of the unitary sending $\psi\in L^2(T,\upsilon, e^{\overline{L^\sigma\R}})$ to
\begin{align*}
\Big(t\mapsto & W^\sigma\big(-p(t)\cdot \hat{G})\big(\psi(t)\big)\Big) \in L^2(T,\upsilon, e^{\overline{L^\sigma\R}}),
\end{align*}
where $\hat{G}(\theta) = \int_0^\theta g(z)\mathrm{d}z-\frac{1}{2}\int_0^{1}g(z)\mathrm{d}z$.
\end{definition}

We can show that the family $[-,-]_g$ is a natural family of unitary isomorphisms of $\sigma$-twisted representations of $\Heis$.

\begin{proposition}\label{prop: TwistedExtendTensoratorsByContinuity}
Let $T$ be a locally compact Hausdorff space with a Borel measure $\upsilon$. Let $I_0\in \Jcal_{\mathrm{p}}$ be an interval with $I_0\subset S^1_+$ and $g\in L_{I_0}\R$ be a loop of charge 1. Let $p\in C_\sqcup(T)$ be a function. Then, 
\begin{enumerate}
\item $[H_0^\sigma,(p, \upsilon)]_g\in U\big(H_0^\sigma\boxtimes \mathfrak{M}_{g}(\upsilon), L^2(T, \upsilon)\cdot H_0^\sigma\big)$ intertwines the twisted actions of $\Heis$ on $H_0^\sigma\boxtimes \mathfrak{M}_{g}(p,\upsilon)$ and $L^2(T, \upsilon)\cdot H_0^\sigma$;
\item The family $[H_0^\sigma,-]_g$ is natural with respect to morphisms in $\Rep(C_0(\R))$.
\end{enumerate}
\end{proposition}
\begin{proof}
This proof is analogous to that of Proposition \ref{prop: ExtendTensoratorsByContinuity}. We write $\pi^\sigma$ for the action of $\Heis$ on $H_0^\sigma$, and $\pi_{g}^{\sigma\boxtimes p}$ for the action of $\Heis$ on $H_0^\sigma\boxtimes\mathfrak{M}_g(p, \upsilon)$. Given $I\in \Jcal$ and $a\in \Heis(I)$, we need to show that
\[
\big(\id_{L^2(T,\upsilon)}\otimes \pi^\sigma(a)\big)\circ [H_0^\sigma,(p, \upsilon)]_g = [H_0^\sigma,(p, \upsilon)]_g\circ \pi_{g, I}^{\sigma\boxtimes p}(a).
\]
Since the actions of $\Heis(I)$ on $H_0^\sigma$ and $H_0^\sigma\boxtimes\mathfrak{M}_g(p, \upsilon)$ are normal, it is enough to show the equality for all $a \in \langle W(f)\rangle_{f\in L_I\R}$. Let $a = W(f)$ for some $f\in L_I\R$. Then, the equation holds if $p\in C_\sqcup(T)_0$, since then $[H_0^\sigma,(p, \upsilon)]_g = \,^{\text{ss}}[H^\sigma_0, (p,\upsilon)]_g$ by Lemma \ref{lemm: TwistedTensoratorInSSSimplified}, and the latter is by construction, a morphism in $\Rep^\sigma(\Heis)$. Recall that $C_\sqcup(T)_0$ is dense in $C_\sqcup(T)$ by Lemma \ref{lemm: dense} and note that the map
\[
\begin{array}{ccccc}
C_\sqcup(T) &\to &U(L^2(T, \upsilon, e^{\overline{L^\sigma\R}})) &\xrightarrow{(\mathfrak{U}^\sigma)^{-1}}&U(H_0^\sigma\boxtimes\mathfrak{M}_{g}(\upsilon), L^2(\upsilon)\cdot H_0^\sigma) \\
p&\mapsto & \Big(t\mapsto W^\sigma\big(-p(t)\cdot \hat{G}\big)\Big)&\mapsto & [H_0^\sigma,(p, \upsilon)]_g
\end{array}
\]
is strongly continuous, as it is the composition of strongly continuous maps. By completely analogous arguments to the second statement in Lemma \ref{lemma: HeisActionsDependContinuously}, the map
\[
\begin{array}{ccc}
C_\sqcup(T) &\to &U(H_0^\sigma\boxtimes \mathfrak{M}_{g}(\upsilon))  \\
p&\mapsto &\pi_{g, I}^{\sigma\boxtimes p}(W(f)) 
\end{array}
\]
is strongly continuous, and hence the first claim follows. The second part follows from the same arguments as the second part of Proposition \ref{prop: ExtendTensoratorsByContinuity}.
\end{proof}

We can now prove the main result of this section, Proposition \ref{prop: TensoratorForMSigma}. Let us recall the statement.

\begin{proposition*}
The equivalence of $\mathrm{W}^*$-categories $\mathfrak{M}\oplus\mathfrak{M}^\sigma: \Rep(C_0(\R))\oplus\Hilb\to \Rep^{\mathbb{Z}/2}(\Heis)$ admits unitary natural isomorphisms filling the diagram
\[\begin{tikzcd}
&& {\Rep(\Heis)} && \\
{\Rep^\sigma(\Heis)} && {\Rep(C_0(\R))} && {\Rep^\sigma(\Heis)} \\
&& \Hilb
\arrow[""{name=0, anchor=center, inner sep=0}, "{H^\sigma_0\boxtimes -}"', from=1-3, to=2-1]
\arrow[""{name=1, anchor=center, inner sep=0}, "{-\boxtimes H^\sigma_0}", from=1-3, to=2-5]
\arrow["{\mathfrak{M}}"', from=2-3, to=1-3]
\arrow["{\text{Forget}}", from=2-3, to=3-3]
\arrow[""{name=2, anchor=center, inner sep=0}, "{\mathfrak{M}^\sigma}", from=3-3, to=2-1]
\arrow[""{name=3, anchor=center, inner sep=0}, "{\mathfrak{M}^\sigma}"', from=3-3, to=2-5]
\arrow["{[H^\sigma_0,-]}"'{pos=0.3}, shift left=3, between={0.2}{0.8}, Rightarrow, from=0, to=2]
\arrow["{[-,H^\sigma_0]}"{pos=0.3}, shift right=3, between={0.2}{0.8}, Rightarrow, from=1, to=3]
\end{tikzcd}\]
\end{proposition*}
\begin{proof}[Proof of Proposition \ref{prop: TensoratorForMSigma}]
Let $I_0\in \Jcal_{\mathrm{p}}$ be an interval such that $I_0\subset S_+^1$ and let $g\in L_{I_0}\R$ be a loop of charge 1. Given $\int^\oplus_{t\in T}\delta_{p(t)}\mathrm{d}\upsilon(t)\in \Rep(C_0(\R))$, we define
\[
[H_0^\sigma,-](p,\upsilon):=[H_0^\sigma,(p,\upsilon)]:= [H_0^\sigma,(p,\upsilon)]_g\circ (\id\boxtimes \kappa_{g}(p, \upsilon)),
\]
where $\kappa$ was introduced in Lemma \ref{Lemm: MgsFormClique}. By Lemma \ref{Lemm: MgsFormClique} and Proposition \ref{prop: TwistedExtendTensoratorsByContinuity}, this provides a well-defined unitary isomorphism $H_0^\sigma\boxtimes\mathfrak{M}(p, \upsilon)\to L^2(T, \upsilon)\cdot H^\sigma_0$ natural in $\int^\oplus_{t\in T}\delta_{p(t)}\mathrm{d}\upsilon(t)$. For $[-,H^\sigma_0]$, we use the $\Z/2$-crossed braiding $\mathbb{B}$ on $\Rep^{\Z/2}(\Heis)$, to define
\[
[(p,\upsilon),H_0^\sigma]: \mathfrak{M}(p,\upsilon)\boxtimes H_0^\sigma\xrightarrow{\mathbb{B}_{\mathfrak{M}(p,\upsilon), H_0^\sigma}} H_0^\sigma\boxtimes \mathfrak{M}(p,\upsilon)\xrightarrow{[H_0^\sigma,(p,\upsilon)]}L^2(T, \upsilon)\cdot H_0^\sigma.
\]
\end{proof}

\subsection{The representation category $\Rep^{\Z/2}(\Heis)$ as a Tambara-Yamagami category}
\label{sec: RepZ/2AsTY}
We can combine the results of the previous sections to conclude the following theorem.

\begin{theorem}\label{Thm: RepZ2HeisIsTYR}
The category $\Rep^{\mathbb{Z}/2}(\Heis)$ of twisted and untwisted representations of the Heisenberg conformal net under the action of $\mathbb{Z}/{2}$ is equivalent to a Tambara-Yamagami $\mathrm{W}^*$-tensor category for $\mathbb{R}$ in the sense of Definition \ref{def: WTY}. 
\end{theorem}
\begin{proof}
By Theorems \ref{thm: RepHeisIsRepR} and \ref{thm: RepSigmaHeisIsHilb}, we obtain an equivalence of $\mathrm{W}^*$-categories
\[  \mathfrak{M}\oplus\mathfrak{M}^\sigma:\Rep(C_0(\R))\oplus\Hilb\cdot\tau\xrightarrow{\cong}\Rep(\Heis)\oplus\Rep^\sigma(\Heis).
\]
Endowing $\Rep(C_0(\R))$ with the tensor structure of $\Hilb\,\R$, by Propositions \ref{prop: TensoratorForM} and \ref{prop: TensoratorForMSigma}, we have unitary natural isomorphisms
\[\begin{tikzcd}
{\Rep(C_0(\R))\times \Rep(C_0(\R))} && {\Rep(\Heis)\times \Rep(\Heis)} \\
{\Rep(C_0(\R))} && {\Rep(\Heis)}
\arrow["{\mathfrak{M}\times \mathfrak{M}}", from=1-1, to=1-3]
\arrow["\otimes"', from=1-1, to=2-1]
\arrow["{[-,-]}"{description}, Rightarrow, from=1-3, to=2-1]
\arrow["\boxtimes", from=1-3, to=2-3]
\arrow["{\mathfrak{M}}"', from=2-1, to=2-3]
\end{tikzcd}\]
\[\begin{tikzcd}
	&& {\Rep(\Heis)} && \\
	{\Rep^\sigma(\Heis)} && {\Rep(C_0(\R))} && {\Rep^\sigma(\Heis)} \\
	&& \Hilb.
	\arrow[""{name=0, anchor=center, inner sep=0}, "{\mathfrak{M}^\sigma(\tau)\boxtimes -}"', from=1-3, to=2-1]
	\arrow[""{name=1, anchor=center, inner sep=0}, "{-\boxtimes \mathfrak{M}^\sigma(\tau)}", from=1-3, to=2-5]
	\arrow["{\mathfrak{M}}", from=2-3, to=1-3]
	\arrow["{\text{Forget}}"', from=2-3, to=3-3]
	\arrow[""{name=2, anchor=center, inner sep=0}, "{\mathfrak{M}^\sigma}", from=3-3, to=2-1]
	\arrow[""{name=3, anchor=center, inner sep=0}, "{\mathfrak{M}^\sigma}"', from=3-3, to=2-5]
	\arrow["{[H_0^\sigma,-]}"', shift left=4, between={0.2}{0.8}, Rightarrow, from=0, to=2]
	\arrow["{[-,H_0^\sigma]}", shift right=2, between={0.2}{0.8}, Rightarrow, from=1, to=3]
\end{tikzcd}\]
Furthermore, by \cite[Thm. 3.32]{GcrossedbraidedRep}, we know that $\Rep(\Heis)\oplus\Rep^\sigma(\Heis)$ is a $\mathbb{Z}/2$-graded tensor category, meaning that $\Hom_{\Rep(\Heis)\oplus\Rep^\sigma(\Heis)}(H_0^\sigma\boxtimes H_0^\sigma, H_0^\sigma) = 0$, and $H_0^\sigma\boxtimes H^\sigma_0\neq 0$ \cite[Sec. 5]{fredenhagen}. Hence, the claim follows.
\end{proof}

\begin{corollary}
    There exist a continuous symmetric nondegenerate bicharacter $\chi: \R\times \R\to U(1)$ and a sign $\xi\in\{\pm1\}$ such that
    \[
    \Rep(\Heis)\oplus\Rep^{\sigma}(\Heis)\cong \TYcal(\R, \chi, \xi)
    \]
    as $\mathrm{W}^*$-tensor categories.
\end{corollary}
\begin{proof}
    This result is the combination of Theorem \ref{Thm: RepZ2HeisIsTYR} and the classification of Tambara-Yamagami $\mathrm{W}^*$-tensor categories \cite[Thm. 4.17]{AM2025}, see also Theorem \ref{thm: TYClassificationThm}.
\end{proof}

\section{Getting the signs right}\label{sec: GettingSignsRight}

In Theorem \ref{Thm: RepZ2HeisIsTYR}, we have identified $\Rep^{\Z/2}(\Heis):=\Rep(\Heis)\oplus\Rep^\sigma(\Heis)$ as a Tambara-Yamagami $\mathrm{W}^*$-tensor category for $\R$ in the sense of Definition \ref{def: WTY}. By \cite[Thm. 4.17]{AM2025}, we know that these are classified, up to equivalence of $\mathrm{W}^*$-tensor categories, by a continuous symmetric nondegenerate bicharacter $\chi: \R\times \R\to U(1)$ and a sign $\xi\in\{\pm1\}$, where the bicharacter is taken up to the action of $\Aut(\R)$ (see Theorem \ref{thm: TYClassificationThm}). There exist, up to this action, two continuous symmetric nondegenerate bicharacters on $\R$, namely
\[
\chi_+(x,y) = e^{i xy}\hspace{2cm}\chi_-(x,y) = e^{-i xy}.
\]
In the current section, we prove that $$\Rep^{\Z/2}(\Heis)\cong \TYcal(\R, \chi_-, +1)$$ as $\mathrm{W}^*$-tensor categories, see Theorem \ref{thm: MainThmCrossed}. In addition, we will compute the structure of a $\Z/2$-crossed braided $\mathrm{W}^*$-tensor category. To do so, we use \cite[Thm. 3.32]{GcrossedbraidedRep}, which states that $\Rep^{\Z/2}(\Heis)$ is actually a $\Z/2$-crossed balanced $\mathrm{W}^*$-tensor category. We will classify all possible such structures on Tambara-Yamagami $\mathrm{W}^*$-tensor categories for $\R$, see Proposition \ref{prop: ClassifyTwists}. Using this classification, we will show that it is enough to compute the square of the crossed balance on $\tau$ in order to determine the $\Z/2$-crossed braided tensor structure on a Tambara-Yamagami $\mathrm{W}^*$-tensor category for $\R$. The square of the crossed balance on the unique irreducible twisted representation of $\Heis$ is computable using Vertex Operator Algebra techniques (see Proposition \ref{prop: DoubleTwistComputation}), and hence this will allow us to obtain the complete $\Z/2$-crossed braided tensor structure on $\Rep^{\Z/2}(\Heis)$. Unfortunately, we only obtain the crossed balance on $\tau$ up to a sign, which does not determine the whole $\Z/2$-crossed balanced tensor structure up to equivalence, see the discussion leading to Conjecture \ref{Conjecture: WhichCrossedBalance}.

\subsection{$\Z/2$-crossed balances on Tambara-Yamagami $\mathrm{W}^*$-tensor categories for $\R$}

We first classify all $\Z/2$-crossed balanced structures on Tambara-Yamagami $\mathrm{W}^*$-tensor categories for $\R$. Fix $a\in \R\setminus\{0\}$ and set
\[
\begin{array}{cccc}
    \chi: & \R\times\R &\to &U(1)\\
     & (x,y)&\mapsto& e^{iaxy}.
\end{array}
\]
Let $\xi\in\{\pm 1\}$ be a sign. Let $\Sigma: \Rep(C_0(\R))\to \Rep(C_0(\R))$ denote the functor $$\Sigma\Big(\int^\oplus_{t\in T}\delta_{p(t)}\mathrm{d}\upsilon(t)\Big)=\int^\oplus_{t\in T}\delta_{-p(t)}\mathrm{d}\upsilon(t),$$ 
for all $\int^\oplus_{t\in T}\delta_{p(t)}\mathrm{d}\upsilon(t)\in\Rep(C_0(\R))$, extended trivially to morphisms. We also denote $\int^\oplus_{t\in T}\delta_{-p(t)}\mathrm{d}\upsilon(t)$ by $(-p, \upsilon)$. We are only interested in actions of $\Z/2$ on $\TYcal(\R,\chi, \xi)$ where the action of $\sigma\in \Z/2$ on $\Rep(C_0(\R))$ is given by $\Sigma$. Such a choice of action is justified by the following lemma. Recall that we write $T_\sigma: \Rep^{\Z/2}(\Heis)\to \Rep^{\Z/2}(\Heis)$ for the action of $\sigma\in\Z/2$.

\begin{lemma}\label{lemm: ActionTransportedToC0R} There is a unitary natural isomorphism
\[\begin{tikzcd}
	{\Rep(C_0(\R))} & {\Rep(C_0(\R))} \\
	{\Rep(\Heis)} & {\Rep(\Heis)}.
	\arrow["\Sigma", from=1-1, to=1-2]
	\arrow["{\mathfrak{M}}"', from=1-1, to=2-1]
	\arrow[Rightarrow, from=1-2, to=2-1]
	\arrow["{\mathfrak{M}}", from=1-2, to=2-2]
	\arrow["{T_\sigma}"', from=2-1, to=2-2]
\end{tikzcd}\]
\end{lemma}
\begin{proof}
Let $\int^\oplus_{t\in T}\delta_{p(t)}\mathrm{d}\upsilon(t)\in \Rep(C_0(\R))$. We define the necessary natural transformation on $\int^\oplus_{t\in T}\delta_{p(t)}\mathrm{d}\upsilon(t)$ to be the unitary $\mathfrak{M}(-p,\upsilon)\cong T_\sigma( \mathfrak{M}(p,\upsilon))$
\[
\begin{array}{ccc}
L^2(\upsilon)\otimes e^{\overline{L\R/\R}}&\to &L^2(\upsilon)\otimes e^{\overline{L\R/\R}}\\
\psi\otimes e^h&\mapsto &\psi\otimes e^{-h}.
\end{array} 
\]
It is easy to see, chasing the definitions, that this is indeed a unitary equivalence of representations of $\Heis$. Naturality is clear.
\end{proof}

The classification of $\Z/2$-crossed balanced structures on $\TYcal(\R, \chi,\xi)$ will follow \cite{galindo}. Let us write $\lambda$ for the Lebesgue measure on $\R$, that is, its Haar measure, so that $\int^\oplus_{x\in \R}\delta_{x}\mathrm{d}\lambda(x)$ denotes the representation $L^2(\R)\in \Rep(C_0(\R))$. The following result is analogous to \cite[Lemm. 4.8]{galindo}. 

\begin{proposition}\label{prop: twoactions}
There exist, up to equivalence, two $\mathbb{Z}/2$-actions on $\TYcal(\R, \chi, \xi)$ such that $\sigma\in \Z/2$ acts as $\Sigma$ on $\Rep(C_0(\R))$. For both of them, the non-trivial tensor functor data is
\[
\sigma(\tau\otimes\tau) = \sigma\Big(\int^\oplus_{x\in \R}\delta_{x}\mathrm{d}\lambda(x)\Big) = \int^\oplus_{x\in \R}\delta_{-x}\mathrm{d}\lambda(x) \xrightarrow{f(x)\mapsto f(-x)} \int^\oplus_{x\in \R}\delta_{x}\mathrm{d}\lambda(x) = \tau\otimes\tau = \sigma(\tau)\otimes\sigma(\tau).
\]
For the first one, the data $\eta: \sigma\circ\sigma\xrightarrow{\cong} \text{Id}$ is trivial, and for the second one, it is given by
\[
\eta_\upsilon = \id_{\int^\oplus_{t\in T}\delta_{p(t)}\mathrm{d}\upsilon(t)}\hspace{2cm}\eta_\tau = -\id_\tau
\]
for all objects $\int^\oplus_{t\in T}\delta_{p(t)}\mathrm{d}\upsilon(t)\in \Rep(C_0(\R))$.
\end{proposition}
\begin{proof}
Let $\underline{\Aut_\otimes}(\TYcal(\R, \chi, \xi))$ be the category of $\mathrm{W}^*$-tensor automorphisms of $\mathcal{TY}(\R, \chi, \xi)$ and unitary $\mathrm{W}^*$-tensor natural transformations. Then, $\underline{\Aut_\otimes}(\TYcal(\R, \chi, \xi))$ is a coherent 2-group in the sense of \cite[Def. 7]{2gps}. The monoidal category $\underline{\Z/2}$ whose objects are $\{1, \sigma\}$ and whose morphisms are identities is also a coherent 2-group by equipping its set of objects with the group structure of $\Z/2$. In addition, the group of automorphisms of the unit $\text{Id}_{\mathcal{TY}(\R, \chi, \xi)}\in \underline{\Aut_\otimes}(\TYcal(\R, \chi, \xi))$ is equivalent to $\Z/2.$ Indeed, let $\zeta:\text{Id}_{\mathcal{TY}(\R, \chi, \xi)}\to \text{Id}_{\mathcal{TY}(\R, \chi, \xi)}$ be a unitary $\mathrm{W}^*$-tensor natural transformation. The data of $\zeta$ is the data of functions $\zeta_\upsilon\in L^\infty(T, \upsilon, U(1))$ for every $\int^\oplus_{t\in T}\delta_{p(t)}\mathrm{d}\upsilon(t)\in \Rep(C_0(\R))$ and $\zeta_\tau\in U(1)$ satisfying naturality and
\[
\zeta_{\upsilon\times\nu}(t,s) = \zeta_{\upsilon}(t)\zeta_\nu (s)\hspace{1cm} \zeta_\upsilon(t)\zeta_\tau = \zeta_\tau\hspace{1cm}\zeta_\tau^2 = \zeta_\lambda
\]
for every $\int^\oplus_{t\in T}\delta_{p(t)}\mathrm{d}\upsilon(t), \int^\oplus_{s\in S}\delta_{n(s)}\mathrm{d}\nu(s)\in \Rep(C_0(\R))$. Hence, $\zeta_\upsilon = 1$ and $\zeta_\tau = \pm1$.

By \cite[Thm. 43]{2gps} or \cite[Prop. 5.4]{MR2763944}, the set of $\mathrm{W}^*$-tensor $\Z/2$-actions on $\TYcal(\R, \chi, \xi)$ up to isomorphism such that $\sigma$ acts as $\Sigma$ is isomorphic to the underlying set~of
\[
H^2(\Z/2; \Z/2)\cong \Z/2.
\]
Since both actions described in the statement of the lemma are indeed inequivalent $\mathrm{W}^*$-tensor $\Z/2$-actions on $\TYcal(\R, \chi, \xi)$, the claim follows.
\end{proof}

We can now characterize all $\mathbb{Z}/2$-crossed braidings on $\mathcal{TY}(\R, \chi, \xi)$. The following result is analogous to \cite[Thm. 4.9]{galindo} with the extra remark in \cite[Rk. 3.4]{galindo2024computinggcrossedextensionsorbifolds}.

\begin{proposition}\label{prop: braidingsTY}
There exists, up to equivalence, a unique $\Z/2$-crossed braiding on $\TYcal(\R, \chi, \xi)$ such that $\sigma\in \Z/2$ acts on $\Rep(C_0(\R))$ by $\Sigma$.  Let $w\in L^\infty(\lambda)$ be the function $w(x) = e^{-iax^2/2}$~and $$\epsilon:= +\sqrt{\xi}\cdot  e^{-i\cdot \text{sgn}(a)\cdot \frac{\pi}{8}}\in U(1)$$
for $\text{sgn}(a):=a/|a|$ and $+\sqrt{1} = 1$ and $+\sqrt{-1} = i$. The crossed braiding is given by the action of $\Z/2$ with $\eta_\tau = 1$ in Proposition \ref{prop: twoactions}~and
\[
\begin{array}{cccc}
\beta_{\upsilon, \nu} :&L^2(\upsilon\times\nu)&\to &L^2(\upsilon\times\nu)\nonumber\\\nonumber
&f(t,s)&\mapsto &\chi(p(t), n(s))f(t,s) \vspace{.5cm}\\
\beta_{\upsilon, \tau} = \beta_{\tau, \label{Eq: braiding}\upsilon}:&L^2(\upsilon)\cdot \tau&\to& L^2(\upsilon)\cdot\tau\\\nonumber
&f(t)&\mapsto& w(p(t))f(t)\vspace{.5cm}\\\nonumber
\beta_{\tau, \tau}: &L^2(\lambda) &\to &L^2(\lambda)\\\nonumber
& f(x)&\mapsto & \epsilon\, w(x)^{-1}f(x) 
\end{array}
\]
for objects $\int^{\oplus}_{t\in T}\delta_{p(t)}\mathrm{d}\upsilon(t),\,\int^{\oplus}_{s\in S}\delta_{n(s)}\mathrm{d}\nu(s)\in\Rep(C_0(\R)).$
\end{proposition}
\begin{proof}
By Proposition \ref{prop: twoactions} there are two actions of $\Z/2$ on $\TYcal(\R, \chi, \xi)$ where $\sigma$ acts by $\Sigma$ on $\Rep(C_0(\R))$, and they differ by $\eta_\tau = \pm \id_\tau$. It is easy to see that the data in the statement of the Proposition indeed defines a $\Z/2$-crossed braiding on $\TYcal(\R, \chi, \xi).$ Let $\beta$ denote an arbitrary $\Z/2$-crossed braiding on $\TYcal(\R, \chi, \xi)$. 

Fix objects $\int^\oplus_{t\in T}\delta_{p(t)}\mathrm{d}\upsilon(t),\, \int^\oplus_{s\in S}\delta_{n(s)}\mathrm{d}\nu(s)\in \Rep(C_0(\R))$. By the same arguments as in \cite[Lemm. 4.10 and 4.11]{AM2025}, the fact that $\beta$ is a natural family of morphisms of $C_0(\R)$-representations implies that $\beta_{\upsilon,\nu}$, $\beta_{\tau, \tau}$ are given by multiplication by functions $b(\upsilon, \nu)\in L^\infty(T\times S, \upsilon\times\nu, U(1))$ and $b\in L^\infty(\R,\lambda, U(1))$ respectively; and $\beta_{\upsilon, \tau}$ and $\beta_{\tau, \upsilon}$ are given by multiplication by functions $b_l(\upsilon),b_r(\upsilon)\in L^\infty(T,\upsilon, U(1))$ respectively. 

Let us denote by $D_1(X,Y,Z)$ and $D_2(X,Y,Z)$ the diagrams \eqref{eq: Xbraiding1} and \eqref{eq: Xbraiding2}. By $D_1(\upsilon, \tau, \nu)$, we find that
\[
b(\upsilon, \nu)(t, s) = \chi(p(t), n(s)).
\]
Also, $D_1(\tau, \tau, \upsilon)$ and $D_2(\upsilon, \tau, \tau)$ read
\[
\chi(p(t), x) = \frac{b(x)}{b(x-p(t))b_r(\upsilon)(t)}
\hspace{1cm}
\chi(p(t), x) = \frac{b(x)}{b(x-p(t))b_l(\upsilon)(t)} 
\]
respectively. Hence, $b_l(\upsilon) = b_r(\upsilon)$. Similarly, $D_1(\tau, \upsilon, \tau)$ and $D_2(\tau, \upsilon, \tau)$ read
\[
\chi(p(t), x) = \frac{b(x+p(t)) b_r(\lambda)(t)}{b(x)}
\hspace{1cm}
\chi(p(t), x) = \frac{b(x+p(t)) b_l(-\lambda)(t)}{b(x)},
\]
where $b_r(\lambda)\in L^\infty(\R, \lambda, U(1))$ denotes the function giving the crossed-braiding between $\tau$ and $L^2(\R) = \int^{\oplus}_{x\in \R}\delta_x\mathrm{d}\lambda(x)$, and $b_l(-\lambda)\in L^\infty(\R, \lambda, U(1))$ denotes the function giving the crossed-braiding between $ \int^{\oplus}_{x\in \R}\delta_{-x}\mathrm{d}\lambda(x)$ and $\tau$. The functions $b_l(\lambda)$ and $b_r(-\lambda)$ are defined analogously. We obtain
\[
b_r(\lambda) = b_l(\lambda) = b_r(-\lambda) = b_l(-\lambda).
\]
Next, equation $D_1(\tau, \upsilon, \nu)$ reads 
\begin{equation}\label{D1tauupsilonnu}
\chi(p(t), n(s)) = \frac{b_r(\upsilon)(t)b_r(\nu)(s)}{b_r(\upsilon\times\nu)(t,s)},
\end{equation}
where $b_r(\upsilon\times\nu)\in L^\infty(T\times S, \upsilon\times\nu, U(1))$ is the function giving the crossed braiding between $\tau$ and $\int^\oplus_{(t, s)\in T\times S}\delta_{p(t) + n(s)}\mathrm{d}(\upsilon\times \nu)(t, s)$. Setting $\int^\oplus_{t\in T}\delta_{p(t)}\mathrm{d}\upsilon(t) = \int^\oplus_{s\in S}\delta_{n(s)}\mathrm{d}\nu(s) = \int^\oplus_{x\in \R}\delta_{x}\mathrm{d}\lambda(x)$ in Equation \eqref{D1tauupsilonnu}, and using that $b_r(\lambda\times\lambda)(x,y) = b_r(\lambda)(x+y)$ by naturality, we obtain
\[
\chi(x,y) = \frac{b_r(\lambda)(x)b_r(\lambda)(y)}{b_r(\lambda)(x+y)}.
\]
Let $\chi^{1/2}:\R\times\R\to U(1)$ denote the map $\chi^{1/2}(x,y):=e^{iaxy/2}$. Let $B(x):=b_r(\lambda)(x)\chi^{1/2}(x,x)\in L^\infty(\R, \lambda, U(1))$. Then, $B(x)B(y) = B(x+y)$ as functions in $L^\infty(\R\times \R, \lambda\times\lambda, U(1))$. Therefore, by \cite[Cor. 5.3]{ramsay}, $B$ can be represented by a genuine measurable homomorphism, and by \cite{sasvari} it is continuous. Hence, $b_r(\lambda)$ is continuous. Therefore, $b_r(\lambda)$ is a continuous quadratic refinement of $e^{-iax^2}$ such that $b_r(\lambda)(x) = b_r(-\lambda)(x) = b_r(\lambda)(-x)$ and it is thus equal to
\[
b_r(\lambda)(x) = \chi^{1/2}(x,x)^{-1} = e^{-iax^2/2} = w(x).
\]
To compute $b_r(\upsilon)$ we note that, by naturality, $b_r(\upsilon\times\lambda)(t, x) = b_r(\lambda)(x+p(t))$ and therefore
\[
b_r(\upsilon)(t) =\chi(p(t), x)\frac{b_r(\lambda)(x+p(t))}{b_r(\lambda)(x)}  = w(p(t)).
\]

It is only left to characterize $b \in L^\infty(\R,\lambda, U(1))$. By $D_1(\tau, \tau, \tau)$, it holds that 
\begin{equation}\label{eq: D1(t,t,t)}
\mathcal{F}\Big(w(-x)\mathcal{F}(f)(-x)\Big)(x) = \xi\, b(x)\,\mathcal{F}\big(b(x)f(x)\big)(x)
\end{equation}
for every $f\in L^2(\R,\lambda)$, where $\mathcal{F}$ denotes the unitary Fourier transform induced by the bicharacter $\chi(x,y) = e^{iaxy}$. Denoting by $M_\omega, M_b\in B(L^2(\R))$ the operators given by pointwise multiplication by $\omega$ and $b$ respectively, Equation \eqref{eq: D1(t,t,t)} reads
\[
\mathcal{F}\circ M_\omega\circ \mathcal{F}^{-1} = \xi\cdot M_b\circ \mathcal{F}\circ M_b,
\]
using that $\omega(-x)= \omega(x)$ and $\mathcal{F}(f)(-x) = \mathcal{F}^{-1}(f)(x)$ for all $f\in L^2(\R, \lambda).$ Computing the kernels on both sides and equating them, together with the Fresnel integral $\int_{-\infty}^{\infty}e^{-iax^2/2}dx = \sqrt{\frac{2\pi}{|a|}}e^{-i\cdot\text{sgn}(a)\cdot \frac{\pi}{4}}$, it is straightforward to see that the only two solutions to Equation \eqref{eq: D1(t,t,t)} are
\[
b(x):=\pm\sqrt{\xi}\cdot  e^{-i\cdot\text{sgn}(a)\cdot\frac{\pi}{8} }\cdot w(x)^{-1} = \pm \sqrt{\xi}\cdot  e^{-i\cdot\text{sgn}(a)\cdot\frac{\pi}{8} }\cdot\, e^{iax^2/2}.
\]
Finally, $D_2(\tau,\tau,\tau)$~reads
\begin{equation}\label{eq: D2(t,t,t)}
\eta_\tau\cdot \xi\cdot \mathcal{F}\big(w(-x)\mathcal{F}(f)(-x)\big) = b(x)\mathcal{F}\big(b(-x)f(x)\big)
\end{equation}
for all $f\in L^2(\R, \lambda)$, where we use that $\mathcal{F}^{-1}(f)(x) = \mathcal{F}(f)(-x)$. Since $b(x) = b(-x)$, Equation \eqref{eq: D2(t,t,t)} holds if and only if $\eta_\tau = 1$, and the only $\Z/2$-action that admits a $\Z/2$-crossed braiding is the one with $\eta_\tau = 1$.

At this point, we have shown that any $\Z/2$-crossed braiding on $\TYcal(\R, \chi, \xi)$ is equivalent to the braiding given by the statement of the Proposition for either $\epsilon$ or $-\epsilon$. We claim that both choices lead to equivalent $\Z/2$-crossed braided $\mathrm{W}^*$-tensor categories. Let $\TYcal(\R, \chi, \xi)_+$ and $\TYcal(\R, \chi, \xi)_-$ be the $\Z/2$-crossed braided $\mathrm{W}^*$-tensor categories obtained by using $\epsilon$ and $-\epsilon$ respectively. Consider the identity $\mathrm{W}^*$-tensor functor $\TYcal(\R, \chi, \xi)_+ \to \TYcal(\R, \chi, \xi)_-$. This can be upgraded to an equivalence of $\Z/2$-crossed braided $\mathrm{W}^*$-tensor categories by defining the compatibility with the $\Z/2$-action to be
\[
\begin{array}{ccc}
\sigma\Big(\int^\oplus_{t\in T}\delta_{p(t)}\mathrm{d}\upsilon(t)\Big)& \to & \sigma\Big(\int^\oplus_{t\in T}\delta_{p(t)}\mathrm{d}\upsilon(t)\Big) \\
f&\mapsto & f\vspace{5mm}\\
\sigma(\tau) &\xrightarrow{-1\cdot \id_{\sigma(\tau)}}&\sigma(\tau)
\end{array}
\]
for any object $\int^\oplus_{t\in T}\delta_{p(t)}\mathrm{d}\upsilon(t)\in\Rep(C_0(\R))$. The diagram \eqref{eq: Gcrossedfunctor} for $X = Y = \tau$ becomes the equality
\[
\epsilon\, w(x)^{-1} = (-1)(-\epsilon)\, w(x)^{-1}, 
\]
and the claim follows.
\end{proof}

We finally classify the $\Z/2$-crossed balances admitted by the unique (up to equivalence) $\mathbb{Z}/2$-crossed braided structure on $\mathcal{TY}(\R, \chi, \xi)$, following \cite[Prop. 4.12]{galindo}.

\begin{proposition}\label{prop: ClassifyTwists}
There exist, up to equivalence, two $\Z/2$-crossed balances on the unique (up to equivalence) $\Z/2$-crossed braided structure on $\mathcal{TY}(\R, \chi, \xi)$. These are classified by a sign $\zeta\in \{\pm1\}$ and are given by
\[
\begin{array}{cccc}
\theta_{\upsilon} :&L^2(\upsilon)&\to &L^2(\upsilon)\\
&f(t)&\mapsto &w(p(t))^{-2}f(t) \vspace{.5cm}\\
\theta_{\tau}:&\tau&\xrightarrow{\zeta/\epsilon\cdot\id_\tau} &\tau
\end{array}
\]
for every object $\int^\oplus_{t\in T}\delta_{p(t)}\mathrm{d}\upsilon(t)\in\Rep(C_0(\R))$ and $w(x) = e^{-iax^2/2}.$
\end{proposition}
\begin{proof}
It is straightforward to see that the formulas in the statement of the Proposition provide a $\Z/2$-crossed balance on $\TYcal(\R, \chi, \xi)$. Let $\theta$ be a $\Z/2$-crossed balance on $\TYcal(\R, \chi, \xi)$. Then, the crossed balance $\theta_\tau$ on $\tau$ is given by $\theta_\tau = \tilde{\zeta}\cdot \id_\tau$ for some $\tilde{\zeta}\in U(1)$, as $\tau$ is simple. Also, given $\int^\oplus_{t\in T}\delta_{p(t)}\mathrm{d}\upsilon(t)\in \Rep(C_0(\R))$, the balance $\theta_\upsilon$ is given by multiplication by a function $u(\upsilon)\in L^\infty(T, \upsilon, U(1))$. Following this notation, Diagram \eqref{eq: Xtwist} for $X = \upsilon$ and $Y = \tau$ implies that
\[
u(\upsilon)(t) = w(-p(t))^{-2} = w(p(t))^{-2},
\]
and for $X = Y = \tau$ it gives
\[
u(\lambda)(t) = \tilde{\zeta}^2\epsilon^2w(t)^{-2}.
\]
Since $u(\lambda)(t) = w(t)^{-2}$, we find that
\[
\tilde{\zeta} = \pm\epsilon^{-1},
\]
as needed.
\end{proof}

This concludes the classification of $\Z/2$-crossed balanced structures on Tambara-Yamagami $\mathrm{W}^*$-tensor categories for $\R$. 

\subsection{Identifying the $\Z/2$-crossed braided structure} The goal of this section is to determine a continuous symmetric nondegenerate bicharacter $\chi:\R\times\R\to U(1)$ and a sign $\xi\in\{\pm 1\}$ such that
\[
\Rep^{\Z/2}(\Heis)\cong \TYcal(\R,\chi,\xi)
\]
as $\Z/2$-crossed braided $\mathrm{W}^*$-tensor categories. Using the characterization of crossed balanced structures on $\TYcal(\R, \chi,\xi)$ from the previous section, we will argue that, in order to determine $\chi$ and $\xi$, it is only necessary to compute the composition
\begin{equation}\label{eq: CompositionSquareTwistTau}
H^\sigma_0\xrightarrow{\theta_{H^\sigma_0}}T_\sigma (H^\sigma_0)\xrightarrow{T_\sigma(\theta_{H^\sigma_0})} T_\sigma\circ T_\sigma (H^\sigma_0) \cong T_{1}(H^\sigma_0) = H^\sigma_0
\end{equation}
in $\End_{\Rep^\sigma(\Heis)}(H_0^\sigma)$. Since $H_0^\sigma$ is simple, this composition is of the form $z\cdot\id_{H^\sigma_0}$ for some $z\in U(1)$. We will use that the action of $\Z/2$ on the $\mathrm{W}^*$-category $\Rep^{\Z/2}(\Heis)$ satisfies that $T_\sigma\circ T_\sigma = T_1 = \text{Id}_{\Rep^{\Z/2}}(\Heis)$, as defined in \cite{GcrossedbraidedRep}, and that $T_\sigma$ acts trivially on~morphisms.

Let $G$ be a discrete group acting on a conformal net $\A$. Recall from Section \ref{sec: CrossedBalancedRepGA} that we define the crossed balance on an object $H\in \Rep^{G}(\A)$ to be given by the action $e^{-2\pi i L_0}$ of $\widetilde{\exp}(-2\pi i L_0)\in \widetilde{\Mob}$ on $H$, where $L_0$ is the generator of rotations in $\Mob$. The computation of the composition in Equation \eqref{eq: CompositionSquareTwistTau} will come from VOA techniques. We refer the reader to \cite{MR1651389, MR2023933, MR2082709} for standard textbooks on VOA theory and to \cite{MR3119224, MR3796433} for their application in the formalization of 2D Conformal Field Theory. In \cite{henriques2025conformalnetassociatedunitary}, the authors construct, associated to any conformal net $\A$, a unitary VOA $V_\A$. In addition, to every representation $(K,\pi)$ of $\A$ on which $L_0$ acts with discrete spectrum and finite-dimensional eigenspaces, they associate a $V_\A$-module $M_K$. On any such $M_K$, one can construct two actions of $\widetilde{\Mob}$, one coming from the fact that $M_K$ is a $V_\A$-module (using VOA techniques) and one coming from the action of $\widetilde{\Mob}$ on $K$. Then, \cite[Cor. 9.6]{henriques2025conformalnetassociatedunitary} states, in particular, that both these actions agree.

Since the result in \cite[Cor. 9.6]{henriques2025conformalnetassociatedunitary} is only stated for untwisted representations, we will apply it to the VOA associated to the conformal net $\Heis^{\Z/2}$, the $\Z/2$-fixed points of the Heisenberg conformal net, as opposed to to twisted representations of the VOA associated to~$\Heis$.

\begin{proposition}\label{prop: DoubleTwistComputation}
In the $\Z/2$-crossed balanced $\mathrm{W}^*$-tensor category $\Rep^{\Z/2}(\Heis)$, the composition
\[
H_0^\sigma\xrightarrow{\theta_{H_0^\sigma}}T_\sigma (H_0^\sigma)\xrightarrow{T_\sigma(\theta_{H_0^\sigma})} T_\sigma\circ T_\sigma (H_0^\sigma) = T_{1}(H_0^\sigma) = H_0^\sigma,
\]
is equal to $e^{-\frac{\pi i}{4}}\cdot \id_{H_0^\sigma}$.
\end{proposition}
\begin{proof}
By \cite[Thm. 4.19]{fixedpoints}, there is an equivalence of balanced $\mathrm{W}^*$-tensor categories
\[
\mathfrak{R}: (\Rep^{\Z/2}(\Heis))^{\Z/2}\cong \Rep(\Heis^{\Z/2}),
\]
see also Section \ref{sec: CrossedBalancedRepGA}. Since the action $T_\sigma$ of $\sigma$ on $\Rep^{\Z/2}(\Heis)$ is degree-preserving, and $H_0^\sigma = \mathfrak{M}^\sigma(\tau)$ is the unique irreducible object in degree $\sigma$ up to isomorphism, there exists an isomorphism $\phi: T_\sigma(H_0^\sigma)\to H_0^\sigma$. Furthermore, we can pick $\phi$ so that the composition $H_0^\sigma = T_{\sigma}\circ T_\sigma(H_0^\sigma) \xrightarrow{T_\sigma(\phi)} T_\sigma(H_0^\sigma)\xrightarrow{\phi}H_0^\sigma$  is equal to $\id_{H_0^\sigma}$. The object $H_0^\sigma$ then produces two non-isomorphic simple objects of $(\Rep^{\Z/2}(\Heis))^{\Z/2}$, namely $H_+^\sigma:=(H^\sigma_0, \{\id_{H^\sigma_0}, \phi\})$ and $H_-^\sigma:=(H^\sigma_0, \{\id_{H^\sigma_0}, -\phi\})$. Denoting by $\theta_{H_0^\sigma}: H_0^\sigma\to T_\sigma(H_0^\sigma)$ the $\Z/2$-crossed balance on $H_0^\sigma$, the balances of $H^\sigma_+$ and $H^\sigma_-$ are
\[
\theta^{\Z/2}_{H^\sigma_+} = \phi\circ \theta_{H_{0}^\sigma} \hspace{2cm} \theta^{\Z/2}_{H^\sigma_-} = -\phi\circ \theta_{H_{0}^\sigma}. 
\]
Under the balanced equivalence $\mathfrak{R}: (\Rep^{\Z/2}(\Heis))^{\Z/2}\cong \Rep(\Heis^{\Z/2})$ described in Section \ref{sec: CrossedBalancedRepGA}, $H^\sigma_+$ and $H^\sigma_-$ produce two honest irreducible $\Heis^{\Z/2}$-representations $\mathfrak{R}(H^\sigma_+)$ and $\mathfrak{R}(H^\sigma_-)$, and their balances satisfy
\[
\theta_{\mathfrak{R}(H^\sigma_+)} = \mathfrak{R}(\theta^{\Z/2}_{H^\sigma_+}) = \mathfrak{R}(\phi\circ \theta_{H^\sigma_0}) \hspace{2cm} \theta_{\mathfrak{R}(H^\sigma_-)} =\mathfrak{R}(\theta^{\Z/2}_{H^\sigma_-}) = -\mathfrak{R}(\phi\circ \theta_{H^\sigma_0}). 
\]
The balances $\theta_{\mathfrak{R}(H^\sigma_\pm)}$ are given, by definition, by the action $e^{-2\pi i L_0}$ of $\widetilde{\exp}(-2\pi i L_0)\in\widetilde{\Mob}$ on $\mathfrak{R}(H^\sigma_\pm)$. By \cite[Cor. 9.6]{henriques2025conformalnetassociatedunitary}, the spectrum of $L_0$ on these representations is equal to the spectrum of $L_0$ on the induced modules of the VOA associated to the conformal net $\Heis^{\Z/2}$. It is well-known (see, for example, \cite[Sec. 3.4]{MR1694542}) that the lowest weight of these representations are spectra are $1/16$ for one of the representations $\mathfrak{R}(H^\sigma_\pm)$ and $9/16$ for the other. Hence, there exists a sign $\varpi\in \{\pm1\}$ such that $\theta_{\mathfrak{R}(H^\sigma_+)} = \mathfrak{R}(\theta^{\Z/2}_{H^\sigma_+}) = \varpi\cdot e^{-\frac{\pi i }{8}}\cdot \id$ and $\theta_{\mathfrak{R}(H^\sigma_-)} =\mathfrak{R}(\theta^{\Z/2}_{H^\sigma_-}) = - \varpi\cdot e^{-\frac{\pi i }{8}}\cdot\id$, where we abuse the notation by identifying the identity morphism of objects in different categories. Hence, it holds that 
\begin{equation}\label{eq: IntermediateStepCharacterizeTtiwstSquared}\phi\circ \theta_{H_0^\sigma} = \varpi\cdot e^{-\frac{\pi i}{8}}\cdot \id_{H_0^\sigma}.\end{equation}
Consider now the following diagram
\[\begin{tikzcd}
{H_0^\sigma} & {T_\sigma(H_0^\sigma)} && {T_\sigma\circ T_\sigma(H_0^\sigma)} & { T_1(H_0^\sigma)} & {H_0^\sigma} \\
&&& {T_\sigma(H_0^\sigma)} \\
& {H_0^\sigma}
\arrow["{\theta_{H_0^\sigma}}", from=1-1, to=1-2]
\arrow["{\varpi\cdot e^{-\frac{\pi i}{8}}\cdot \id}"', from=1-1, to=3-2]
\arrow["{T_\sigma(\theta_{H_0^\sigma})}", from=1-2, to=1-4]
\arrow["\phi", from=1-2, to=3-2]
\arrow["{=}", from=1-4, to=1-5]
\arrow["{T_\sigma(\phi)}", from=1-4, to=2-4]
\arrow["{=}", from=1-5, to=1-6]
\arrow["\phi"{pos=0.3}, from=2-4, to=1-6]
\arrow["{\varpi\cdot e^{-\frac{\pi i}{8}}\cdot \id}"', curve={height=24pt}, from=3-2, to=1-6]
\arrow["{\theta_{H_0^\sigma}}", from=3-2, to=2-4]
\end{tikzcd}\]
The left-most and bottom-right triangles commute by Equation \eqref{eq: IntermediateStepCharacterizeTtiwstSquared}. The top-right diagram commutes by the choice of $\phi$. The middle square commutes because $T_\sigma$ is the identity on morphisms and $\theta$ is a natural family of isomorphisms. Hence, we find that
\[
T_\sigma(\theta_{H_{0}^\sigma})\circ \theta_{H^\sigma_0} = e^{-\frac{\pi i}{4}}\cdot \id_{H_0^\sigma},
\]
as needed.
\end{proof}

We can now state and prove the main theorem of this section. Recall that we write $\chi_-$ for the bicharacter on $\R$ given by $\chi_-(x,y)  = e^{-ixy}$. Let us also remind the definition of the $\mathrm{W}^*$-category $\TYcal(\R, \chi_-, +1)$ from Definition \ref{def: TYcal(G,chi,xi)}. It is the $\mathrm{W}^*$-tensor category on the $\mathrm{W}^*$-category
\(
\Rep(C_0(\R))\oplus \Hilb\cdot\tau
\)
with, for any objects $\int^\oplus_{t\in T}\delta_{p(t)}\mathrm{d}\upsilon(t)$, $\int^\oplus_{s\in S}\delta_{n(s)}\mathrm{d}\nu(s)$, and $\int^\oplus_{r\in R}\delta_{z(r)}\mathrm{d}\eta(r)$ of $\Rep(C_0(\R))$, $$\upsilon\otimes \nu = \upsilon\times\nu\hspace{2cm}\upsilon\otimes\tau = \tau\otimes\upsilon = L^2(\upsilon)\cdot \tau\hspace{2cm}\tau\otimes \tau = L^2(\R,\lambda),$$ and associators given by
\[
\begin{array}{cccc}
\alpha_{\upsilon, \nu, \eta}:& \upsilon\times\nu\times\eta&\xrightarrow{\id}& \upsilon\times\nu\times \eta\vspace{.3cm}\\
\alpha_{\tau, \upsilon,\nu} = \alpha_{\upsilon, \nu, \tau} :& L^2(\upsilon\times \nu)\cdot \tau& \xrightarrow{\id} &L^2(\upsilon\times\nu )\cdot \tau\vspace{.3cm}\\
\alpha_{\upsilon,\tau,\nu}  :& L^2(\upsilon\times \nu)\cdot \tau& \to &L^2(\upsilon\times\nu )\cdot \tau\\ 
&f(t,s)&\mapsto& e^{- i p(t) n(s)}f(t,s)\vspace{.3cm}\\
\alpha_{\upsilon,\tau,\tau}:&L^2( \upsilon)\cdot \lambda& \to &\upsilon \times\lambda\\
&f(t, x)&\mapsto& f\big(t, p(t)+x\big)\vspace{.3cm}\\
\alpha_{\tau,\tau, \upsilon}:& \lambda\times\upsilon& \to &L^2(\upsilon)\cdot \lambda\\
&f(x,t)&\mapsto& f\big(t,-p(t)+x\big)\vspace{.3cm}\\
\alpha_{\tau,\upsilon,\tau}:& L^2(\upsilon)\cdot \lambda &\to & L^2(\upsilon)\cdot \lambda\\
&f(t,x)&\mapsto &e^{-i p(t)x} f(t,x)
\end{array}
\]
and 
\[
\begin{array}{cccc}
\alpha_{\tau,\tau,\tau} :& L^2(\R,\lambda)\cdot \tau&\to&L^2(\R,\lambda)\cdot \tau\\
& f&\mapsto &\Big(x\mapsto  \frac{1}{\sqrt{2\pi}}\int_{y\in \R} e^{i   xy}f(y)\text{d}\lambda\Big).
\end{array}
\]

\begin{theorem}\label{thm: MainThmCrossed}
The $\Z/2$-crossed braided $\mathrm{W}^*$-tensor category $\Rep^{\mathbb{Z}/2}(\Heis)$ of twisted and untwisted representations of the Heisenberg conformal net under the action of $\mathbb{Z}/{2}$ is equivalent to the following $\mathbb{Z}/2$-crossed braided $\mathrm{W}^*$-tensor category. The underlying $\mathrm{W}^*$-tensor category~is
\[
\TYcal(\R, \chi_-, +1).
\]
The action of $\Z/2 = \{1,\sigma\}$ is given by $\sigma\big(\int^{\oplus}_{t\in T}\delta_{p(t)}\mathrm{d}\upsilon(t)\big) = \int^{\oplus}_{t\in T}\delta_{-p(t)}\mathrm{d}\upsilon(t)$ and $\sigma(\tau) = \tau$, with the only non-trivial structure data $\sigma(\tau\otimes \tau) = \sigma\big(\int^{\oplus}_{x\in \R}\delta_{x}\mathrm{d}\lambda(x)\big) = \int^{\oplus}_{x\in \R}\delta_{-x}\mathrm{d}\lambda(x)\xrightarrow{f(x)\mapsto f(-x)}\int^{\oplus}_{x\in \R}\delta_{x}\mathrm{d}\lambda(x) = \tau\otimes\tau = \sigma(\tau)\otimes\sigma(\tau)$. The $\Z/2$-crossed braiding is given by 
\[
\begin{array}{cccc}
\beta_{\upsilon, \nu} :&L^2(\upsilon\times\nu)&\to &L^2(\upsilon\times\nu)\\
&f(t,s)&\mapsto & e^{-ip(t)n(s)}f(t,s) \vspace{.5cm}\\
\beta_{\upsilon, \tau} = \beta_{\tau, \upsilon}:&L^2(\upsilon)\cdot \tau&\to& L^2(\upsilon)\cdot\tau\\
&f(t)&\mapsto& e^{ip(t)^2/2}f(t)\vspace{.5cm}\\
\beta_{\tau, \tau}: &L^2(\lambda) &\to &L^2(\lambda)\\
& f(x)&\mapsto & e^{\pi i/8 } \,e^{-ix^2/2}f(x),
\end{array}
\]     
for objects $\int^{\oplus}_{t\in T}\delta_{p(t)}\mathrm{d}\upsilon(t)$ and $\int^{\oplus}_{s\in S}\delta_{n(s)}\mathrm{d}\nu(s)$ of $\Rep(C_0(\R))$.
\end{theorem}
\begin{proof}
Let $\chi_\pm:\R\times\R\to U(1)$ be the continuous symmetric nondegenerate bicharacters on $\R$ given by
\[
\chi_{\pm}(x,y) = e^{\pm i xy}.
\]
Every continuous symmetric nondegenerate bicharacter on $\R$ is equivalent to either $\chi_+$ or $\chi_-$ under the action of $\Aut(\R)$, and $\chi_+$ and $\chi_-$ are inequivalent under this action. Therefore, by Theorem \ref{Thm: RepZ2HeisIsTYR} and the classification of Tambara-Yamagami $\mathrm{W}^*$-tensor categories \cite[Thm. 4.17]{AM2025} (see also Theorem \ref{thm: TYClassificationThm}), there exists exactly one pair $(\chi,\xi)\in \{\chi_{\pm}\}\times \{\pm1\}$ for which the $\mathrm{W}^*$-equivalence (possibly precomposed by the equivalence $\Rep(C_0(\R))\to \Rep(C_0(\R))$ induced by an automorphism of $\R$)
\[
\mathfrak{M}\oplus\mathfrak{M^\sigma}: \Rep(C_0(\R))\oplus \Hilb\cdot \tau\to \Rep^{\Z/2}(\Heis)
\]
admits a tensor structure $s: \mathfrak{M}\oplus\mathfrak{M}^\sigma(-\otimes - )\cong \mathfrak{M}\oplus\mathfrak{M}^\sigma(-)\boxtimes \mathfrak{M}\oplus\mathfrak{M}^\sigma(-)$. The statement of the current theorem is precisely that this pair is $(\chi_-, +1)$. 

Let us denote by $(\chi_0,\xi_0)$ the pair $(\chi,\xi)$ for which $\mathfrak{M}\oplus\mathfrak{M^\sigma}$ admits a tensor structure. By \cite[Thm. 3.31]{GcrossedbraidedRep}, the category $\Rep^{\Z/2}(\Heis)$ admits the structure of a $\Z/2$-crossed balanced $\mathrm{W}^*$-tensor category. In addition, the functor $\mathfrak{M}\oplus\mathfrak{M^\sigma}$ clearly preserves the grading. We can then transport this $\Z/2$-crossed balanced structure along the equivalence $\mathfrak{M}\oplus\mathfrak{M^\sigma}$. By Lemma \ref{lemm: ActionTransportedToC0R}, we can assume that the action of $\sigma$ on any $\Rep(C_0(\R))$ is given by $\Sigma$. Therefore by Propositions \ref{prop: braidingsTY} and \ref{prop: ClassifyTwists}, there exists $\zeta\in\{\pm1\}$ and an equivalence of $\Z/2$-crossed balanced $\mathrm{W}^*$-tensor categories
\begin{equation}\label{eq: EquivalenceZ2ZBalanced}
\Rep^{\Z/2}(\Heis)\cong \TYcal(\R,\chi_0,\xi_0),
\end{equation}
where the crossed braiding on the right-hand side is given by Proposition \ref{prop: braidingsTY} and the crossed balance is given by Proposition \ref{prop: ClassifyTwists}, using $\zeta$. The underlying tensor functor of this equivalence is $(\mathfrak{M}\oplus\mathfrak{M^\sigma}, s)$, and we denote by $\Phi$ the structure data witnessing compatibility with the $\Z/2$-action. We can further assume that $\Phi_1 = \id$, where 1 denotes the unit of $\{1,\sigma\} = \Z/2$. We write $\theta_\tau$ for the $\Z/2$-crossed balance on the object $\tau\in \TYcal(\R, \chi_0, \xi_0)$ defined by Proposition \ref{prop: ClassifyTwists} and the chosen $\zeta$. 

Consider the following diagram
\[\begin{tikzcd}[column sep=large]
{\mathfrak{M}^\sigma(\tau)} && {\mathfrak{M}^\sigma(\tau)} && {\mathfrak{M}^\sigma(\tau)} & \\
&& {T_\sigma(\mathfrak{M}^\sigma(\tau))} && {T_\sigma(\mathfrak{M}^\sigma(\tau))} \\
&&&& {T_\sigma\circ T_\sigma(\mathfrak{M}^\sigma(\tau))} & {T_1(\mathfrak{M}^\sigma(\tau))}.
\arrow["{\mathfrak{M}^\sigma(\theta_\tau)}", from=1-1, to=1-3]
\arrow["{\theta_{\mathfrak{M}^\sigma(\tau)}}"', from=1-1, to=2-3]
\arrow["{\mathfrak{M}^\sigma(\theta_\tau)}", from=1-3, to=1-5]
\arrow["{\theta_{\mathfrak{M}^\sigma(\tau)}}"', from=1-3, to=2-5]
\arrow["{\Phi_\sigma(\tau)}", from=2-3, to=1-3]
\arrow["{T_\sigma(\theta_{\mathfrak{M}^\sigma(\tau)})}"', from=2-3, to=3-5]
\arrow["{\Phi_\sigma(\tau)}", from=2-5, to=1-5]
\arrow["{T_\sigma(\Phi_\sigma(\tau))}"', from=3-5, to=2-5]
\arrow["{=}"', from=3-5, to=3-6]
\arrow["{\Phi_1(\tau)}"', from=3-6, to=1-5]
\end{tikzcd}\]
The two top triangles commute by the fact that the equivalence \eqref{eq: EquivalenceZ2ZBalanced} is compatible with the crossed balances. The central quadrilateral commutes because $T_\sigma$ acts trivially on morphisms and $\Phi$ is a natural family of isomorphisms. Finally, the right-most diagram commutes by the fact that the equivalence \eqref{eq: EquivalenceZ2ZBalanced} is compatible with the $\Z/2$-action. In addition, note that $T_1(\mathfrak{M^\sigma}(\tau)) = \mathfrak{M}^\sigma(\tau)$ and we have taken $\Phi_1(\tau) = \id_{\mathfrak{M}^\sigma(\tau)}$. Hence, it follows that $\mathfrak{M}^\sigma(\theta_\tau^2)$ equals the composition
\[\begin{tikzcd}
{\mathfrak{M}^\sigma(\tau)} && {T_\sigma(\mathfrak{M}^\sigma(\tau))} && {T_\sigma\circ T_\sigma(\mathfrak{M}^\sigma(\tau))} & {\mathfrak{M}^\sigma(\tau)}.
\arrow["{\theta_{\mathfrak{M}^\sigma(\tau)}}", from=1-1, to=1-3]
\arrow["{T_\sigma(\theta_{\mathfrak{M}^\sigma(\tau)})}", from=1-3, to=1-5]
\arrow["{=}", from=1-5, to=1-6]
\end{tikzcd}\]
By Proposition \ref{prop: DoubleTwistComputation}, this composition is exactly $e^{-\frac{\pi i}{4}}\cdot \id_{\mathfrak{M}^\sigma(\tau)}$. Let us now compute $\theta_\tau^2$, which is independent of $\zeta\in\{\pm1\}$, using the explicit formula in Proposition \ref{prop: ClassifyTwists}. Writing $\chi(x,y) = e^{aixy}$ for $a\in\{\pm1\}$, we have 
\[
\theta_\tau^2 = \xi\cdot e^{i\cdot \text{sgn}(a)\cdot \frac{\pi}{4}}\cdot \id_{\tau}.
\]
Therefore, we require $\xi\cdot e^{\text{sgn}(a)\cdot \frac{\pi i}{4}} = e^{-\frac{\pi i}{4}}$, hence $\text{sgn}(a) = -1$ and $\xi = +1$. 
\end{proof}

As an immediate corollary of Theorem \ref{thm: MainThmCrossed}, and using Proposition \ref{prop: ClassifyTwists}, we obtain the following result.

\begin{theorem}\label{thm: CharacterizationRepHeis}
The balanced $\mathrm{W}^*$-tensor category $\Rep(\Heis)$ of representations of the Heisenberg conformal net is equivalent to $\Hilb\,\R$ with braiding given by
\[
\begin{array}{cccc}
\beta_{\upsilon, \nu}: & L^2(\upsilon\times\nu)&\to &L^2(\upsilon\times\nu) \\
& f(t,s)&\mapsto & e^{-ip(t)n(s)}f(t,s) 
\end{array}
\]
and balance
\[
\begin{array}{cccc}
\theta_\upsilon :& L^2(\upsilon)&\to &L^2(\upsilon)  \\
& f(t)& \mapsto & e^{-ip(t)^2}f(t)
\end{array}
\]
for objects $\int^{\oplus}_{t\in T}\delta_{p(t)}\mathrm{d}\upsilon(t)$ and $\int^{\oplus}_{s\in S}\delta_{n(s)}\mathrm{d}\nu(s)$ of $\Rep(C_0(\R))$.
\end{theorem}

We know by \cite[Thm. 3.32]{GcrossedbraidedRep} that the $\mathrm{W}^*$-tensor category $\Rep^{\Z/2}(\Heis)$ is not only $\Z/2$-crossed braided but also $\Z/2$-crossed balanced. However, in Theorem \ref{thm: MainThmCrossed} we are not able to determine the $\Z/2$-crossed balance on $\Rep^{\Z/2}(\Heis)$. The reason is that, with the techniques we use, we are only able to compute the square of the balance on the non-invertible simple object, which by Proposition \ref{prop: ClassifyTwists} is not enough to distinguish between the two possible $\Z/2$-crossed balances on $\Rep^{\Z/2}(\Heis).$

We conjecture which of the two possible $\Z/2$-crossed balanced structures on $\TYcal(\R, \chi_-, +1)$ corresponds to $\Rep^{\Z/2}(\Heis)$. Let $\mathcal{C}$ be a unitary fusion category with a unitary braiding $\beta$. We write $X^\vee$ for the dual of an object $X\in\mathcal{C}$ and $\text{ev}_X$ and $\text{coev}_X$ for its evaluation and coevaluation maps. Given a {ribbon structure} $\theta$ on $\mathcal{C}$ (that is, a balance compatible with the rigid structure of $\mathcal{C}$), one can define the \emph{quantum dimension} of an object $X\in \mathcal{C}$ as the number $\dim_{\theta}(X) = \text{ev}_X\circ \beta_{X, X^\vee}\circ (\theta_X\otimes\id_{X^\vee})\circ \text{coev}_X$. The ribbon structure $\theta$ is called \emph{unitary} if the dimension of every object is a positive number, see \cite[Sec. 3.2]{MR3239112}. Then, there exists a unique unitary ribbon structure on every unitary braided fusion category $\mathcal{C}$, see \cite[Thm. 3.5]{MR3239112}. Analogous results hold in the context of a discrete group $G$ acting on $\mathcal{C}$, where one obtains the concept of a $G$-crossed unitary ribbon structure on $\mathcal{C}$, see \cite[Sec. 3.5]{Galindo_2026}. In particular, if $K$ is a finite group, $\chi:K\times K\to U(1)$ is a symmetric nondegenerate bicharacter and $\xi\in\{\pm1\}$ is a sign, the unique unitary $\Z/2$-crossed ribbon structure on the $\Z/2$-crossed braided unitary fusion category $\TYcal(K, \chi, \xi)$ is the one with $\zeta = +1$ in the notation of Proposition \ref{prop: ClassifyTwists}, see \cite[Prop. 4.6]{Galindo_2026}.

If $\A$ is a rational conformal net, then $\Rep(\A)$ is a unitary modular tensor category \cite{klm01}, hence the ribbon structure is unitary. If $G$ is a finite group acting on $\A$, the same holds for the $G$-crossed ribbon category $\Rep^G(\A)$ of $G$-twisted representations of $\A$, since $(\Rep^G(\A))^G\cong \Rep(\A^G)$ is a unitary ribbon fusion category by \cite{klm01} and \cite{fixedpoints}. 

We expect that, if $\mathcal{C}$ is a braided $\mathrm{W}^*$-tensor category which is furthermore involutive (see \cite{MR2861112} for the definition), there exists an analogous notion of a {unitary balance} on $\mathcal{C}$, and that the canonical balance on $\Rep(\A)$ from \cite{balanceRepA}, with its canonical involutive $\mathrm{W}^*$-tensor structure from \cite[App. A]{GcrossedbraidedRep}, is unitary even if $\A$ is non-rational. We also expect there to be an analogous concept in the $G$-crossed setting, so that $\Rep^G(\A)$ is a unitary balanced involutive $\mathrm{W}^*$-tensor category if $G$ is a discrete group acting on a (possibly non-rational) conformal net $\A$, when using the $G$-crossed balance from \cite{GcrossedbraidedRep} and the involutive structure from \cite[App. A]{GcrossedbraidedRep}. Both of these concepts should restrict to the usual notions of unitary ribbon and unitary crossed ribbon structures in the case of unitary fusion categories. Generalizing the case of finite Tambara-Yamagami fusion categories, we expect the unique unitary balance on a Tambara-Yamagami $\mathrm{W}^*$-tensor category $\TYcal(G, \chi, \xi)$ to be the one given by $\zeta = +1$ in the notation of Proposition \ref{prop: ClassifyTwists}. Assembling these expectations, we conjecture the following.

\begin{conjecture}\label{Conjecture: WhichCrossedBalance}
    The $\Z/2$-crossed balanced $\mathrm{W}^*$-tensor category $\Rep^{\Z/2}(\Heis)$ of $\Z/2$-twisted representations of the Heisenberg conformal net is equivalent to the $\Z/2$-crossed braided $\mathrm{W}^*$-tensor category described in Theorem \ref{thm: MainThmCrossed} with the $\Z/2$-crossed balance
\[
\begin{array}{cccc}
\theta_{\upsilon} :&L^2(\upsilon)&\to &L^2(\upsilon)\\
&f(t)&\mapsto &e^{-ip(t)^2}f(t) \vspace{.5cm}\\
\theta_{\tau}:&\tau&\xrightarrow{e^{-\pi i/8}\cdot\id_\tau} &\tau,
\end{array}
\]
for $\int^{\oplus}_{t\in T}\delta_{p(t)}\mathrm{d}\upsilon(t)\in \Rep(C_0(\R))$.
\end{conjecture}

\section{The braided tensor category of representations of $\Heis^{\Z/2}$}

We finish with the most important result of this paper, which is the explicit characterization of the braided $\mathrm{W}^*$-tensor category $\Rep(\Heis^{\Z/2})$ of representations of the $\Z/2$-fixed points of the Heisenberg conformal net. 

\begin{theorem}\label{thm: RepFixedPointsMainTheorem}
The braided $\mathrm{W}^*$-tensor category $\Rep(\Heis^{\Z/2})$ of representations of the $\Z/2$-fixed points of the Heisenberg conformal net is equivalent to the $\Z/2$-equivariantization
\[
\TYcal(\R, \chi_-, +1)^{\Z/2}
\]
of the following $\mathbb{Z}/2$-crossed braided $\mathrm{W}^*$-tensor category. The underlying $\mathrm{W}^*$-tensor category~is
\(
\TYcal(\R, \chi_-, +1).
\)
The action of $\Z/2 = \{1,\sigma\}$ is given by $\sigma\big(\int^{\oplus}_{t\in T}\delta_{p(t)}\mathrm{d}\upsilon(t)\big) = \int^{\oplus}_{t\in T}\delta_{-p(t)}\mathrm{d}\upsilon(t)$ and $\sigma(\tau) = \tau$, with the only non-trivial structure data $\sigma(\tau\otimes \tau) = \sigma\big(\int^{\oplus}_{x\in \R}\delta_{x}\mathrm{d}\lambda(x)\big) = \int^{\oplus}_{x\in \R}\delta_{-x}\mathrm{d}\lambda(x)\xrightarrow{f(x)\mapsto f(-x)}\int^{\oplus}_{x\in \R}\delta_{x}\mathrm{d}\lambda(x) = \tau\otimes\tau = \sigma(\tau)\otimes\sigma(\tau)$. The $\Z/2$-crossed braiding is given by 
\[
\begin{array}{cccc}
\beta_{\upsilon, \nu} :&L^2(\upsilon\times\nu)&\to &L^2(\upsilon\times\nu)\\
&f(t,s)&\mapsto & e^{-ip(t)n(s)}f(t,s) \vspace{.5cm}\\
\beta_{\upsilon, \tau} = \beta_{\tau, \upsilon}:&L^2(\upsilon)\cdot \tau&\to& L^2(\upsilon)\cdot\tau\\
&f(t)&\mapsto& e^{ip(t)^2/2}f(t)\vspace{.5cm}\\
\beta_{\tau, \tau}: &L^2(\lambda) &\to &L^2(\lambda)\\
& f(x)&\mapsto & e^{\pi i/8 } \,e^{-ix^2/2}f(x),
\end{array}
\]     
for objects $\int^{\oplus}_{t\in T}\delta_{p(t)}\mathrm{d}\upsilon(t)$ and $\int^{\oplus}_{s\in S}\delta_{n(s)}\mathrm{d}\nu(s)$ of $\Rep(C_0(\R))$.
\end{theorem}
\begin{proof}
This result is a direct application of \cite[Thm. 4.19]{fixedpoints} and Theorem \ref{thm: MainThmCrossed}.
\end{proof}

Using \cite[Thm. 4.19]{fixedpoints}, Conjecture \ref{Conjecture: WhichCrossedBalance} is equivalent to the following conjecture.

\begin{conjecture}
     The balanced $\mathrm{W}^*$-tensor category $\Rep(\Heis^{\Z/2})$ of representations of the $\Z/2$-fixed points of the Heisenberg conformal net is equivalent to the $\Z/2$-equivariantization of the $\Z/2$-crossed braided $\mathrm{W}^*$-tensor category described in Theorem \ref{thm: MainThmCrossed} with the $\Z/2$-crossed balance
\[
\begin{array}{cccc}
\theta_{\upsilon} :&L^2(\upsilon)&\to &L^2(\upsilon)\\
&f(t)&\mapsto &e^{-ip(t)^2}f(t) \vspace{.5cm}\\
\theta_{\tau}:&\tau&\xrightarrow{e^{-\pi i/8}\cdot\id_\tau} &\tau,
\end{array}
\]
for $\int^{\oplus}_{t\in T}\delta_{p(t)}\mathrm{d}\upsilon(t)\in \Rep(C_0(\R))$.
\end{conjecture}
\bibliographystyle{halpha-abbrv}
\bibliography{RepHeisBiblio}
\end{document}

%% file: Commands.tex
\newcommand{\Z}{\mathbb{Z}}
\newcommand{\R}{\mathbb{R}}
\newcommand{\C}{\mathbb{C}}

\newcommand{\Diff}{\text{Diff}}

\newcommand{\Mob}{\mathbf{M\ddot{o}b}}
\newcommand{\A}{\mathcal{A}}
\newcommand{\Rot}{\text{Rot}}

\newcommand{\id}{\text{id}}
\newcommand{\INT}{\text{INT}}

\newcommand{\Hom}{\text{Hom}}
\newcommand{\Aut}{\text{Aut}}
\newcommand{\Rep}{\text{Rep}}
\newcommand{\colim}{\text{colim}}
\newcommand{\pe}{{\text{pe}}}
\newcommand{\Hilb}{\text{Hilb}}
\newcommand{\MorC}{\text{C}^*\text{Alg}}
\newcommand{\Tcal}{\mathcal{T}}
\newcommand{\TYcal}{\mathcal{TY}}
\newcommand{\Ccal}{\mathcal{C}}
\newcommand{\End}{\text{End}}
\newcommand{\Heis}{\text{Heis}}
\newcommand{\Jcal}{\mathcal{J}}

\newcommand{\Conf}{\text{Conf}}

\newcommand{\Vect}{\text{Vect}}
\newcommand{\Loc}{\text{Loc}}

\newcommand{\Dcal}{\mathcal{D}}

\newcommand{\fixedxrightarrow}[2][3cm]{%
  \xrightarrow{\makebox[#1]{\ensuremath{#2}}}%
}
\newcommand{\1}{\mathbf{1}}

%% file: Preamble.tex
\usepackage{amssymb}
\usepackage{amsfonts}
\usepackage{amsmath}
\usepackage{quiver}
\usepackage{amsthm}
\usepackage{comment}
\usepackage[english]{babel}
\usepackage{nicefrac}
\usepackage[a4paper]{geometry}
\usepackage{amssymb}
\usepackage{mathtools}
\usepackage[hidelinks]{hyperref}
\usepackage{mathrsfs}

\usepackage[dvipsnames]{xcolor}
\usetikzlibrary{arrows}

\geometry{top=2.40cm, bottom=2.40cm, left=2.40cm, right=3cm}

\theoremstyle{definition}
\newtheorem{definition}{Definition}[section]
\newtheorem{proposition}[definition]{Proposition}
\newtheorem{lemma}[definition]{Lemma}
\newtheorem{theorem}[definition]{Theorem}
\newtheorem{corollary}[definition]{Corollary}
\newtheorem{example}[definition]{Example}
\newtheorem{remark}[definition]{Remark}
\newtheorem{conjecture}[definition]{Conjecture}
\newtheorem*{theorem*}{Theorem}
\newtheorem*{definition*}{Definition}
\newtheorem*{example*}{Example}
\newtheorem*{corollary*}{Corollary}
\newtheorem*{proposition*}{Proposition}
\newtheorem{maintheorem}{Theorem}

\newtheorem{mainconjecture}{Conjecture}

\newtheorem*{rep@theorem}{\rep@title}
\newcommand{\newreptheorem}[2]{%
\newenvironment{rep#1}[1]{%
 \def\rep@title{#2 \ref{##1}}%
 \begin{rep@theorem}}%
 {\end{rep@theorem}}}
\makeatother

\newreptheorem{theorem}{Theorem}

\definecolor{myviolet}{HTML}{7D3C98}
\definecolor{myorange}{HTML}{F39C12}
\definecolor{myblue}{HTML}{2E86C1}
\definecolor{mygreen}{HTML}{1E8449}
\definecolor{myred}{HTML}{C0392B}



\usetikzlibrary{decorations.markings}